\documentclass[11pt]{aims}
\usepackage{amsmath, amssymb, mathrsfs}
  \usepackage{paralist}
  \usepackage{graphics} %% add this and next lines if pictures should be in esp format
  \usepackage{graphicx}
  \usepackage{epsfig} %For pictures: screened artwork should be set up with an 85 or 100 line screen
\usepackage{graphicx}  \usepackage{epstopdf}%This is to transfer .eps figure to .pdf figure; please compile your paper using PDFLeTex or PDFTeXify.
 \usepackage[toc,page]{appendix}
\usepackage{chngcntr}
 \usepackage[foot]{amsaddr}

\usepackage{titlesec}
\titleformat{\section}{\Large\bfseries}{\thesection.}{4pt}{}
\titleformat{\subsection}{\large\bfseries}{\thesection.\arabic{subsection}.}{4pt}{}
\titleformat{\subsubsection}{\bfseries}{\thesection.\arabic{subsection}.\arabic{subsubsection}.}{4pt}{}
\titleformat*{\paragraph}{\bfseries}
\titleformat*{\subparagraph}{\bfseries}
\usepackage[margin=1.2in]{geometry}

  \def\currentyear{}
   \def\currentmonth{}
    \def\ppages{}
     \def\DOI{}

\newtheorem{theorem}{Theorem}[section]
\newtheorem{corollary}[theorem]{Corollary}

\newtheorem{lemma}[theorem]{Lemma}
\newtheorem{proposition}[theorem]{Proposition}
\theoremstyle{definition}
\newtheorem{definition}[theorem]{Definition}
\newtheorem{remark}[theorem]{Remark}

\newtheorem{classification}[theorem]{Classification}

\newcommand{\ep}{\varepsilon}

\newcommand{\Rb}{\mathbb{R}}

\newcommand{\Hc}{\mathcal{H}}

\newcommand{\Gc}{\mathcal{G}}

\newcommand{\be}{\begin{equation}}
\newcommand{\ee}{\end{equation}}
\newcommand{\pa}{\partial}

\newcommand{\indic}{1\!\!1}
\newcommand{\cbf}{\mathbf c}

\newcommand{\eps}{\epsilon}

\newcommand{\as}{\underline{a}}

\renewcommand{\theenumi}{\roman{enumi}}

\numberwithin{equation}{section}

\author[C.~Collot]{Charles Collot}
\address{CNRS and AGM (UMR 8088) laboratory of CY Cergy Paris Universit\'e, 2 rue Adolphe Chauvin, 95300 Pontoise, France}
\email{ccollot@cyu.fr}
\author[T.~Duyckaerts]{Thomas Duyckaerts}
\address{LAGA (UMR 7539), Universit\'e Sorbonne Paris Nord, Institut Galil\'ee, 99 avenue Jean-Baptiste Cl\'ement, 93430 Villetaneuse, France}
\email{duyckaer@math.univ-paris13.fr}
\author[C.~Kenig]{Carlos Kenig}
\address{University of Chicago, Department of Mathematics, 5734 University Avenue, Chicago, IL 60637-1514, USA}
\email{ckenig@uchicago.edu}
\author[F.~Merle]{Frank Merle}
\address{Institut des Hautes \'Etudes Scientifiques, and AGM (UMR 8088) laboratory of CY Cergy Paris Universit\'e 2 rue Adolphe Chauvin, 95300 Pontoise, France}
\email{frank.merle@cyu.fr}

\title[Non-radiative solutions for energy-critical wave equations] %use the shortened version of the full title
{On classification of non-radiative solutions for various energy-critical wave equations}
\thanks{ 
\today}
\begin{document}

\begin{abstract}

Non-radiative solutions of energy critical wave equations are such that their energy in an exterior region $|x|>R+|t|$ vanishes asymptotically in both time directions. This notion, introduced by \cite{DuKeMe11a}, has been key in solving the soliton resolution conjecture for these equations in the radial case. In the present paper, we first classify their asymptotic behaviour at infinity, showing that they correspond to a $k$-parameters family of solutions where $k$ depends on the dimension. This generalises the previous results \cite{DuKeMe13,DuKeMaMe22} in three and four dimensions. We then establish a unique maximal extension of these solutions.

\end{abstract}

\maketitle

\section{Introduction and main results}

\subsection{Non-radiative solutions, channels of energy, and soliton resolution}

This paper is about non-radiative finite energy solutions of radial energy critical wave type equations. Prototypes of such equations are the energy critical semilinear wave equation in dimensions $N\geq 3$, the $k$-equivariant wave maps equation in $N=2$ dimensions, and the radial $4+1$ dimensional Yang Mills equation:
\begin{align}
\label{id:nonlinearwave}  \pa_t^2 u-\Delta u&=|u|^{\frac{4}{N-2}}u,\\
\label{id:wavemaps} \pa_t^2 u-\Delta u&=-\frac{k^2}{2r^2}\sin 2u,\\
\label{id:YangMills}
\partial_t^2u-\partial_r^2u-\frac{1}{r}\partial_ru&=2u(1-u^2). 
\end{align}
where we introduced $|x|=r$. A solution is said to be non-radiative for $|x|>R+|t|$ if
\begin{equation} \label{eq:nonradiativeintro}
\sum_{\pm } \lim_{t\to \pm \infty}  \int_{|x|>R+|t|} |\nabla_{t,x}u(t,x)|^2dx=0
\end{equation}
where we write $\nabla_{t,x}u=(\pa_t u,\pa_{x_1}u,...,\pa_{x_d}u)$. To track down this property, a stronger notion of energy channels was introduced by the last three authors in \cite{DuKeMe11a} and \cite{DuKeMe13}, which takes the form of an estimate
\begin{equation} \label{eq:channelsintro}
\sum_{\pm } \lim_{t\to \pm \infty}  \int_{|x|>R+|t|} |\nabla_{t,x}u(t,x)|^2dx\geq c_0 \sum_{\pm }  \int_{|x|>R} |\nabla_{t,x}u(0,x)|^2dx, \qquad c_0>0,
\end{equation}
for both linear and nonlinear problems. In odd dimensions $N\geq 3$, for the linear wave equation
\begin{equation} \label{eq:freewaveintro}
\left\{ 
\begin{array}{l l} \pa_t^2 u-\Delta u=0,\\
\vec u(0)=(u_0,u_1),
\end{array}
\right.
\end{equation}
the channels of energy estimate \eqref{eq:channelsintro} holds true on the orthogonal of a finite-dimensional subspace of non-radiative solutions. Its dimension grows with the spatial one $N$, and a clear difference appears between high dimensions, and low ones where rigidity occurs. Namely, for $N=3$ \cite{DuKeMe13} proves that the only non-radiative solutions of \eqref{id:nonlinearwave} are the solitons, which allows them to show the soliton resolution for \eqref{id:nonlinearwave}. For higher dimensions $N\geq 5$ other non-radiative solutions exist as is proved in the present article. Nevertheless \cite{DuKeMe19Pb} could still obtain the soliton resolution, reducing its proof to showing that a collision of solitons made up part of the energy of the solution channel. This latter property was proved by contradiction, deducing from the channels of energy estimate \eqref{eq:channelsintro} an explicit system of ODEs for the scaling parameters of the solitons as well as an estimate for the largest scale, which were shown to be incompatible.

 Even dimensions $N\geq 4$ are more degenerate, since the estimate \eqref{eq:channelsintro} for solutions to \eqref{eq:freewaveintro} only holds for one component of the initial data $(u_0,u_1)$ and fails for the other. Nevertheless, in $N=4$ dimensions \cite{DuKeMaMe22} could still obtain the soliton resolution for \eqref{id:nonlinearwave} and the co-rotational ($k=1$) wave maps equation \eqref{id:wavemaps}, using as a key intermediate result that the only non-radiative solutions were solitons, in a strategy similar to \cite{DuKeMe13}. For higher dimensions $N\geq 6$ other non-radiative solutions exist, which is proved in the present article. The strategy of \cite{DuKeMe19Pb} could still be adapted to the $N=6$ dimensional case by \cite{CoDuKeMe22}, by introducing new weakened channels of energy estimates in the vicinity of a multisoliton like \eqref{eq:channelsintro} but with a logarithmic loss \cite{CoDuKeMe22}, and by classifying the behaviour of non-radiative solutions at infinity, which is the purpose of the present article. We believe this analysis can be extended to $k\geq 2$-equivariant wave maps \eqref{id:wavemaps} as well as the radial $4+1$ dimensional Yang-Mills system since both the weakened channels of energy estimates of \cite{CoDuKeMe22} and the classification of non-radiative solutions are valid. 
 
 After the co-rotational result ($k=1$ in \eqref{id:wavemaps}), due to \cite{DuKeMaMe22}, appeared on arXiv, the soliton resolution for general $k$ was later  obtained by Jendrej-Lawrie \cite{JendrejLawrie21P}. The general strategy of \cite{JendrejLawrie21P} follows the one introduced in \cite{DuKeMe19Pb}, of proving the inelastic collision of solitons. As shown in \cite{DuKeMe19Pb}, this strategy gives the passage from a sequential decomposition to a continuous in time one. In \cite{JendrejLawrie21P}, the mechanism to prove the inelastic collision of solitons is not through a rigidity theorem, as in \cite{DuKeMe13}, \cite{DuKeMe19Pb}, \cite{DuKeMaMe22} and \cite{CoDuKeMe22Pv1}, but through the use of the modulation equations (introduced by these authors in their work \cite{JendrejLawrie18} on two-bubble dynamics of threshold solutions) combined with a delicate ``no-return analysis'' in the neighborhood of a multisoliton. This was inspired by earlier work, in the neighborhood of a single soliton, due to Duyckaerts-Merle \cite{DuMe08}, Nakanishi-Schlag (see \cite{NaSc11Bo} and references therein), and Krieger-Nakanishi-Schlag (see \cite{KrNaSc15} and references therein).

 We believe the extension of our result here to the semilinear wave equation \eqref{id:nonlinearwave} in dimensions $N\geq 8$ to be more technical, but amenable by combining the analysis of \cite{CoDuKeMe22}, \cite{CoDuKeMe22Pv1},\cite{CoDuKeMe22Pc}, and of the present paper.

After \cite{CoDuKeMe22Pv1} was posted on arXiv in January 2022, Jendrej-Lawrie \cite{JendrejLawrie22P} posted a new preprint on arXiv, in which they extend the full resolution for \eqref{id:nonlinearwave} proved in odd dimension in \cite{DuKeMe13}, \cite{DuKeMe19Pb}, in dimension $N=4$ in \cite{DuKeMaMe22}, and in dimension $N=6$ in \cite{CoDuKeMe22Pv1} to all $N\geq 4$ (all these results are in the radial setting). The approach in \cite{JendrejLawrie22P} is to prove the inelastic collision of solitons (shown in \cite{DuKeMe13}, \cite{DuKeMe19Pb}, \cite{DuKeMaMe22} and \cite{CoDuKeMe22Pv1}) by the ``no return method'' as in \cite{JendrejLawrie21P}. The results in \cite{JendrejLawrie22P} have no overlap with the ones in this paper.
 
\subsection{Classification of non-radiative solutions}

Non-radiative solutions of the linear wave equation \eqref{eq:freewaveintro} are classified. There exist explicit polynomials $(p_m)_{0\leq m< \frac{N}{2}-1}$ of degree $m$, with $p_m$ even or odd if $m$ is even or odd, such that for each $0\leq m<\frac{N}{2}-1$, the function
$$
\phi_m(t,x)= \frac{1}{|x|^{N-2-m}}p_{m}(\frac{t}{|x|})
$$
satisfies $\pa_t^2\phi_m-\Delta \phi_m=0 $ for $|x|\neq 0$, and is non-radiative for any $R\in \mathbb R$ in the sense that it satisfies \eqref{eq:nonradiativeintro} (the integral being well-defined if $|t|>-R$). These properties are also satisfied by linear combinations of such functions, that is, for any $\cbf =(c_{m})_{0\leq m<\frac{N}{2}-1}\in \mathbb R^{\lfloor \frac{N-1}{2} \rfloor}$, for the function $a_F=a_F[\cbf]$ defined by
\begin{equation}
\label{def_aF}
a_F(t,x)=\sum_{0\leq m< \frac{N}{2}-1}c_m\phi_m(t,x).
\end{equation}

A proof is recalled in Lemma \ref{lem:freenonrad}. These functions describe all finite energy non-radiative solutions:

\begin{classification}[Classification of non-radiative linear waves, see \cite{KeLaLiSc15}, \cite{DuKeMaMe22}, \cite{LiShenWei21P}] \label{pr:nonradiativefree}

Let $N\geq 2$ and $R>0$. A radial solution $u$ to \eqref{eq:freewaveintro} is non-radiative for $|x|>R+|t|$ if and only if there exists $\cbf\in \mathbb R^{\lfloor \frac{N-1}{2}\rfloor}$ such that $u(t,x)=a_F[\cbf](t,x)$ for all $|x|>R+|t|$.

\end{classification}

Note that for $N=2$, the conclusion of Classification \ref{pr:nonradiativefree} has to be interpreted in the sense that $u(t,x)=0$ for all $|x|>R+|t|$. In \cite{KeLaLiSc15}, \cite{DuKeMaMe22} and \cite{LiShenWei21P} a much stronger statement is showed: that on the orthogonal of these non-radiative solutions, a channels of energy estimate \eqref{eq:channelsintro} holds true (only for half of the data in even dimensions, see also \cite{CoKeSc14}).

Much of the difficulty in even dimensions is linked to the existence of a resonance
\begin{equation} \label{id:defphiN/2-1}
\phi_{\frac N2 -1}(t,x)= \frac{1}{|x|^{\frac{N}{2}-1}}p_{\frac N2 -1}(\frac{t}{|x|}),
\end{equation}
for $p_{\frac N2 -1}$ a polynomial of degree $\frac{N}{2}-1$ that has the same parity as $\frac{N}{2}-1$. It solves $\pa_t^2\phi_{\frac N2 -1}-\Delta \phi_{\frac N2 -1}=0 $ for $|x|\neq 0$ but critically fails to belong to the energy space, see Section \ref{sec:resonance}.

In this article, we extend Classification \ref{pr:nonradiativefree} to non-linear wave equations in $N\geq 3$ dimensions,
\begin{equation} \label{eq:nonlinearwaveintro}
\left\{ 
\begin{array}{l l} \pa_t^2 u-\Delta u=\varphi(|x|,u),\\
\vec u(0)=(u_0,u_1),
\end{array}
\right.
\end{equation}
with analytic energy critical nonlinearities
\begin{equation} \label{eq:nonlinearityanalytic}
\varphi (|x|,u)=\sum_{k\geq 2}^\infty \varphi_k |x|^{(k-1)(\frac N2-1)-2}u^k,
\quad \forall \tau>0,\quad \lim_{k\to\infty} \tau^k\varphi_k=0,
\end{equation}
or energy critical power nonlinearities in $N=4,6$ or $N$ odd dimensions:
\begin{equation} \label{eq:nonlinearitypower}
\varphi (|x|,u)=|u|^{\frac{4}{N-2}}u.
\end{equation}
Nonlinearities of the form \eqref{eq:nonlinearityanalytic} and \eqref{eq:nonlinearitypower} include the aforementioned equations \eqref{eq:nonlinearitypower}, \eqref{id:wavemaps} and \eqref{id:YangMills} after a suitable dimensional reduction for the last two. We expect that the techniques of the present paper and those of \cite{CoDuKeMe22} can treat nonlinearities \eqref{eq:nonlinearitypower} in higher even dimensions.

Our first result concerns solutions with small energy in the exterior cone $|x|>R+|t|$. We introduce for $\cbf\in \mathbb R^{\lfloor \frac{N-1}{2}\rfloor}$ and $R>0$ the norm $|\cbf|_R=\big|(R^{-\frac N2+1+m}c_m)_{0\leq m <\frac{N}{2}-1}\big|$ where $|\cdot |$ denotes the usual Euclidean norm. This is the energy of non-radiative free waves for $|x|>R+|t|$ as:
\begin{equation}
\label{bd:energyaFiscbfR}
\delta |\cbf|_R^2 \leq \sup_{t\in \mathbb R} \ \int_{|x|> R+|t|} |\nabla_{t,x}a_F[\cbf](t,x)|^2dx\leq \frac{1}{\delta} |\cbf|_R^2
\end{equation}
where $\delta >0$ is independent of $R$. We remark that $a_F$ enjoys the additional decay:
\begin{equation}
\label{add_decay_aF}
\forall \tilde R>R+|t|>0, \qquad  \int_{|x|> \tilde R} |\nabla_{t,x}a_F[\cbf](t,x)|^2dx \lesssim \left(\frac{R}{\tilde R}\right)^{2\kappa_N} |\cbf|_R^2
\end{equation} 
with $\kappa_N=\frac 12$ if $N$ is odd, and $\kappa_N=1$ if $N$ is even. For $R>0$, we introduce for functions defined for $|x|>R$:
$$
\| u\|_{L^2_R}^2= \int_{|x|\geq R} |u|^2 dx, \qquad \| u\|_{\dot H^1_R}=\| \nabla u \|_{L^2_R}
$$
and denote by $L^2_R$ and $\dot H^1_R$ their associated Hilbert spaces (where in addition we require that $\lim_{|x|\to \infty}u(x)=0$ for $u\in \dot H^1_R$). We define
\begin{align}
 \label{id:defPidotH1} \Pi_{\dot H^1,R}=\Pi_{\dot H^1_R}\left(\mbox{Span}\left(\left(\frac{1}{|x|^{N-2-2l}}\right)_{0\leq l \leq \lfloor \frac{N-3}{4} \rfloor} \right) \right) , \qquad \Pi_{\dot H^1,R}^\perp=1-\Pi_{\dot H^1,R}\\
 \label{id:defPiL2} \Pi_{L^2,R}=\Pi_{L^2_R}\left(\mbox{Span}\left(\left(\frac{1}{|x|^{N-2-2l}}\right)_{0\leq l \leq \lfloor \frac{N-5}{4}\rfloor} \right) \right), \qquad  \Pi_{L^2,R}^\perp=1- \Pi_{L^2,R}
\end{align}
where $\Pi_{H}(E)$ is the orthogonal projection in the Hilbert space $H$ onto the closed set $E$, and
\begin{equation}\label{id:defPimathcalH}
\Pi_{\mathcal H,R}(u_0,u_1)=(\Pi_{\dot H^1,R}u_0,\Pi_{L^2,R}u_1), \qquad \Pi_{\mathcal H,R}^\perp(u_0,u_1)=(\Pi^\perp_{\dot H^1,R}u_0,\Pi^\perp_{L^2,R}u_1).
\end{equation}
We remark by Classification \ref{pr:nonradiativefree} above that such spaces characterise the initial data of free non-radiative waves as a solution $u$ to \eqref{eq:freewaveintro} with $(u_0,u_1)\in \mathcal H$ satisfies $u(t,x)=a(t,x)$ for all $|x|>R+|t|$ if and only if $(u_0,u_1)=\Pi_{\mathcal H,R}(u_0,u_1)$.

 Our main results are the following:

\begin{theorem}[Classification of small non-radiative nonlinear waves] \label{th:main}

Assume either $N\geq 4$ and $\varphi$ of the form \eqref{eq:nonlinearityanalytic}, or $N\geq 3$ is odd, or $N\in\{4,6\}$, and $\varphi(u)=|u|^{\frac{4}{N-2}}u$, and let $\delta=1$ or $\delta=\frac{4}{N-2}$ in the former or latter cases respectively. Then there exist $\epsilon,\epsilon'>0$ such that the following holds true for any $R>0$. 

\begin{enumerate}
\item \emph{Existence}. Assume that $\cbf \in \mathbb R^{\lfloor \frac{N-1}{2} \rfloor}$ satisfies $|\cbf|_R\leq \epsilon$. Then there exists a unique solution $a=a[\cbf]$ to \eqref{eq:nonlinearwaveintro} that is non-radiative for $|x|>R+|t|$ such that
$$
\Pi_{\mathcal H_R} \vec a (0)=\vec a_F[\cbf](0) \qquad \mbox{and} \qquad \| \Pi^\perp_{\mathcal H_R} \vec a (0)\|_{\mathcal H_R}\lesssim |\cbf|^{1+\delta}_R,
$$
and
\begin{equation}
\label{bound_a_aF}
\forall \tilde R>R+|t|, \qquad \int_{|x|> \tilde R} |\nabla_{t,x}a(t)-\nabla_{t,x} a_F(t)|^2dx\lesssim \left(\frac{R}{\tilde R}\right)^{2\kappa_N} |\cbf|_R^{2(1+\delta)}.
\end{equation}
\item \emph{Uniqueness}. Conversely, if $u$ is any solution to \eqref{eq:nonlinearwaveintro} that is non-radiative for $|x|>R+|t|$ with $\int_{|x|> R} |\nabla_{t,x}u(0)|^2dx\leq \epsilon^{'2}$, then there exists $\cbf \in \mathbb R^{\lfloor \frac{N-1}{2}\rfloor}$ with $|\cbf|_R\leq \epsilon$ such that for $a[\cbf]$ described in i. above:
$$
\forall |x|>R+|t|, \qquad u(t,x)=a[\cbf](t,x).
$$
\end{enumerate}
\end{theorem}

We refer to the Subsection \ref{sub:notations} for the definition of the space $S(\{|x|>R+|t|\})$.

We next have the extension of non-radiative solutions to a maximal region $\{|x|>R_*+|t|\}$, which allows us to pass from a small data result to a large data result, with a blow-up alternative when $R_*>0$. 

\begin{theorem}[Maximal non-radiative solutions]
\label{th:maximal}
Assume $N\geq 4$ and $\varphi$ of the form \eqref{eq:nonlinearityanalytic}, or $N\geq 3$ is odd, or $N\in\{4,6\}$, and $\varphi(u)=|u|^{\frac{4}{N-2}}u$. Let $R_0> 0$, and $u$ be a radial non-radiative solution of \eqref{eq:nonlinearwaveintro} for $|x|>R_0+|t|$. Then there exists $R_*\in [0,R_0]$, and an extension $u_*$ of $u$ to $\{|x|>R_*+|t|\}$, such that 
\begin{enumerate}
 \item for all $R>R_*$, $\vec{u}_*\in C^0(\Rb,\Hc_{R+|t|})$, $u_*\in S(\{|x|>R+|t|\})$ and $u_*$ is non-radiative for $\{|x|>R+|t|\}$,
 \item if $R_*>0$, then $(u_*(0),\partial_tu_*(0))\notin \Hc_{R_*}$, or $u^*\notin S(\{|x|>R_*+|t|\})$.
\end{enumerate}
\end{theorem}
\begin{remark}
 The restriction on the dimension $N$ in Theorem \ref{th:maximal} in the case where $\varphi(u)=|u|^{\frac{4}{N-2}}u$ is technical. On one hand, the proof of Theorem \ref{th:maximal} given in Section \ref{sec:maximal} yields the conclusion of Theorem \ref{th:maximal} with the additional assumption 
 \begin{equation}
 \label{extra_as}
 |u(t,r)|\lesssim |x|^{\frac{1-N}2},\quad |x|>R_0+|t|,
 \end{equation} 
 for all $N\geq 3$. On the other hand, it is possible to prove the same result, for general $N\geq 6$, without the assumption \eqref{extra_as}, using the well-posedness theory developped in \cite{BuCzLiPaZh13}.
\end{remark}

\paragraph*{Comments.}

\emph{Previous results in low dimensions}. We recover the results of \cite{DuKeMe13} and \cite{DuKeMaMe22}, concerning the energy critical semilinear wave equation in dimensions $3$ and $4$ and the critical $k=1$-equivariant wave maps.

\emph{Extension to supercritical nonlinearities}. Our results and proofs adapt straightforwardly to energy supercritical power-like analytic non-linearities $u^p$ (where $p>\frac{N+2}{N-2}$ is an odd integer) with spherically symmetric data.

\emph{Rigidity in dimensions $3$ and $4$}. We recover the results of \cite{DuKeMe13} and \cite{DuKeMaMe22} that the only non-radiative solutions are stationary states. 
Note that such a classification result is used to prove a strong version of the soliton resolution, namely that any radial solution $u$ of \eqref{id:nonlinearwave} with $N=3$ or $N=4$ such that 
$$\liminf_{t\to T_+}\left\|(u(t),\partial_tu(t))\right\|_{\dot{H}^1\times L^2}<\infty$$
is indeed bounded on $[0,T_+)$ and satisfies the soliton resolution.

 \paragraph*{Applications and novelties.} Theorem \ref{th:main} is one of the key ingredients in the proof of the soliton resolution in $N=6$ dimensions \cite{CoDuKeMe22Pc}, \cite{CoDuKeMe22Pv1}, where the only non-radiative solutions (in the stronger sense that \eqref{eq:nonradiativeintro} holds for any $R\in \mathbb R$) are proved to be the solitons.
 
The first main novelty of the present article lies in the introduction of a fixed point scheme between radiation fields and initial data. This is enough for the proofs of Theorems \ref{th:main} and \ref{th:maximal} in the non-degenerate case of odd dimensions. The main difficulty resides in the treatment of degenerate directions in even dimensions. First, we show a regularity gain of a derivative for $u_0$ and $u_1$. Then, a resonant direction needs to be ruled out. We do so by approximating with solutions of the form $r^{-d/2+1}\phi (t/r)$, which introduces an elliptic operator with a resonance. Careful computations show that depending on the nonlinearity, quadratic effects either do not see the resonance, or lead to an energy increase between nearby dyadic regions. In both cases, an excitation of the resonance is prohibited since it is shown to lead to infinite energy data. We expect such solutions to be relevant in other contexts. This is the main part of the proof and the main innovation; the method is carried out for various nonlinear terms, thus showing the robustness of the use of such solutions.

The extension Theorem \ref{th:maximal} for non-radiative solutions, which allows us to pass from the small data case to the large data case, is the second main novelty of the paper.

\section*{Acknowledgements}

This work was supported by the National Science Foundation [DMS-2153794 to C.K.]; the CY Initiative of Excellence Grant
"Investissements d'Avenir" [ANR-16-IDEX-0008 to C.C. and F.M.]; and the France and Chicago Collaborating in the Sciences [FACCTS award \#2-91336 to C.C. and F.M.]

\section{Notation and preliminaries}
\subsection{Notation}
\label{sub:notations}
We introduce:
$$
m_0=\lfloor \frac{N-3}{2} \rfloor , \qquad l_0=\lfloor \frac{N-3}{4} \rfloor , \qquad l_1=\lfloor \frac{N-5}{4}\rfloor ,
$$
where $\lfloor \cdot \rfloor$ denotes the floor function, and recall that $l_0+1$, $l_1+1$ and $m_0+1$ correspond to the ranks of $\Pi_{\dot H^1,R}$, $\Pi_{L^2, R}$ and $\Pi_{\mathcal H,R}$ given by \eqref{id:defPidotH1}, \eqref{id:defPiL2} and \eqref{id:defPimathcalH} respectively. Note that $l_0+l_1+2=m_0+1$.

We denote
$$
r=|x|, \qquad \sigma =\frac{t}{|x|}.
$$
We define the energy and $L^2$ scaling transformations by:
$$
f_{(\lambda)}(t,x)=\frac{1}{\lambda^{N/2-1}}f(\frac{t}{\lambda},\frac{x}{\lambda}), \qquad g_{[\lambda]}(x)=\frac{1}{\lambda^{N/2}}g(\frac{x}{\lambda}).
$$
We define for $R>0$ and $\kappa\in \mathbb R$ the following semi-norms:
\begin{align}
\label{id:defWkappa} & \| u\|_{W_R^\kappa}= \sup_{t\in \mathbb R}\sup_{\tilde R>R+|t|} \tilde R^{\kappa}\| (u (t),\pa_t u(t))\|_{\mathcal H_{\tilde R}},\\
\label{id:deftildeWkappa} &\| u\|_{\tilde W^\kappa_R} =\left\{ 
\begin{array}{l l}
\sup_{t\in \mathbb R} \sup_{\tilde R>R+|t|} \tilde R^{\kappa} \| u(t)\|_{\dot H^1_{\tilde R}}  \qquad \mbox{if }N\equiv 4\mod 4,\\
\sup_{t\in \mathbb R} \sup_{\tilde R>R+|t|} \tilde R^{\kappa} \| \pa_t u(t)\|_{L^2_{\tilde R}}  \qquad \mbox{if }N\equiv 6\mod 4,
\end{array} \right. \\
\label{id:defW'kappa}&\| f\|_{W^{'\kappa}_R}=\sup_{\tilde t \in \mathbb R} \sup_{\tilde R>R+|\tilde t|} \tilde R^{\kappa} \|\indic(|x|>\tilde R+|t-\tilde t|) f \|_{L^1L^2},
\end{align}
and denote by $W_R^\kappa$, $\tilde W_R^\kappa$ and $W^{'\kappa}_R$ their associated spaces. 

We denote the even and odd in time parts of a function $f$ as
\begin{equation}\label{notation:id:fpm}
f_+(t) =\frac{1}{2}(g(t)+ g(-t)) \quad \mbox{and} \quad f_-(t) =\frac{1}{2}(g(t)- g(-t)).
\end{equation}
We shall write $a\lesssim b$ if there exists a constant $C$ that is independent of the parameters of the problem at hand such that $a\leq Cb$. We shall write $a\approx b$ if $a\lesssim b$ and $b\lesssim a$.

If $N\geq 4$ and the nonlinearity $\varphi$  is of the form \eqref{eq:nonlinearityanalytic}, we let
\begin{equation}
\label{defS1}   
\|u\|_{S(\Rb\times \Rb^N)}=
\left\|r^{\frac{N-6}{4}} u\right\|_{L^2(\Rb,L^4(\Rb^N))}.
\end{equation} 
If $N\geq 3$ and the nonlinearity is of the form \eqref{eq:nonlinearitypower} we let
\begin{equation}
\label{defS2}
S\left(\Rb\times \Rb^N\right)=L^{\frac{2(N+1)}{N-2}}\left(\Rb\times \Rb^N\right).
\end{equation}
We denote by $S(\{r>R+|t|\})$ the space of radial functions $u$ defined on $\{r>R+|t|\}$ such that the extension of $u$ by $0$ to $\Rb\times \Rb^N$ is in $S(\Rb\times \Rb^N)$.

If $R>0$ and $V$ is the space of functions on $\Rb^N$, we will denote by $V_R$ the spaces of restriction to $\{|x|>R\}$ of radial functions that are in $V$. We will specially use the notations $\Hc_R$, $L^2_R$, $\dot{H}_R^1$. If $I$ is an interval, the notation 
$u\in C^0(I,V_{R+|t|})$ means that $u$ is the restriction to $\{|x|>R+|t|\}$ of a function which is in $C^0(I,V)$. 

\subsection{Weighted Strichartz estimates}
\label{sub:Strichartz}

To treat the case of analytic nonlinearities, we will use, similarly to \cite{JendrejLawrie18}, weighted Strichartz estimates:

\begin{lemma}
\label{lem:weighted_Strichartz}
Assume $N\geq 4$.
 Let $u$ be a radial solution of $\partial_t^2u-\Delta u=F$ on $\Rb\times \Rb^N$, with $\vec{u}(0)\in \Hc$ and $F\in L^1(\Rb,L^2(\Rb^N))$. Then
 \begin{equation}
  \label{Weight_Strichartz}
  \left\|r^{\frac{N-6}{4}} u\right\|_{L^2(\Rb,L^4(\Rb^N))}\lesssim \|(u_0,u_1)\|_{\Hc}+\left\|F\right\|_{L^1L^2}.
 \end{equation} 
 Furthermore, if $R\geq 0$ 
\begin{equation}
  \label{Weight_StrichartzR}
  \left\|r^{\frac{N-6}{4}} u \indic(|x|>R+|t|)\right\|_{L^2L^4}
  \lesssim \big\|(u_0,u_1)\big\|_{\Hc_R}+\big\|F\indic(|x|>R+|t|)\big\|_{L^1L^2}.
 \end{equation}
 \end{lemma}
\begin{proof}
Using finite speed of propagation, it is sufficient to prove \eqref{Weight_Strichartz}. Let $v(r)=r^{\frac{N}{2}-3}u(r)$, that we consider as a radial function on $\Rb^6$. Then by explicit computation,
 \begin{equation}
  \label{LWD6}
  \partial_t^2v-\Delta_6v+\left(\frac{N}{2}-3\right)\left(\frac N2+1\right)\frac{1}{r^2}v=r^{\frac{N}{2}-3}F,
 \end{equation}
where $\Delta_6=\frac{\partial^2}{\partial_r^2}+\frac{5}{r}\frac{\partial}{\partial_r}$.
 By explicit computation and Hardy's inequality, we see that $\vec{v}(0)\in (\dot{H}^1\times L^2)(\Rb^6)$. Using Strichartz estimates for the wave equation with an inverse square potential \cite[Corollary 3.9]{PlStTZ03}, we obtain 
 $$ \sup_{t\in \Rb} \|\vec{v}(t)\|_{(\dot{H}^1\times L^2)(\Rb^6)}+\|v\|_{L^2(\Rb,L^4(\Rb^6))}\lesssim \|\vec{v}(0)\|_{(\dot{H}^1\times L^2)(\Rb^6)}.$$
 We note that with the notations of \cite[Corollary 3.9]{PlStTZ03}, $\nu^2=\frac{N^2}{4}-N+1$ and $\lambda+\alpha-\frac{n}{q}=1/2$ (in the case of the $L^2L^4$ Strichartz estimate), so that the assumption $\nu>\lambda+\alpha-\frac{n}{q}$ is satisfied if and only if $N\geq 4$. Going back to $u$, and observing that, by Hardy's inequality, for a large constant $C$,
 $$ \|v(t)\|_{\dot{H}^1(\Rb^6)}\approx \|v(t)\|_{\dot{H}^1(\Rb^6)}+C\|r^{-1}v(t)\|_{L^2(\Rb^6)}\approx \|u(t)\|_{\dot{H}^1(\Rb^N)},$$
 we obtain \eqref{Weight_Strichartz}.
\end{proof}

We claim
\begin{lemma}
\label{lem:lipschitz}
Assume $N\in \{3,4,5,6\}$ and $\varphi$ is of the form \eqref{eq:nonlinearitypower}. Let $u,v \in L^{\frac{N+2}{N-2}}L^{\frac{2(N+2)}{N-2}}(\{r>R+|t|\})$ for some $R\in\Rb$. Then $\varphi(u)$ and $\varphi(v)$ are in $L^1L^2(\{r>R+|t|\})$ and
\begin{multline}
 \label{lipschitz1}
 \|\varphi(u)-\varphi(v)\|_{L^1L^2(\{r>R+|t|\})}\\
 \lesssim \|u-v\|_{L^{\frac{N+2}{N-2}}(\{r>R+|t|\})}\left(\|u\|_{L^{\frac{N+2}{N-2}}(\{r>R+|t|\})}^{\frac{4}{N-2}}+\|v\|_{L^{\frac{N+2}{N-2}}(\{r>R+|t|\})}^{\frac{4}{N-2}}\right).
\end{multline}
 Assume $N\geq 4$ and $\varphi$  is of the form \eqref{eq:nonlinearityanalytic}. Let $u,v\in C^0(\{|x|>R+|t|\})$ such that 
 \begin{gather}
 \label{assumption_u_v}
 \sup_{\substack{t\in \Rb\\ r>R+|t|}} r^{\frac{N}{2}-1}\Big(|u(t,r)|+|v(t,r)|\Big)=M<\infty\\
  \label{assumption_u_v'}
 r^{\frac{N-6}{4}} u,\;r^{\frac{N-6}{4}} v  \in L^2L^4(\{|x|>R+|t|\}).
 \end{gather} 
 Then $\varphi(u)$ and $\varphi(v)$ are in $L^1L^2(\{r>R+|t|\})$ and
\begin{multline}
 \label{lipschitz2}
 \|\varphi(u)-\varphi(v)\|_{L^1L^2(\{r>R+|t|\})}\\
 \lesssim_M \left\|r^{\frac{N-6}{4}}(u-v)\right\|_{L^2L^4(\{r>R+|t|\})}\left(\left\|r^{\frac{N-6}{4}}u\right\|_{L^2L^4(\{r>R+|t|\})}+\left\|r^{\frac{N-6}{4}}v\right\|_{L^2L^4(\{r>R+|t|\})}\right).
\end{multline}
\end{lemma}
\begin{remark}
By the radial Sobolev embedding, if $u$ and $v$ are in 
$C^0(\Rb,\dot{H}_{R+|t|}^1)$, and satisfy
 $$\sup_{t\in \Rb}\left(\|u(t)\|_{\dot{H}^1_{R+|t|}}+\|v(t)\|_{\dot{H}^1_{R+|t|}}\right)=M/C,$$
 then \eqref{assumption_u_v} is satisfied.
\end{remark}

\begin{proof}
The estimate \eqref{lipschitz1} follows from the pointwise inequality 
$$\left| |u|^{\frac{4}{N-2}}u-|v|^{\frac{4}{N-2}}v\right|\lesssim |u-v|\left(|u|^{\frac{4}{N-2}}+|v|^{\frac{4}{N-2}}\right)$$
and H\"older's inequality. 

% To prove \eqref{lipschitz2}, we first observe that by the radial Sobolev embedding, 
% \begin{equation}
%  \label{Sobolev_uv}
% \forall t\in \Rb,\quad \forall r>R+|t|,\quad |u(t,r)|+|v(t,r)|\lesssim \frac{\tau}{M} r^{1-\frac{N}{2}}.
% \end{equation}
To prove \eqref{lipschitz2}, we observe that 
\begin{equation}
\label{NLLipsch1}
\varphi(|x|,u)-\varphi(|x|,v)=|x|^{-\frac{N}{2}-1}\left(\vartheta\big(|x|^{\frac{N}{2}-1}u\big)-\vartheta\big(|x|^{\frac{N}{2}-1}v\big)\right), 
\end{equation} 
where $\vartheta$ is the analytic function,
$$ \vartheta(s)=\sum_{k\geq 2} \varphi_k s^k,\quad |s|< 2M.$$
By the mean value theorem, there exists $C>0$ such that for $\max(|s|,|\sigma|)\leq M$, one has $\left|\vartheta(s)-\vartheta(\sigma)\right|\leq C|s - \sigma|\left(|s|+|\sigma|\right)$. Combining with \eqref{assumption_u_v}, we obtain
\begin{equation}
 \label{NLLipsch2}
 \forall t\in \Rb,\quad \forall r>R+|t|,\quad \left|\vartheta(r,u)-\vartheta(r,v)\right|\leq C |x|^{\frac N2-3}|u-v|(|u|+|v|),
\end{equation} 
and the estimate \eqref{lipschitz2} follows from H\"older's inequality.
\end{proof}

\subsection{Local well-posedness outside wave cones}

We refer to \cite{DuKeMe21a} for the notion of solution of \eqref{eq:nonlinearwave2} outside wave cones. We have the following well-posedness result:
\begin{proposition}
\label{pr:well-posedness}
Assume that $N\geq 4$ and that $\varphi$ is of the form \eqref{eq:nonlinearityanalytic}, or that $N\geq 3$ and that $\varphi$ is of the form \eqref{eq:nonlinearitypower}. Let $R>0$ and $(u_0,u_1)\in \Hc_R$, radial. Then there exist $T_-<0<T_+$ (which can be infinite), and a solution $u$ of \eqref{eq:nonlinearwaveintro} on $\{|x|>R+|t|,\; T_-<t<T_+\}$ such that $\vec{u}\in C^0((-T_-,T_+),\Hc_R)$, $u\in S(\{|x|>R+|t|, t\in I\})$ for all compact interval $I\subset (T_-,T_+)$. Furthermore, if $T_+<\infty$ (respectively $|T_-|<\infty$), one has $u\notin S(\{|x|>R+|t|, t\in (0,T_+)\})$ (respectively 
$u\notin S(\{|x|>R+|t|, t\in (T_-,0)\})$).
\end{proposition}
In the case where $u=|u|^{\frac{4}{N-2}}u$, the proof is an adaptation, using finite speed of propagation, of the well-posedness theory (from \cite{KeMe08}, \cite{BuCzLiPaZh13}). We refer to \cite{DuKeMe21a} for the details.

The proof can easily be adapted to nonlinearities of the form \eqref{eq:nonlinearityanalytic}, using finite speed of propagation, radial Sobolev embedding, and Lemma \ref{lem:weighted_Strichartz} and \ref{lem:lipschitz}.

\subsection{Decay of non-radiative solutions in odd dimension}
We next state a decay estimate for non-radiative solutions, which will be needed in Subsection \ref{sub:uniqueness2}, and follows from the main result of \cite{DuKeMe21a}.
\begin{proposition}
 \label{pr:aprioribound}
 Assume that $N\geq 3$ is odd. Let $R_0\geq 0$, and let $u$ be a nonradiative solution of \eqref{id:nonlinearwave} for $\{|x|>R_0+|t|\}$. Then there exists $R_1\geq R$ and $C_1$, depending on $u$, such that 
 \begin{equation}
 \label{aprioribound}
  \forall R\geq R_1+|t|,\quad \|u(t)\|_{\Hc_R}\leq \frac{C_1}{R^{1/2}},\quad |u(t,R)|\leq \frac{C_1}{R^{\frac{N-1}{2}}}.
 \end{equation} 
\end{proposition}
\begin{proof}
 By the radial Sobolev inequality, the second inequality in \eqref{aprioribound} follows from the first one. The first one is a consequence of Theorem 2 of \cite{DuKeMe21a}, according to which there exists large positive constants $R_1$ and $C_1$ (depending on $u$), $\ell\in \Rb$ and
 $$\Xi \in \bigcup_{0\leq l \leq \lfloor \frac{N-3}{4} \rfloor}\left(\frac{1}{|x|^{N-2-2l}},0\right)\cup \bigcup _{0\leq l \leq \lfloor \frac{N-5}{4}\rfloor} \left(0,\frac{1}{|x|^{N-2-2l}}\right)$$
 such that for $R\geq R_1+|t|$,
 $$ \left\| u(t)-\ell \Xi\right\|_{\Hc(R)}\leq C_1\max\left(\frac{1}{R^{(k-\frac{1}2)\frac{N+2}{N-2}}},\frac{1}{R^{k+\frac 12}}\right),$$
 where $1\leq k\leq \frac{N-1}{2}$ is the integer such that there exists $c>0$,
 $$ \left\|\Xi\right\|_{\Hc_R}=\frac{c}{R^{k-\frac{1}{2}}}.$$
 The inequality \eqref{aprioribound} follows immediately.
 \end{proof}

\section{Existence of non-radiative solutions}

We assume throughout this section that $N\geq 3$.

\subsection{The free wave equation}

We start by studying solutions to the free wave equation in the energy space
\begin{equation} \label{eq:freewave}
\left\{ 
\begin{array}{l l} \pa_t^2 u-\Delta u=0,\\
\vec u(0)=(u_0,u_1)\in \mathcal H.
\end{array}
\right.
\end{equation}
We recall that $u$ is non-radiative for $|x|>R+|t|$ if it satisfies \eqref{eq:nonradiativeintro}.

\begin{lemma} \label{lem:freenonrad}

For any $0\leq l\leq l_0$, there exists an even polynomial $p_{2l}$ of degree $2l$ with $p_{2l}(0)=1$ such that
$$
\phi_{2l}(t,x)=\frac{1}{|x|^{N-2-2l}} p_{2l}\left(\frac{t}{|x|} \right)
$$
satisfies $\pa_t^2 \phi_{2l} -\Delta \phi_{2l}=0$ for all $t\in \mathbb R$ and $x\neq 0$, and, for any $R\in \mathbb R$, is non-radiative for $|x|>R+|t|$.

For any $0\leq l\leq l_1$, there exists an odd polynomial $p_{2l+1}$ of degree $2l+1$ with $\pa_\sigma p_{2l+1}(0)=1$ such that
$$
\phi_{2l+1}(t,x)=\frac{1}{|x|^{N-2l-3}} p_{2l+1}\left(\frac{t}{|x|} \right)
$$
satisfies $\pa_t^2 \phi_{2l+1} -\Delta \phi_{2l+1}=0$ for all $t\in \mathbb R$ and $x\neq 0$, and, for any $R\in \mathbb R$, is non-radiative for $|x|>R+|t|$.

\end{lemma}

\begin{proof}

For $m\in  \mathbb N$ we have the identity:
\begin{equation}\label{id:freenonrad1}
(\pa_t^2-\Delta)\left( \frac{1}{|x|^{N-2-m}}f(\frac{t}{|x|}) \right)=\frac{1}{|x|^{N-m}} (-\mathcal L_{N,m} f)(\frac{t}{|x|}),
\end{equation}
where
\begin{equation}\label{id:freenonrad2}
\mathcal L_{N,m}=-(1-\sigma^2)\pa_\sigma^2+(N-1-2m)\sigma\pa_\sigma -(N-2-m)m.
\end{equation}
We have
$$
\mathcal L_{N,m}(\sigma^{\alpha})=\mu_{N,m,\alpha}\sigma^{\alpha}-\alpha (\alpha-1)\sigma^{\alpha-2}, \quad \mu_{N,m,\alpha}=\alpha^2+(N-2-2m)\alpha-(N-2-m)m.
$$
We remark that for $0\leq m\leq \min (2l_0,2l_1+1)$ and $\alpha\in \mathbb N$, $\mu_{N,m,\alpha}=0$ if and only if $\alpha=m$. Therefore, writing $m=2l+\delta$ where $\delta =0$ or $\delta=1$, and looking for a zero of $\mathcal L_{N,m}$ of the form $p_{m}=\sum_{j=0}^{l} c_{N,m,j}\sigma^{\delta+2j}$ we get:
$$
\mathcal L_{N,m}\left(\sum_{j=0}^{l} c_{N,m,j}\sigma^{\delta+2j} \right) = \sum_{j=0}^{l-1} (c_{N,m,j}\mu_{N,m,2j+\delta}-(\delta+2j+2)(\delta+2j+1)c_{N,m,j+1})\sigma^{2j+\delta}.
$$
We therefore choose $c_{N,m,0}=1$ and iteratively $c_{N,m,j+1}=\frac{c_{N,m,j}\mu_{N,m,2j+\delta}}{(\delta+2j+2)(\delta+2j+1)}$ for $j=1,...,l$, and get the desired polynomial $p_{m}$. The fact that the solutions $(\phi_{m})_{0\leq m \leq m_0}$ are non-radiative for $|x|>R+|t|$ for any $R\in \mathbb R$ is a direct computation.

\end{proof}

We write $\cbf_0=(c_{0,l})_{0\leq l \leq l_0}\in \mathbb R^{l_0+1}$, $\cbf_1=(c_{1,l})_{0\leq l \leq l_1}\in \mathbb R^{l_1+1}$ and $\cbf=(c_{0,0},c_{1,0},c_{0,1},c_{1,1},...)\in \mathbb R^{m_0+1}$ (i.e. $c_{m}=c_{0,\frac m2}$ for $m$ even, and $c_{m}=c_{1,\frac{m-1}{2}}$ for $m$ odd), and define:
\begin{align}
\label{id:aF} a_F[\cbf](t,x)&=\sum_{0\leq l\leq l_0} \frac{c_{0,l}}{|x|^{N-2l-2}}p_{0,l}(\frac{t}{|x|})+\sum_{0\leq l \leq l_1} \frac{c_{1,l}}{|x|^{N-2l-3}}p_{1,l}(\frac{t}{|x|})
&=\sum_{0\leq m\leq m_0}c_m \phi_m(t,x).
\end{align}
By a direct estimate:
\begin{equation}
\label{bd:energyaFiscbfR2}
 |\cbf|_R^2 \approx \int_{|x|> R} |\nabla_{t,x}a_F[\cbf](0,x)|^2dx \approx \sup_{t\in \mathbb R} \ \int_{|x|> R+|t|} |\nabla_{t,x}a_F[\cbf](t,x)|^2dx,
\end{equation}
with implicit constants independent of $R$. It follows from Lemma \ref{lem:freenonrad} that $\pa_t^2 a_F -\Delta a_F=0$ for all $t\in \mathbb R$ and $x\neq 0$, and, for any $R\in \mathbb R$, that $a_F$ is non-radiative for $|x|>R+|t|$. It also follows for the initial datum that 
\begin{equation}\label{id:aF0}
\vec a_F(0)=\left(\sum_{0\leq l \leq l_0} \frac{c_{0,l}}{|x|^{N-2l-2}} \ , \  \sum_{0\leq l \leq l_1} \frac{c_{1,l}}{|x|^{N-2l-2}} \right)
\end{equation}
Non-radiative free waves are classified and correspond to the examples of Lemma \ref{lem:freenonrad}, as was recalled in the Classification \ref{pr:nonradiativefree}. On their orthogonal, there holds the channels of energy estimate, with restriction to symmetric data in even dimensions:

\begin{proposition}[Channels of energy estimate \cite{KeLaLiSc15,CoKeSc14,DuKeMaMe22,LiShenWei21P}] \label{pr:channels}

There exists $C>0$ such that the following holds true for any $R>0$. Assume $u$ solves \eqref{eq:freewave}, and in addition if $N$ is even that $u$ is even in time if $N\equiv 4 \mod 4$ or odd in time if $N\equiv 6\mod4$. Then:
\begin{equation}\label{bd:channels}
\| \Pi_{\mathcal H,R} \vec u(0)\|_{\mathcal H_R}^2\leq C \sum_{\pm}\lim_{t\to \pm \infty} \int_{|x|>R+|t|} |\nabla_{t,x}u(t,x)|^2dx.
\end{equation}

\end{proposition}

\begin{remark} \label{re:failurechannels}

If $N$ is even but $u$ is not even in time if $N\equiv 4 \mod 4$ or odd in time if $N\equiv 6\mod4$ then \eqref{bd:channels} is no longer true, see \cite{CoKeSc14}.

\end{remark}

Solutions to the free wave equation \eqref{eq:freewave} evolve asymptotically as $t\to \pm \infty$ towards self-similar solutions, associated to so-called radiation profiles. Non-radiative solutions and channels of energy estimates are intimately related to them.

\begin{proposition}[Radiation profiles for free waves \cite{Friedlander62,Friedlander80,DuKeMe19,LiShenWei21P}] \label{construction:pr:radiation}

For any radial solution $u$ to \eqref{eq:freewave}, there exist $G_{\pm}\in L^2(\mathbb R)$ such that 
\begin{equation}\label{construction:id:radiation}
\sum_{\pm} \lim_{t\rightarrow \pm \infty} \int_{0}^\infty \left|r^{\frac{N-1}{2}} (\pa_t u,\pa_r u)(t,r)-(G_\pm,\mp G_\pm)(r-|t|)\right|^2dr =0.
\end{equation}
The maps $\vec u(0)\mapsto \sqrt{2|\mathbb S^{N-1}|} G_\pm$ are isometries between $\mathcal H$ and $L^2(\mathbb R)$, with
\begin{equation}\label{construction:id:radiationrelation}
G_+(\rho)=(-1)^{\frac{N-1}{2}}G_-(-\rho) \mbox{ for }N\mbox{ odd} \quad \mbox{and}\quad G_+(\rho)=(-1)^{\frac{N}{2}} (\mathfrak H G_-)(-\rho)\mbox{ for }N\mbox{ even}
\end{equation}
for $\mathfrak H$ the Hilbert transform. In addition, $\vec u(0)\in \dot H^2\times \dot H^1$ if and only if $G_+ \in \dot H^1$ (equivalently, if and only if $G_- \in \dot H^1$) in which case $(- \pa_\rho G_+,\pa_\rho G_-)$ are the radiation profiles of $\pa_t u$.

\end{proposition}

Only the last gain of regularity property is not proved in \cite{Friedlander62,Friedlander80,DuKeMe19,LiShenWei21P}, so we prove it in Appendix \ref{sec:reggrain}. We shall need a localised gain of regularity:

\begin{lemma} \label{lem:gainregradiation2}

Assume that $u$ solves \eqref{eq:freewave} and its radiation profiles $G_\pm$ satisfy $\pa_\rho G_+(u)\in L^2(\rho \geq R)$ and $\pa_\rho G_-(u)\in L^2(\rho \geq R)$ for some $R>0$. Then $(u_1,\Delta u_0)\in \mathcal H_{8R}$ with
\begin{equation}\label{bd:gainregradiation2}
\| (u_1,\Delta u_0)\|_{\mathcal H_{8R}}\lesssim \| \pa_\rho G_+\|_{L^2(\rho\geq R)}+\| \pa_\rho G_-\|_{L^2(\rho\geq R)}+\frac 1R \| G_+\|_{L^2(\mathbb R)}.
\end{equation}

\end{lemma}

The proof is done in Appendix \ref{sec:reggrain}. 

We next fix $R>0$ and consider a radial solution of \eqref{eq:freewave} for $|x|>R+|t|$, with initial data $(u_0,u_1)\in \Hc_R$. By  Proposition \ref{construction:pr:radiation} and finite speed of propagation, there exists $G_{\pm}\in L^2([R,\infty))$ such that 
\begin{equation}
 \label{T11}
 \sum_{\pm} \lim_{t\rightarrow \pm \infty} \int_{R+|t|}^\infty \left|r^{\frac{N-1}{2}} (\pa_t u,\pa_r u)(t,r)-(G_\pm,\mp G_\pm)(r-|t|)\right|^2dr =0.
\end{equation} 
If furthermore $R'>R$ and $(u_0,u_1)(r)=0$ for a.a. $r>R'$, then by finite speed of propagation, $G_{\pm}(\rho)=0$ for a.a. $\rho>R'$. Denote by \begin{equation}
\label{T12}
\Hc_{R,R'}=\{(u_0,u_1)\in \Hc_R,\quad (u_0,u_1)(r)=0 \text{ for a.a. }r>R'\},                                                                                                    \end{equation} 
so that by the considerations above the map $\Gc_{R,R'}:(u_0,u_1)\mapsto (G_{+},G_-)$ is a bounded map from $\Hc_{R,R'}$ to $\left(L^2([R,R'])\right)^2$, where $L^2([R,R'])$ is identified with the subspace of $L^2([R,\infty))$ of functions that vanish for almost all $r>R'$, and is endowed with the norm defined by 
$$\|\varphi\|_{L^2([R,R'])}^2=\int_R^{R'} |\varphi(\rho)|^2d\rho.$$
\begin{proposition}\label{pr:T10}
There exists $\delta_0$ with the following property.
Let $R,R'$ be two positive numbers with $R<R'$ and $\frac{R'-R}{R}\leq \delta_0$.
The map $\Gc_{R,R'}$ defined above is an isomorphism from $\Hc_{R,R'}$ to $\left(L^2([R,R'])\right)^2$: for any $G_{\pm}\in \left(L^2([R,R'])\right)^2$, there exists a unique $(u_0,u_1)\in \Hc_{R,R'}$ so that the corresponding solution $u$ of \eqref{eq:freewave} satisfies \eqref{T11}. Furthermore,
$$ \|(u_0,u_1)\|_{\Hc_R}\lesssim \sum_{\pm}\|G_{\pm}\|_{L^2([R,R'])}.$$
\end{proposition}
\begin{proof}
\noindent\textbf{Step 1. Reduction.}
Let $u$ be a solution of \eqref{eq:freewave} for $\{|x|>R+|t|\}$. We notice that
 \begin{equation}
\label{T13}
 \left\{
 \begin{aligned}
  u \text{ is odd in time }&\iff u_0=0 \text{ for }|x|>R\Longrightarrow G_+=G_-\text{  on }[R,\infty)\\
 u \text{ is even in time }&\iff u_1=0 \text{ for }|x|>R\Longrightarrow G_+=-G_-\text{  on }[R,\infty)
 \end{aligned}\right.
 \end{equation}
(here and in all the proof, the equality of $L^2$ functions have to be understood as holding almost everywhere).

We will decompose the solution $u$ into its even and odd parts, and treat separately the two cases. Namely, we will consider the maps $\Psi_0$ and $\Psi_1$ from $L^2([R,R'])$ to $L^2([R,R'])$, where 
\begin{itemize}
 \item $\Psi_0:\varphi\mapsto G_+$
where $G_+$ satisfies \eqref{T11} where $u$ is the solution of \eqref{eq:freewave} with initial data $(u_0,0)$ where $u_0\in \dot{H}^1_R$ is given by
\begin{equation}
\label{give_u0}
u_0(r)=\begin{cases}
           0 & \text{ if }r>R'\\
           \int_{R'}^r \varphi(\rho)d\rho&\text{ if }R<r<R'.
          \end{cases} 
\end{equation} 
\item $\Psi_1:u_1\mapsto G_+$, where $G_+$ satisfies \eqref{T11} for the solution $u$ of \eqref{eq:freewave} with initial data $(0,u_1)$ (with the convention $u_1(r)=0$ for $r\geq R'$).
\end{itemize}
Note that (with $u_0$, $\varphi$ and $G_+$ as in the definition of $\Psi_0$):
$$ \|u_0\|^2_{\dot{H}_R}\lesssim R^{N-1} \|\varphi\|^2_{L^2([R,R'])},$$
and thus, by Proposition \ref{construction:pr:radiation}, $\|G_+\|_{L^2([R,R'])}\lesssim \|\varphi\|_{L^2([R,R'])}$, so that $\Psi_0$ is a bounded linear operator. Similarly $\Psi_1$ is a bounded linear operator.

We will prove that if $(R'-R)/R$ is small enough, then $\Psi_0$ and $\Psi_1$ are invertible and have bounded inverse. Assuming this, and noting that, by \eqref{T13},
$$ \Gc_{R,R'}(u_0,u_1)=\Big(\Psi_0(\partial_r u_0)+\Psi_1(u_1),-\Psi_0(\partial_ru_0)+\Psi_1(u_1)\Big),$$
we will would obtain the desired conclusion.

We will deduce the invertibility of $\Psi_0$ and $\Psi_1$ when $\frac{R'-R}{R}$ is small from the inequalities
\begin{gather}
 \label{Tbound}
 \forall \varphi\in L^2([R,R']),\quad \left\| -2\rho^{\frac{1-N}{2}}\Psi_0(\varphi)(\rho)-\varphi(\rho)\right\|_{L^2_{\rho}([R,R'])}\lesssim \frac{R'-R}{R} \|\varphi\|_{L^2([R,R'])}\\
\label{Tbound'}
 \forall \varphi\in L^2([R,R']),\quad \left\| 2\rho^{\frac{1-N}{2}}\Psi_1(\varphi)(\rho)-\varphi(\rho)\right\|_{L^2_{\rho}([R,R'])}\lesssim \frac{R'-R}{R} \|\varphi\|_{L^2([R,R'])}.
\end{gather} 
 We will focus on proving \eqref{Tbound} for $j=0$, and omit the proof of \eqref{Tbound'} which is very similar.

\medskip

\noindent\textbf{Step 2. Approximation and end of the proof.}

To prove \eqref{Tbound} for $j=0$, we consider $\varphi\in L^2([R,R'])$, and $u$ the solution of \eqref{eq:freewave} with initial data $(u_0,0)$, where $u_0$ is given by \eqref{give_u0}. We will approximate $u$ by 
$$u_{app}(t,r)=\frac{1}{2r^{\frac{N-1}{2}}}\left( (r-t)^{\frac{N-1}{2}}u_0(r-t)+(r+t)^{\frac{N-1}{2}}u_0(r+t) \right).$$
By explicit computation,
$$ (\partial_t^2-\Delta)u_{app}=-\frac{(N-1)(N-3)}{4r^2}u_{app}.$$
Letting $h=u-u_{app}$, we deduce
$$ (\partial_t^2-\Delta)h=\frac{(N-1)(N-3)}{4r^2}u_{app},\quad |x|>R+|t|,\quad \vec{h}(0)=(0,0).$$
By energy estimates and finite speed of propagation, for all $t\in \Rb$,
$$\|\vec{h}(t)\|_{\Hc_{R+|t|}}\lesssim \left\|\frac{1}{r^2} u_{app} \right\|_{(L^1L^2)(\{|x|>R+|t|\})}.$$
By explicit computations, we have $\left\|\frac{1}{r^2} u_{app} \right\|_{(L^1L^2)(\{|x|>R+|t|\})}\lesssim \frac 1R\|u_0\|_{L^2_R}$ and thus
\begin{equation}
 \label{T40}
\|\vec{h}(t)\|_{\Hc_{R+|t|}}\lesssim \frac{1}{R}\|u_0\|_{L^2_R}.
 \end{equation} 
Since for $r\in [R,R']$, $u_0(r)=\int_{r}^{R'}\varphi$, we obtain by Cauchy-Schwarz 
$$|u_0(r)|\lesssim (R'-R)^{1/2} \|\varphi\|_{L^2([R,R'])}$$
and thus 
$$\|u_0\|_{L^2_R}\lesssim (R'-R)R^{\frac{N-1}{2}}\|\varphi\|_{L^2([R,R'])}.$$
Combining with \eqref{T40}, we deduce
\begin{equation}
 \label{T41}
\|\vec{h}(t)\|_{\Hc_{R+|t|}}\lesssim \frac{R'-R}{R}{R'}^{\frac{N-1}{2}}\|\varphi\|_{L^2([R,R'])}\lesssim \frac{R'-R}{R}R^{\frac{N-1}{2}}\|\varphi\|_{L^2([R,R'])}.
 \end{equation} 
Since
$$\lim_{t\to+\infty} \int_{R+|t|}^{\infty} \left(r^{\frac{N-1}{2}}\partial_ru_{app}(t,r)-\frac{1}{2}(r-t)^{\frac{N-1}{2}}\varphi(r-|t|)\right)^2dr=0,$$
we deduce from \eqref{T41}
$$\left\|\Psi(\varphi)(\rho)-\frac{\rho^{\frac{N-1}{2}}}{2}\varphi(\rho)\right\|_{L^2_{\rho}([R,R'])}\lesssim \frac{R'-R}{R}R^{\frac{N-1}{2}}\|\varphi\|_{L^2([R,R'])},$$
concluding the proof of \eqref{Tbound}. The proof of \eqref{Tbound'} is similar, approaching the solution with initial data $(0,u_1)$ by  
$$ \tilde{u}_{app}(t,r)=\frac{1}{{2r^{\frac{N-1}{2}}}}(u_1(r+t)-u_1(r-t)).$$
\end{proof}

\subsection{The inhomogenous wave equation}

We now turn to solutions to the inhomogeneous wave equation:
\begin{equation} \label{eq:inhomwave}
\left\{ 
\begin{array}{l l} \pa_t^2 u-\Delta u=f,\\
\vec u(0)=(u_0,u_1)\in \mathcal H.
\end{array}
\right.
\end{equation}

We start by proving, as a corollary to Propositions \ref{construction:pr:radiation} and \ref{pr:channels}, that solutions to the inhomogeneous wave equation also have radiation profiles and enjoy channel of energy estimates:

\begin{corollary} \label{cor:inhomogeneousradiation}

Let $u$ solve \eqref{eq:inhomwave} with $f\in L^1L^2$. There exists $G_\pm$ with $\| G_\pm \|_{L^2(\mathbb R)}\lesssim \| \vec u(0)\|_{\mathcal H}+\| f\|_{L^1L^2}$ such that the solution $u$ to \eqref{eq:inhomwave} has radiation profiles $G_\pm$ as $t\to \pm \infty $ in the sense that \eqref{construction:id:radiation} holds true. Moreover:
\begin{equation}\label{bd:channels2}
\left| \begin{array}{l l l} \| \Pi^\perp_{\mathcal H,R} \vec u(0)\|_{\mathcal H_R} & \mbox{if }N\mbox{ is odd},\\
\| \Pi^\perp_{\dot H^1,R} u_0 \|_{\dot H^1_R} & \mbox{if }N\equiv 4\mod 4,\\
\| \Pi^\perp_{L^2,R} u_1\|_{L^2_R}  &  \mbox{if }N\equiv 6\mod 4,
 \end{array} \right. \ \ \lesssim \| \indic(r>R+|t|)f\|_{L^1L^2}+ \sum_{\pm} \| G_\pm \|_{L^2(\rho\geq R)}
\end{equation}

\end{corollary}

\begin{proof}

The existence of radiation profiles follows from Proposition \ref{construction:pr:radiation} and standard energy estimates, see \cite{DuKeMaMe22}. The estimate \eqref{bd:channels2} follows from Proposition \ref{pr:channels} and standard energy estimates, applied to $u$ for $N$ odd, $u_+$ for $N\equiv 4\mod 4$, and $u_-$ for $N\equiv 6\mod 4$.

\end{proof}

The main result of this section is that an initial data can always be prescribed to ensure that the solution to \eqref{eq:inhomwave} is non-radiative for $|x|>|t|$, with a suitable estimate. In odd dimensions, or in even dimensions with certain symmetry conditions, this follows as direct consequence of \cite{DuKeMe19,LiShenWei21P}:

\begin{proposition} \label{pr:nonradiativeforcing}

There exists $C>0$ such that the following holds true. Assume $f\in L^1L^2$ and in addition for $N$ even that $f$ is even in time if $N\equiv 4 \mod 4$ (resp. odd in time if $N\equiv 6 \mod 4$). Then there exists a unique initial data $(u_0,u_1)\in \mathcal H$ with
\begin{equation}\label{id:nonradiativeforcingenergy}
\| (u_0,u_1)\|_{\mathcal H}\leq C \| \indic (|x|>|t|)f\|_{L^1L^2}
\end{equation}
such that the solution to \eqref{eq:inhomwave} is non-radiative for $|x|> |t|$. For even dimensions, $u$ is even (resp. odd) in time if $N\equiv 4 \mod 4$ (resp. $N\equiv 6 \mod 4$).

\end{proposition}

\begin{remark}

We believe the result to be false if $N$ is even and $f$ is odd in time if $N\equiv 4 \mod 4$, or even in time if $N\equiv 6 \mod 4$.

\end{remark}

\begin{proof}

Let $v$ be the solution to
$$
\left\{ \begin{array}{l l} \pa_t^2 v-\Delta v=\indic(|x|>|t|)f,\\
\vec v(0)=\vec 0.
\end{array}
\right.
$$
Then, by Corollary \ref{cor:inhomogeneousradiation}, $v$ admits radiation profiles $H_\pm$ as $t\to \pm \infty$ in the sense that \eqref{construction:id:radiation} holds true replacing $(u,G_\pm)$ by $(v,H_\pm)$ in the formula, and $\| H_\pm \|_{L^2(\mathbb R)}\lesssim \| \indic(|x|>|t|)f\|_{L^1L^2}$.

In the case of $N$ odd, we define $G_+(\rho)=\indic(\rho >0) H_+(\rho)+\indic(\rho <0)(-1)^{\frac{N-1}{2}}H_-(-\rho)$ and using Proposition \ref{construction:pr:radiation} we let $w$ be the solution to \eqref{eq:freewave} whose radiation profile as $t\to \infty$ is $G_+$. It thus satisfies $\| \vec w(0)\|_{\mathcal H}\lesssim \| \indic (|x|>|t|)f\|_{L^1L^2}$. By \eqref{construction:id:radiationrelation} the radiation profile of $w$ as $t\to -\infty$ is $G_-(\rho)=(-1)^{\frac{N-1}{2}}G_+(-\rho)$, so that $G_-(\rho)=H_-(\rho)$ for all $\rho>0$. Let then $u=v-w$. It indeed solves \eqref{eq:inhomwave} and satisfies \eqref{id:nonradiativeforcingenergy}. Moreover, its radiation profile are $H_\pm-G_\pm$ as $t\to\pm \infty$. As $H_+(\rho)-G_+(\rho)=0$ and $H_-(\rho)-G_-(\rho)=0$ for all $\rho>0 $, we get by \eqref{construction:id:radiation} that $u$ is non-radiative for $|x|>|t|$.

In the case of $N$ even, then $v$ is even in time as $f$ is if $N\equiv 4\mod 4$ or $v$ is odd in time as $f$ is if $N\equiv 6\mod 4$. By \eqref{construction:id:radiation} this implies that for all $\rho \in \mathbb R$:
\begin{equation}\label{construction:id:symetryHeven}
H_+(\rho)=(-1)^{\frac N2 +1}H_-(\rho).
\end{equation}
By Lemma 5.2. in \cite{LiShenWei21P} for any $H_+\in L^2(\mathbb R^+)$ there exists $G$ such that
\begin{equation}\label{construction:id:defGnonradiative}
G(\rho)-(\mathfrak H G)(-\rho)=2H_+(\rho), \quad \forall \rho>0, \quad \mbox{with} \quad \| G\|_{L^2(\mathbb R)}\leq 2 \| H_+\|_{L^2(\mathbb R^+)}.
\end{equation}
Let then $G_+(\rho)= \frac 12(G(\rho)-(\mathfrak H G)(-\rho))$ and $w$ be the solution to \eqref{eq:freewave} whose radiation profile as $t\to \infty$ is $G_+$. Then by \eqref{construction:id:radiationrelation}, the radiation profile $G_-$ of $w$ as $t\to -\infty$ is
$$
G_-(\rho)= \frac{(-1)^{\frac N2}}{2}(\mathfrak H(G(\cdot)-(\mathfrak H G)(-\cdot)))(-\rho)=(-1)^{\frac N2 +1}G_+(\rho)
$$
so that $G_-(\rho)=H_-(\rho)$ for all $\rho>0$ by \eqref{construction:id:symetryHeven} and \eqref{construction:id:defGnonradiative}. Next, proceeding as for $N$ odd above, one finds that the function $u=v-w$ satisfies the desired conclusions of the Lemma.

\end{proof}

As a direct consequence of Proposition \ref{pr:nonradiativeforcing} and of Classification \ref{pr:nonradiativefree} we have:

\begin{corollary}\label{co:nonradiativeforcing}

There exists $C>0$, such that under the hypotheses of Proposition \ref{pr:nonradiativeforcing}, for any $R>0$, there exists a solution $u$ to \eqref{eq:inhomwave} such that \eqref{id:nonradiativeforcingenergy} holds and
\begin{equation} 
\label{id:nonradiativeforcingortho} (u_0,u_1)=\Pi^\perp_{\mathcal H,R}((u_0,u_1)),
\end{equation}
that is is non-radiative for $|x|> |t|+R$. For any other solution $u'$ to \eqref{eq:inhomwave} that is non-radiative for $|x|>R+|t|$ and which satisfies \eqref{id:nonradiativeforcingortho} there holds $u(t,x)=u'(t,x)$ for $|x|>R+|t|$.

\end{corollary} 

Restricting outside wave cones, and considering a forcing term $f$ that has a small compact support at fixed $t$, we can take off the symmetry assumption on $f$ that appears in Proposition \ref{pr:nonradiativeforcing} when $N$ is even. We recall from \ref{T12} the definition of $\Hc_{R,R'}$
\begin{proposition}
 \label{pr:nonradiative_compactforcing}
 Let $0<R<R'$ with $\frac{R'-R}{R}\leq \delta_0$, $\delta_0$ given by Proposition \ref{pr:T10}. Let  $f \in L^1L^2(\{|x|>R+|t|\})$, radial. 
 Assume that $f(t,x)=0$ for a.a. $|x|>R'+|t|$. Then there exists a unique $(u_0,u_1)\in \Hc_{R,R'}$ such that the solution $u$ of \eqref{eq:inhomwave} on $\{|x|>R+|t|\}$ is nonradiative for $|x|>R+|t|$. Furthermore, 
 $$ \|(u_0,u_1)\|_{\Hc_R}\lesssim \|f\indic(|x|>R+|t|)\|_{L^1L^2}.$$
\end{proposition}
\begin{proof}
 We let $v$ be the solution of
\begin{equation*} %\label{eq:inhomwave}
\left\{ 
\begin{array}{l l} \pa_t^2 v-\Delta v=f,\quad |x|>R+|t|\\
\vec v(0)=(0,0).
\end{array}
\right.
\end{equation*}
By Corollary \ref{cor:inhomogeneousradiation} and finite speed of propagation, there exist radiation profiles $G_{\pm}\in L^2([R,+\infty))$ such that \eqref{T11} holds with $u=v$. By finite speed of propagation, $v(t,x)=0$ if $|x|>R'+|t|$, and thus $G_{\pm}(\rho)=0$ if $\rho>R'$. By Proposition \ref{pr:T10}, there exists a solution $w$ of the free wave equation \eqref{eq:freewave} such that \eqref{T11} holds with $u=w$, $\vec{w}(0)\in \Hc_R$, $\vec{w}(0,r)=0$ for $r>R'$ and 
$$\|\vec{w}(0)\|_{\Hc_R}\lesssim \sum_{\pm}\|G_{\pm}\|_{L^2_R}\lesssim \|f\indic(|x|>R+|t|)\|_{L^1L^2}.$$
The function $u=v-w$ satisfies the desired conclusion.
\end{proof} 

The main result of this Section is a twofold extension of Corollary \ref{co:nonradiativeforcing}. First, we will remove the symmetry assumptions on $f$ for $N$ even, thanks to an extra regularity for $f$ that will hold true for our applications to the nonlinear problems \eqref{eq:nonlinearwaveintro}. Second, we will obtain additional decay for $u$ if $f$ has additional decay, which will be used to show the extra decay \eqref{bound_a_aF} of non-radiative nonlinear waves, and also to show an approximation result for their classification in Section \ref{sec:uniquenesseven}.

Recall that $W_R^\kappa$ and $W_{R}^{'\kappa}$ are defined by \eqref{id:defWkappa} and \eqref{id:defW'kappa}. Observe that, by the radial Sobolev embedding if $u\in W_1^\kappa$:
\begin{equation}\label{bd:SobolevW}
\forall |x|>R+|t|, \qquad |u(t,x)|\lesssim \frac{1}{|x|^{N/2-1+\kappa }} \| u\|_{W^\kappa_R}.
\end{equation}

\begin{proposition} \label{pr:nonradiativeforcingmainodd}

For any $N\geq 3$ odd and $\kappa>0$ with $\kappa\neq \kappa_N=1/2$, there exists $C>0$ such that the following holds true for any $R>0$. Assume $f\in L^1L^2\cap W^{'\kappa}_R$. Then:

\begin{itemize}
\item \emph{Existence of a non-radiative solution}. There exists a solution $u$ to \eqref{eq:inhomwave} with
\begin{align} \label{bd:nonradiativeforcingmain1odd}
& R^{\kappa}\| \vec u(0) \|_{\mathcal H} +R^{\kappa-\min(\kappa,\kappa_N)} \| u \|_{W_R^{\min(\kappa,\kappa_N)}}\leq C \| f \|_{W_R^{'\kappa}},\\
& \label{id:nonradiativeforcingortho2odd} (u_0,u_1)=\Pi^\perp_{\mathcal H,R}((u_0,u_1)),
\end{align}
that is non-radiative for $|x|>R+|t|$.
\item \emph{Uniqueness}. For any other solution $u'$ to \eqref{eq:inhomwave} that is non-radiative for $|x|>R+|t|$ and which satisfies \eqref{id:nonradiativeforcingortho2odd} there holds $u(t,x)=u'(t,x)$ for $|x|>R+|t|$.
\end{itemize}

\end{proposition}

\begin{proposition} \label{pr:nonradiativeforcingmaineven}

For any $N\geq 4$ even and $\kappa>0$ with $\kappa\neq \kappa_N=1$, there exists $C>0$ such that the following holds true for any $R>0$. Assume $f\in L^1L^2$ with $\sup_{\tilde R>R+|t|}\tilde R^{1+\kappa}\| f(t)\|_{L^2_{\tilde R}}<\infty$ and $\pa_t f\in L^1L^2\cap W^{'\kappa+1}_R$. Then:
\begin{itemize}
\item \emph{Existence of a non-radiative solution}. There exists a solution $u$ to \eqref{eq:inhomwave} with
\begin{align} \label{bd:nonradiativeforcingmain1even}
& R^{\kappa}\| \vec u(0) \|_{\mathcal H} +R^{\kappa-\min(\kappa_N,\kappa)} \| u \|_{W_R^{\min(\kappa_N,\kappa)}}+\left| \begin{array}{l l} R^{1+\kappa}\| u_1\|_{\dot H^1} \ \mbox{ if }N\equiv 4\mod 4,\\ R^{1+\kappa} \| u_0\|_{\dot H^2} \ \mbox{ if }N\equiv 6\mod 4, \end{array}\right. \\
\nonumber &\qquad \qquad  \qquad \qquad \qquad \qquad \qquad \qquad \leq C(\sup_{\tilde R>R+|t|}\tilde R^{1+\kappa}\| f(t)\|_{L^2_{\tilde R}}+\|\pa_t f \|_{W_R^{'1+\kappa}}),\\
&\label{id:nonradiativeforcingortho2even}
(u_0,u_1)=\Pi^\perp_{\mathcal H,R}((u_0,u_1)),
\end{align}
such that $u$ is non-radiative for $|x|>R+|t|$.
\item \emph{Uniqueness}. For any other solution $u'$ to \eqref{eq:inhomwave} that is non-radiative for $|x|>R+|t|$ and which satisfies \eqref{id:nonradiativeforcingortho2even} there holds $u(t,x)=u'(t,x)$ for $|x|>R+|t|$.
\end{itemize}
\end{proposition}

The proof of Propositions \ref{pr:nonradiativeforcingmainodd} and \ref{pr:nonradiativeforcingmaineven} will be given at the end of this section, with the help of several technical Lemmas.

First, we prove that if $f$ in Corollary \ref{co:nonradiativeforcing} enjoys additional decay, then so does $u$. The decay gain is limited by the non-radiative direction that has the slowest decay at infinity, that we might need to single out. We notice that $\vec \phi_{m_0}(0)$ is the initial datum of a non-radiative free wave with the slowest decay as $|x|\to \infty$:
\begin{equation}\label{construction:id:phim0}
\vec \phi_{m_0}(0,x)=\left\{ \begin{array}{l l}
(|x|^{-\frac N2+\frac 12},0) \qquad \mbox{if }N\equiv 3\mod 4,\\ 
(|x|^{-\frac N2},0) \qquad \mbox{if }N\equiv 4\mod 4,\\ 
(0,|x|^{-\frac N2-\frac 12}) \qquad \mbox{if }N\equiv 5\mod 4,\\ 
(0,|x|^{-\frac N2-1}) \qquad \mbox{if }N\equiv 6\mod 4.
\end{array}
\right.
\end{equation}

\begin{lemma} \label{lem:nonradiativeforcinggain}

Let $\kappa>0$ satisfy $\kappa\neq \kappa_N$ and $\kappa<\kappa_N+1$ for $N$ odd or $\kappa<\kappa_N+2$ for $N$ even, $R>0$, $f\in L^1L^2 \cap W^{'\kappa}_R$ and $u$ be given by Corollary \ref{co:nonradiativeforcing}. Then
\begin{itemize}
\item[(i)] If $\kappa<\kappa_N$ there holds
\begin{equation}\label{pr:nonradiativeforcinggain100}
\forall \tilde R>R, \quad \| (u_0,u_1)\|_{\mathcal H_{\tilde R}}\lesssim \tilde R^{-\kappa} \| f\|_{W^{'\kappa}_{R}}.
\end{equation}
\item[(ii)] If $\kappa>\kappa_N$, there exists $c=c(R,f)\in \mathbb R$ with $|c|\lesssim R^{\kappa_N-\kappa}\| f\|_{W^{'\kappa}_R}$ such that $ \vec u(0)=c\vec \phi_{m_0}(0) +(\tilde u_0,\tilde u_1)$ with
$$
\forall \tilde R>R, \quad \| (\tilde u_0,\tilde u_1)\|_{\mathcal H_{\tilde R}}\lesssim \tilde R^{-\kappa} \| f\|_{W^{'\kappa}_{R}}.
$$
\end{itemize}
\end{lemma}

\begin{remark}

\begin{itemize}

\item As $\| \vec{\phi}_{m_0}\|_{\mathcal H_R}\lesssim R^{-\kappa_N}$ by \eqref{construction:id:phim0}, a direct consequence for $u$ given by Lemma \ref{lem:nonradiativeforcinggain} is that:
\begin{equation}\label{pr:nonradiativeforcinggain1}
\forall \tilde R>R, \qquad \|\vec u(0)\|_{\mathcal H_{\tilde R}}\lesssim \tilde R^{-\min(\kappa,\kappa_N)} R^{\min(\kappa,\kappa_N)-\kappa} \| f\|_{W^{'\kappa}_{R}}.
\end{equation}

\item As is clear from the proof, for larger values of $\kappa$ an analogous statement would hold true, up to singling out additional directions among $(\vec \phi_m(0))_{0\leq m \leq m_0}$. We however do not need such result for the purposes of the present article.
\end{itemize}

\end{remark}

\begin{proof}

\noindent \textbf{Step 1}. \emph{Preliminary decompositions.} Let $E=E[N]$ be the set defined by $E=\{0,1,...,m_0\}$ if $N$ is odd, and $E=\{m\in \{0,...,m_0\}, \ m_0-m \mbox{ even }\}$ if $N$ is even. Observe $\vec \phi_m(0,x)=(|x|^{-N+2+m},0)$ for $m$ even, and $\vec \phi_m(0,x)=(0,|x|^{-N+1+m})$ for $m$ odd. We claim that for all $\tilde R>R$ there exists $\cbf[\tilde R]\in \mathbb R^{m_0}$ such that
\begin{equation}\label{id:lemnonradiativeforcinggain1}
\vec u(0)=\sum_{m\in E} c_{m}[\tilde R]\tilde R^{\frac{N}{2}-1-m}\vec \phi_m(0) + \left(\bar u_0, \bar u_1 \right)
\end{equation}
where $\bar u_0=\bar u_0[\tilde R]$ and $\bar u_1=\bar u_1[\tilde R]$ satisfy:
\begin{equation}\label{id:lemnonradiativeforcinggain2}
\| (\bar u_0,\bar u_1)\|_{\mathcal H_{\tilde R}}\lesssim \tilde R^{-\kappa} \| f\|_{W_R^{'\kappa}}.
\end{equation}
We now prove this claim. Let $(\bar u_0,\bar u_1)\in \mathcal H$ be the initial data of a solution $\bar u $ to \eqref{eq:inhomwave}, provided by Corollary \ref{co:nonradiativeforcing} with $\tilde R$. Then $\bar u$ is non-radiative for $|x|>\tilde R+|t|$, and 
$$
\| (\bar u_0,\bar u_1)\|_{\mathcal H_{\tilde R}}\lesssim \| \indic (|x|>\tilde R+|t|)f \|_{L^1L^2}\leq \tilde R^{-\kappa}\| f\|_{W_R^{'\kappa}}.
$$
The function $v=u-\bar u$ thus solves $\pa_t^2 v-\Delta v=0$, and is non-radiative for $|x|>\tilde R+|t|$, so that by Classification \ref{pr:nonradiativefree} there holds $\vec v(0)=\Pi_{\mathcal H,\tilde R}\vec v(0)$. Therefore, there holds $\vec v(0,x)=\sum_{0\leq m\leq m_0}\tilde c_m \vec \phi_m(x)$ for some constants $(\tilde c_m)_{0\leq m\leq m_0}$. In even dimensions, by the symmetry in time properties of of Proposition \ref{pr:nonradiativeforcing} (which are that of Corollary \ref{co:nonradiativeforcing}), if $N\equiv 4 \mod 4$  (resp. if $N\equiv 6 \mod 4$) then both $u$ and $\bar u$ are even in time (resp. odd in time). Hence $\tilde c_m=0$ for all $m\in \{0,...,m_0\}\backslash E$. This implies the result claimed in this first step.\\

\noindent \textbf{Step 2}. \emph{Estimates for dyadic projections on non-radiative directions}. We now prove the result of the Lemma. Let $R_k=2^k R$ and $c_{m,k}=c_k[R_k]$. We match the two decompositions provided by \eqref{id:lemnonradiativeforcinggain1}, giving:
$$
\sum_{m\in E} \left(c_{m,k+1}R^{\frac{N}{2}-1-m}_{k+1}-c_{m,k}R^{\frac{N}{2}-1-m}_k\right) \vec \phi_m(0)=(\bar u_0[R_k]-\bar u_0[R_{k+1}],\bar u_1[R_k]-\bar u_1[R_{k+1}]).
$$
Renormalising, this gives:
$$
\sum_{m\in E} \left(c_{m,k+1}-2^{-\frac{N}{2}+1+m} c_{m,k}\right) \vec \phi_m(0)=\left( (\bar u_0[R_k]-\bar u_0[R_{k+1}])_{(\frac{1}{R_{k+1}})},(\bar u_1[R_k]-\bar u_1[R_{k+1}])_{[\frac{1}{R_{k+1}}]}\right).
$$
Since 
$$
\| \left( (\bar u_0[R_k]-\bar u_0[R_{k+1}])_{(\frac{1}{R_{k+1}})},(\bar u_1[R_k]-\bar u_1[R_{k+1}])_{[\frac{1}{R_{k+1}}]}\right)\|_{\mathcal H_1}\lesssim R_{k}^{-\kappa}\| f\|_{W_R^{'\kappa}}
$$
by \eqref{id:lemnonradiativeforcinggain2}, and since $(\vec \phi_m)_{m\in E}$ are linearly independent in $\mathcal H_1$, we infer that:
$$
\forall m\in E, \quad \forall k\geq 0, \qquad c_{m,k+1}=2^{-\frac{N}{2}+1+m} c_{m,k}+O(R_k^{-\kappa}\| f\|_{W_R^{'\kappa}}).
$$
Let $d_{m,k}=2^{(\frac{N}{2}-1-m)k}c_{m,k}$ then for all $m\in E$:
\begin{equation}\label{bd:lemnonradiativeforcinggain3}
\forall k\geq 0, \qquad d_{m,k+1}=d_{m,k}+O(R^{-\kappa}2^{(\frac N2-1-m-\kappa)k}\| f\|_{W_R^{'\kappa}}).
\end{equation}
By \eqref{id:nonradiativeforcingortho} we have $c_{m,0}=0=d_{m,0}$ for all $m\in E$. We also remark that, by the assumptions on $\kappa$:
\begin{equation}\label{bd:lemnonradiativeforcinggain333}
\frac N2-1-m_0-\kappa=\kappa_N-\kappa \quad \mbox{and} \quad  \forall m\in E\backslash \{m_0\}, \quad \frac N2-1-m-\kappa >0.
\end{equation}
If $\kappa>\kappa_N$, we deduce from \eqref{bd:lemnonradiativeforcinggain3} and \eqref{bd:lemnonradiativeforcinggain333} after summation that there exists $d_{m_0,\infty}\in \mathbb R$ with $|d_{m_0,\infty}|\lesssim R^{-\kappa} \| f\|_{W_R^{'\kappa}}$ such that $|d_{m_0,k}-d_{m_0,\infty}|\lesssim R^{-\kappa} 2^{(\frac N2-1-m-\kappa)k} \| f\|_{W_R^{'\kappa}}$ for all $k\geq 0$. If $\kappa<\kappa_N$ we deduce from \eqref{bd:lemnonradiativeforcinggain3} and \eqref{bd:lemnonradiativeforcinggain333} that $|d_{m_0,k}|\lesssim R^{-\kappa} 2^{(\frac N2-1-m-\kappa)k} \| f\|_{W_R^{'\kappa}}$. Letting $d_{m_0,\infty}=0$ by convention if $\kappa<\kappa_N$ we have in both cases that for $k\geq 0$:
$$
|d_{m_0,k}-d_{m_0,\infty}|\lesssim R^{-\kappa} 2^{(\frac N2-1-m-\kappa)k} \| f\|_{W_R^{'\kappa}}.
$$
Similarly, by \eqref{bd:lemnonradiativeforcinggain3} and \eqref{bd:lemnonradiativeforcinggain333}, we deduce after summation that for $k\geq 0$:
\begin{align*}
& \forall m\in E, \quad  |d_{m,k}|\lesssim R^{-\kappa}2^{(\frac N2-1-m-\kappa)k} \| f\|_{W_R^{'\kappa}} .
\end{align*}
This implies that for all $k\geq 0$:
\begin{align} \label{bd:lemnonradiativeforcinggain1}
& |c_{m_0,k}-2^{-\kappa_N k}d_{m_0,\infty}|\lesssim R^{-\kappa} 2^{-\kappa k} \| f\|_{W_R^{'\kappa}},\\
 \label{bd:lemnonradiativeforcinggain2}&  \forall m\in E, \qquad  |c_{m,k}|\lesssim R^{-\kappa}2^{ -\kappa k} \| f\|_{W_R^{'\kappa}}.
 \end{align}

\noindent \textbf{Step 3}. \emph{End of the proof}. We first rewrite \eqref{bd:lemnonradiativeforcinggain1} as:
\begin{equation} \label{bd:lemnonradiativeforcinggain33}
\forall k\geq 0, \qquad c_{m_0,k} R_k^{\frac N2-1-m_0}=c+O(R_k^{\kappa_N-\kappa}\| f\|_{W_R^{'\kappa}}), \qquad c=d_{m_0,\infty}R^{\kappa_N}.
\end{equation}
Then using \eqref{id:lemnonradiativeforcinggain1}, \eqref{bd:lemnonradiativeforcinggain33} and \eqref{bd:lemnonradiativeforcinggain2}, for any $k\geq 0$:
\begin{align*}
(u(0,x),\pa_t u(0,x)) & =c\vec \phi_{m_0}(0)+O(R_k^{\kappa_N-\kappa}\| f\|_{W_R^{'\kappa}})\vec \phi_{m_0}(0)\\
&+\sum_{m\in E\backslash\{m_0\}} O(R_k^{\frac{N}{2}-1-m-\kappa}\| f\|_{W_R^{'\kappa}})\vec \phi_m(0)+ (\bar u_0[R_k],\bar u_1[R_k])
\end{align*}
Using $\| \vec \phi_m(0)\|_{\mathcal H_{R_k}}\lesssim R_k^{-(\frac{N}{2}-1-m)}$ and \eqref{id:lemnonradiativeforcinggain2} we infer:
$$
\forall k\geq 0, \qquad \| (u(0,x),\pa_t u(0,x))-c\vec \phi_{m_0}\|_{\mathcal H_{R_k}}\lesssim R_k^{-\kappa}\| f\|_{W_R^{'\kappa}}.
$$
By dyadic partitioning, this implies that $\| (u(0,x),\pa_t u(0,x))-c\vec \phi_{m_0}\|_{\mathcal H_{\tilde R}}\lesssim \tilde R^{-\kappa}\| f\|_{W_R^{'\kappa}}$ for all $\tilde R>R$. This shows (ii) if $\kappa>\kappa_N$. If $\kappa<\kappa_N$, this shows (i) using that $\| \vec \phi_m(0)\|_{\mathcal H_{R_k}}\lesssim R_k^{-\kappa_N}$.

\end{proof}

The decay gain of Lemma \ref{lem:nonradiativeforcinggain} allows us to treat, in even dimensions, forcings $f$ that are not even in time if $N\equiv 4 \mod 4$, or odd in time if $N\equiv 6\mod 4$, provided they enjoy a regularity gain.

\begin{lemma} \label{lem:nonradforcingeven}

Assume $N\geq 4$ is even. Let $0<\kappa<2 $, $R>0$ and assume $f\in L^1L^2$ with $\pa_t f\in L^1L^2 \cap W_R^{'1+\kappa}$.

\begin{itemize}
\item If $N\equiv 4 \mod 4$ and $f$ is odd in time then there exists $u_1 \in L^2\cap \dot H^1$ with
\begin{align} \label{nonradforcingeven:energy} 
& R^{\kappa} \| u_1\|_{L^2}+R^{1+\kappa}\| u_1\|_{\dot H^1}+\sup_{ \tilde R> R}( \tilde R^{\kappa}\| u_1\|_{L^2_{\tilde R}}+\tilde R^{1+\kappa}\| u_1\|_{\dot H^1_{\tilde R}})\lesssim  \| \pa_t f\|_{W^{'1+\kappa}_R},\\
& \label{nonradforcingeven:orthogonalite} 
u_1=\Pi^\perp_{L^2,R}(u_1),
\end{align}
such that the solution to \eqref{eq:inhomwave} is non-radiative for $|x|> |t|+R$.
\item If $N\equiv 6 \mod 4$ and $f$ is even in time then there exists $u_0 \in \dot H^1\cap \dot H^2$ with
\begin{align}
\nonumber & R^\kappa \| u_0\|_{\dot H^1}+R^{1+\kappa}\| u_0\|_{\dot H^2} +\sup_{ \tilde R> R}( \tilde R^{\kappa}\| u_0\|_{\dot H^1_{\tilde R}}+\tilde R^{1+\kappa} \| \Delta u_0\|_{L^2_{\tilde R}})\\
 \label{nonradforcingeven:energy2}  &\qquad \qquad \qquad \qquad \qquad \qquad \qquad  \lesssim  \| \pa_t f\|_{W_R^{'1+\kappa}}+\sup_{\tilde R>R}\tilde R^{1+\kappa} \| f(0)\|_{L^2_{\tilde R}},\\
 \label{nonradforcingeven:orthogonalite2} &u_0=\Pi^\perp_{\dot H^1,R}(u_0),
\end{align}
such that the solution to \eqref{eq:inhomwave} is non-radiative for $|x|> |t|+R$.
\end{itemize}

\end{lemma}

\begin{proof}

\noindent \textbf{Step 1}. \emph{The case $N\equiv 6\mod 4$}. We let $q= \| \pa_t f\|_{W_R^{'1+\kappa}}+\sup_{\tilde R>R}\tilde R^{1+\kappa} \| f(0)\|_{L^2_{\tilde R}}$ to ease notations. As $f$ is even in time, $\pa_t f$ is odd and we let $v$ be the solution to
\begin{equation}\label{id:nonradforcingeven1000}
\left\{ 
\begin{array}{l l} \pa_t^2 v-\Delta v=\pa_t f,\\
\vec v(0)=(0,v_1)\in \mathcal H,
\end{array}
\right.
\end{equation}
that is non-radiative for $|x|>R+|t|$ given by Corollary \ref{co:nonradiativeforcing}. By Lemma \ref{lem:nonradiativeforcinggain} (ii) there exists $c\in \mathbb R$ such that
\begin{equation}\label{id:nonradforcingeven2}
v_1(x)=\frac{c}{|x|^{N/2+1}}+\tilde v_1(x) \qquad \mbox{with} \qquad \forall \tilde R>R, \quad \| \tilde v_1\|_{L^2_{\tilde R}}\lesssim \tilde R^{-1-\kappa}q.
\end{equation}
Let then $\tilde v$ solve
\begin{equation}\label{id:nonradforcingeven1}
\left\{ 
\begin{array}{l l} \pa_t^2 \tilde v-\Delta \tilde v=\pa_t f,\\
\tilde v(0)=0, \quad \pa_t \tilde v(0,x)=\indic(|x|>R)\tilde v_1(x).
\end{array}
\right.
\end{equation}
By finite speed of propagation and \eqref{construction:id:phim0}, $\tilde v(t,x)=v(t,x)-c\phi_{m_0}(t,x)$ for all $|x|>R+|t|$. Hence $\tilde v $ is non-radiative for $|x|>R+|t|$ since both $v$ and $\phi_{m_0}$ are. 

We now define $w_0$ as follows. We define it first for $r>R$, as the unique solution to
\begin{equation}\label{id:nonradforcingeven8} 
\left\{ \begin{array}{l l}
& -\Delta w_0(r)=f(0,r)-\tilde v_1(r), \quad \forall r>R,\\
& w_0(R)=0
\end{array}
\right.
\end{equation}
that decays as $|x|\to \infty$ (such solution exists, as it is given by the formula \eqref{bd:hardyoutsidetech4} up to rescaling). Using the definition of $q$, \eqref{id:nonradforcingeven2}, one finds that:
\begin{equation}\label{id:nonradforcingeven100} 
\sup_{ \tilde R> R} \tilde R^{1+\kappa} \| \Delta w_0\|_{L^2_{\tilde R}} \lesssim q.
\end{equation}
By \eqref{id:nonradforcingeven8}, Lemma \ref{lem:Hardy} with $\kappa+1$ and \eqref{id:nonradforcingeven100}, we infer $\| \pa_{r}u\|_{L^2_{\tilde R\leq r \leq 2\tilde R}}\lesssim \tilde R^{-\kappa}$ for all $\tilde R\geq R$. After dyadic summation, we deduce:
\begin{equation}\label{id:nonradforcingeven7} 
\sup_{ \tilde R> R} \tilde R^{\kappa}\| w_0\|_{\dot H^1_{\tilde R}}\lesssim q.
\end{equation}
We now denote by $w_0$ an extension to the whole space of $w_0$ in a way such that
\begin{equation}\label{id:nonradforcingeven4} 
R^\kappa \| w_0\|_{\dot H^1}+R^{1+\kappa}\| w_0\|_{\dot H^2} \lesssim q.
\end{equation}
Note that by \eqref{id:nonradforcingeven7} and \eqref{id:nonradforcingeven100} such an extension always exists. We now let $w$ be the solution to
\begin{equation}\label{id:nonradforcingeven45}
\left\{ 
\begin{array}{l l} \pa_t^2 w-\Delta w= f,\\
\vec w(0)=(w_0,0)\in \mathcal H,
\end{array}
\right.
\end{equation}
By \eqref{id:nonradforcingeven4} we have $\vec w(0)\in \mathcal H$, and $f \in L^1L^2$ by assumption, so that by Corollary \ref{cor:inhomogeneousradiation}, $w$ has a radiation profile $G_+\in L^2$ as $t\to \infty$. As $(\pa_t w(0),\pa_{tt}w(0))=(0,\Delta w_0+f(0))\in \mathcal H$ by \eqref{id:nonradforcingeven4} (using $f(0)\in L^2$ as $f,\pa_t f \in L^1L^2$) and $\pa_t f\in L^1L^2$, we have that $\pa_t w$ is the solution to 
$$
\left\{ 
\begin{array}{l l} \pa_t^2 \pa_t w-\Delta \pa_t w= \pa_t f,\\
(\pa_t w(0),\pa_{t}^2 w(0))=(0,\Delta w(0)+f(0))\in \mathcal H.
\end{array}
\right.
$$
By \cite{DuKeMe20}, $\pa_t w$ admits as a radiation profile $-\pa_\rho G_+$ as $t\to \infty$. Due to \eqref{id:nonradforcingeven8} $\pa_{tt}w(0,x)=\tilde v_1(x)$ for $|x|>R$ so that $\pa_t w(t,x)=\tilde v(t,x)$ for all $|x|>R+|t|$ by finite speed of propagation. Since $\tilde v$ is non-radiative for $|x|>R+|t|$, so is $\pa_t w$, and hence $-\pa_\rho G_+(\rho)=0$ for all $\rho \geq R$. As $G_+\in L^2(\mathbb R)$ we deduce that $G_+(\rho)=0$ for all $\rho\geq R$, so that $\lim_{t\to \infty} \int_{r>R+t}|\nabla_{t,x}w(t)|^2dx=0$. Since moreover $w$ is even in time by \eqref{id:nonradforcingeven45} and as $f$ is even in time, then $w$ is non-radiative for $|x|>R+|t|$.

Let us finally define $u_0$. We first define it for $r>R$ as
\begin{equation} \label{id:nonradforcingeven22} 
u_0(r)= \Pi^\perp_{\dot H^1,R}(w_0)(r)  \quad \mbox{if }r> R.
\end{equation}
There exists $(c_{0,0},...,c_{0,l_0})_{0\leq l\leq l_0}\in \mathbb R^{l_0+1}$ such that
\begin{equation} \label{id:nonradforcingeven28} 
\forall |x|>R, \qquad u_0(x)-\sum_{0\leq l \leq l_0} \frac{c_{0,l}}{|x|^{N-2l-2}}=w_0(x)
\end{equation}
The two terms in the left-hand side of \eqref{id:nonradforcingeven28} being orthogonal in $\dot H^1_R$, we deduce by \eqref{id:nonradforcingeven4} and \eqref{bd:energyaFiscbfR2} that
\begin{equation} \label{id:nonradforcingeven25} 
\forall 0\leq l \leq l_0,\qquad |c_{0,l}|\lesssim R^{N/2-2l-1-\kappa}q.
\end{equation}
Injecting \eqref{id:nonradforcingeven25}, \eqref{id:nonradforcingeven7} and \eqref{id:nonradforcingeven100} into \eqref{id:nonradforcingeven28} one finds by a direct estimate, using that for all $0\leq l\leq l_0$ there holds $N/2-2l-1-\kappa\geq N/2-2l_0-1-\kappa=2-\kappa>0$ as $N\equiv 6\mod 4$:
\begin{equation}\label{id:nonradforcingeven26} 
\sup_{ \tilde R> R}( \tilde R^{\kappa}\| u_0\|_{\dot H^1_{\tilde R}}+\tilde R^{1+\kappa} \| \Delta u_0\|_{L^2_{\tilde R}})\lesssim q.
\end{equation}
We now define $u_0$ on the whole space $\mathbb R^N$ as any extension such that 
\begin{equation}\label{id:nonradforcingeven30} 
R^\kappa \| u_0\|_{\dot H^1}+R^{1+\kappa}\| u_0\|_{\dot H^2} \lesssim q,
\end{equation}
which exists by \eqref{id:nonradforcingeven26}. We eventually define $u$ as the solution to \eqref{eq:inhomwave} with data $(u_0,0)$. By \eqref{id:nonradforcingeven28}, \eqref{id:nonradforcingeven45} and finite speed of propagation we have $u=w+\sum_{0}^{l_0} c_{0,l} \phi_{2l}$ for all $|x|>R+|t|$ and hence $u$ is non-radiative for $|x|>R+|t|$ as both $w$ and $\phi_{2l}$ for $0\leq l \leq l_0$ are. The desired estimate and orthogonality \eqref{nonradforcingeven:energy2} and \eqref{nonradforcingeven:orthogonalite2} are consequences of \eqref{id:nonradforcingeven22}, \eqref{id:nonradforcingeven26} and \eqref{id:nonradforcingeven30}.\\

\noindent \textbf{Step 2}. \emph{The case $N\equiv 4\mod 4$.} The proof is very similar. It consists in defining $v$ as the solution to \eqref{id:nonradforcingeven1000} with data $(v_0,0)$ given by Corollary \ref{co:nonradiativeforcing}, then $\tilde v$ as the solution to \eqref{id:nonradforcingeven1} with data $(\tilde v_0,\tilde v_1)=(\indic (r\leq R)v_0(R)+\indic(r>R)v_0(r),0)$, then to directly set $w$ as the solution to \eqref{id:nonradforcingeven45} with data $(w_0,w_1)=(0,\tilde v_1)$ and finally to define $u$ as the solution to \eqref{eq:inhomwave} with data $u_0=0$ and $u_1=\indic (r\leq R)\Pi_{L^2,R}^\perp w_1(R)+\indic (r>R)\Pi^\perp_{L^2,R} w_1(r)$.

All estimates follow exactly as in Step 1, but are slightly easier since the definition of $w_1$ requires no elliptic equation like \eqref{id:nonradforcingeven8}, and since the extension procedures to $r<R$ are simpler. We therefore omit the details.

\end{proof}

We are now ready to give a single proof for Propositions \ref{pr:nonradiativeforcingmainodd} and \ref{pr:nonradiativeforcingmaineven}.

\begin{proof}[Proof of Propositions \ref{pr:nonradiativeforcingmainodd} and \ref{pr:nonradiativeforcingmaineven}]

For $R_0>0$ and a function $g$ we introduce $q(g,R_0)=\| g \|_{W_{R_0}^{'\kappa}}$ for $N$ odd, and $q(g,R_0)=\sup_{\tilde R>R_0+|t|}\tilde R^{1+\kappa} \| g(t)\|_{L^2_{\tilde R}}+\|\pa_t g \|_{W_{R_0}^{'1+\kappa}}$ for $N$ even. Since $\kappa>0$, we have $\| g\|_{W^{'\kappa}_{R_0}}\lesssim \sup_{\tilde R>R_0+|t|}\tilde R^{1+\kappa} \| g(t)\|_{L^2_{\tilde R}}$ by a direct estimate, and hence:
\begin{equation}\label{bd:nonradiativeforcingmainprooftech1}
\| g\|_{W^{'\kappa}_{R_0}}\lesssim q(g,R_0).
\end{equation}

\noindent \textbf{Step 1}. \emph{Existence and decay at initial time}. We claim that there exists $C>0$ such that the following holds true for any $R_0>0$. Assume $g\in L^1L^2$ with $\pa_t g \in L^1L^2$ and $q(g,R_0)<\infty$. Then there exists a solution $w$ to \eqref{eq:inhomwave}
\begin{equation}\label{bd:nonradiativeforcingmainprooftech3}
\left\{ 
\begin{array}{l l} \pa_t^2 w-\Delta w=g,\\
\vec w(0)=(w_0,w_1)\in \mathcal H.
\end{array}
\right.
\end{equation}
with
\begin{align} \label{bd:nonradiativeforcingmainproof1}
& \forall \tilde R\geq R_0, \qquad \| \vec w(0) \|_{\mathcal H_{\tilde R}}\leq C\tilde R^{-\min(\kappa,\kappa_N)} R^{\min(\kappa_N,\kappa)-\kappa} q(g,R_0),\\
\label{id:nonradiativeforcingmainproof2}
& (w_0,w_1)=\Pi^\perp_{\mathcal H,R_0}((w_0,w_1)),\\
\label{id:nonradiativeforcingmainproof200} & \| w_0,w_1\|_{\mathcal H}\lesssim R^{-\kappa}q(g,R_0),
\end{align}
such that $w$ is non-radiative for $|x|>R_0+|t|$. We now prove this claim. If the dimension $N$ is odd, we take $w$ given by Corollary \ref{co:nonradiativeforcing}, and \eqref{bd:nonradiativeforcingmainproof1}, \eqref{id:nonradiativeforcingmainproof2} and \eqref{id:nonradiativeforcingmainproof200} are consequences of \eqref{bd:nonradiativeforcingmainprooftech1}, \eqref{pr:nonradiativeforcinggain1}, \eqref{id:nonradiativeforcingenergy} and \eqref{id:nonradiativeforcingortho}. This proves the claim of Step 1 for odd dimensions $N$.

If the dimension $N$ is even we decompose $w=w_++w_-$, recalling \eqref{notation:id:fpm}, and look for $w_\pm$ solving
\begin{equation}\label{bd:nonradiativeforcingmainprooftech4}
\left\{ 
\begin{array}{l l} \pa_t^2 w_\pm-\Delta w_\pm=g_\pm,\\
\vec w_+(0)=(w_{0},0)\in \mathcal H \mbox{ and }\vec w_-(0)=(0,w_1)\in \mathcal H .
\end{array}
\right. 
\end{equation}
Observe that $g_+ $ is even with time, while $g_-$ is odd in time, and by \eqref{bd:nonradiativeforcingmainprooftech1} that:
\begin{equation}\label{bd:nonradiativeforcingmainprooftech2}
\sum_\pm (\| g_\pm\|_{W_{R_0}^{'\kappa}}+ \| \pa_t g_\pm\|_{W_{R_0}^{'\kappa}}+\sup_{\tilde R>R_0+|t|}\| g_\pm(t)\|_{L^2_{\tilde R}} ) \lesssim q(g,R_0)
\end{equation}
If $N\equiv 4\mod 4$, we take $w_+$ given by Corollary \ref{co:nonradiativeforcing}, and $w_-$ given by Lemma \ref{lem:nonradforcingeven}. Then the desired estimate and orthogonality \eqref{bd:nonradiativeforcingmainproof1}, \eqref{id:nonradiativeforcingmainproof2} and \eqref{id:nonradiativeforcingmainproof200} are consequences of \eqref{pr:nonradiativeforcinggain1}, \eqref{id:nonradiativeforcingenergy}, \eqref{id:nonradiativeforcingortho} and \eqref{bd:nonradiativeforcingmainprooftech2} for $w_0$, and of \eqref{nonradforcingeven:energy}, \eqref{nonradforcingeven:orthogonalite} and $\kappa>\kappa_N$ for $w_1$.

If $N\equiv 6\mod 4$, we take $w_-$ given by Corollary \ref{co:nonradiativeforcing}, and $w_+$ given by Lemma \ref{lem:nonradforcingeven}. Then the desired estimate and orthogonality \eqref{bd:nonradiativeforcingmainproof1} and \eqref{id:nonradiativeforcingmainproof2} are consequences of \eqref{pr:nonradiativeforcinggain1}, \eqref{id:nonradiativeforcingenergy}, \eqref{id:nonradiativeforcingortho} and \eqref{bd:nonradiativeforcingmainprooftech2} for $w_1$, and of \eqref{nonradforcingeven:energy2}, \eqref{nonradforcingeven:orthogonalite2} and $\kappa>\kappa_N$ for $w_0$.

In both cases $N\equiv 4\mod4$ and $N\equiv 6\mod 4$ we have that $w=w_++w_-$ solves \eqref{bd:nonradiativeforcingmainprooftech3} because of \eqref{bd:nonradiativeforcingmainprooftech4} and of $g=g_++g_-$, and that $w$ is indeed non-radiative for $|x|>R_0+|t|$ as both $w_+$ and $w_-$ are. This proves the claim of Step 1 for even dimensions.\\

\noindent \textbf{Step 2}. \emph{Decay in the exterior cone and end of the proof}. We define $u$ as the solution $u=w$ to \eqref{eq:inhomwave} provided by Step 1 with $R_0=R$ and forcing $g=f$. The desired orthogonalities \eqref{id:nonradiativeforcingortho2odd} and \eqref{id:nonradiativeforcingortho2even}, as well as the desired estimate for $\| \vec u(0)\|_{\mathcal H}$ in \eqref{bd:nonradiativeforcingmain1odd} and \eqref{bd:nonradiativeforcingmain1even}, follow as immediate consequences of \eqref{id:nonradiativeforcingmainproof2} and \eqref{id:nonradiativeforcingmainproof200}. The desired estimates for $\| u_1\|_{\dot H^1}$ if $N\equiv 4 \mod 4$ or $\| u_0\|_{\dot H^2}$ if $N\equiv 6\mod 4$ in \eqref{bd:nonradiativeforcingmain1even} follow from \eqref{nonradforcingeven:energy} and \eqref{nonradforcingeven:energy2}.

There only remains to prove the estimate for $\| u\|_{W^{\kappa_N}_R}$ in \eqref{bd:nonradiativeforcingmain1odd} and \eqref{bd:nonradiativeforcingmain1even}. To that aim, we fix $\tilde t \in \mathbb R$ and $\tilde R>0$ such that $\tilde R>R+|\tilde t|$.

If $\tilde R\geq R+2|\tilde t|$, applying energy estimates and finite speed of propagation to \eqref{eq:inhomwave}, using \eqref{bd:nonradiativeforcingmainproof1}, we obtain:
\begin{align}
\nonumber \| \vec u(\tilde t)\|_{\mathcal H_{\tilde R}} & \lesssim \| \vec u(0)\|_{\mathcal H_{\tilde R-|\tilde t|}}+\| \indic (|x|\geq \tilde R-|\tilde t|+|t|) f\|_{L^1L^2} \\
\nonumber & \lesssim R^{\min(\kappa,\kappa_N)-\kappa} (\tilde R-|\tilde t|)^{-\min(\kappa,\kappa_N)} q(f,R)+\tilde R^{-\kappa} \| f\|_{W^{'\kappa}_R}\\
\label{bd:nonradiativeforcingmainproof3} &\lesssim  \tilde R^{-\min(\kappa,\kappa_N)}R^{\min(\kappa,\kappa_N)-\kappa} q(f,R)
\end{align}
where we used that in this case $ \tilde R \leq 2(\tilde R-|\tilde t|)$ and $\tilde R\geq R$.

If now $R+|\tilde t|\leq \tilde R\leq R+2|\tilde t|$, we let $v$ be the solution to
$$
\left\{ 
\begin{array}{l l} \pa_t^2 v-\Delta v= f(\tilde t+t),\\
\vec v(0)=(v_0,v_1)\in \mathcal H.
\end{array}
\right.
$$
given by Step 1 with $g(\cdot)=f(\tilde t+\cdot)$ and $R_0=\tilde R$. Then $v$ is non-radiative for $|x|>\tilde R+|t|$ and satisfies
\begin{equation}\label{bd:nonradiativeforcingmainproof2}
\| \vec v(0) \|_{\mathcal H_{\tilde R}}\leq C \tilde R^{-\kappa}  q(f(\tilde t+\cdot),\tilde R)\leq C \tilde R^{-\min(\kappa,\kappa_N)} R^{\min(\kappa,\kappa_N)-\kappa} q(f,R).
\end{equation}
Consider $\bar u(t)=u(\tilde t+t)-v(t)$. Then $\bar u$ solves $\pa_t^2\bar u-\Delta \bar u=0$, and is non-radiative for $|x|>\tilde R+|t|$. By Classification \ref{pr:nonradiativefree}, there exists $\cbf [\tilde t,\tilde R]\in \mathbb R^{m_0+1}$ such that
\begin{equation}\label{bd:nonradiativeforcingmainproof4}
\vec a_F[\cbf [\tilde t,\tilde R]](0)=\vec u(\tilde t)-\vec v(0).
\end{equation}
Injecting \eqref{bd:energyaFiscbfR2}, \eqref{bd:nonradiativeforcingmainproof3} and \eqref{bd:nonradiativeforcingmainproof2} in \eqref{bd:nonradiativeforcingmainproof4} we get
\begin{align}
\nonumber  | \cbf [\tilde t,\tilde R]|_{R+2|\tilde t|}&\approx \| \vec a_F[\cbf [\tilde t,\tilde R]](0)\|_{\mathcal H_{R+2|\tilde t|}} \\
\nonumber &\lesssim  (R+2|\tilde t|)^{-\min(\kappa,\kappa_N)}R^{\min(\kappa,\kappa_N)-\kappa} q(f,R)+ \tilde R^{-\min(\kappa,\kappa_N)} R^{\min(\kappa,\kappa_N)-\kappa} q(f,R)\\
 \label{bd:nonradiativeforcingmainproof100}&\lesssim \tilde R^{-\min(\kappa,\kappa_N)}R^{\min(\kappa,\kappa_N)-\kappa} q(f,R)
\end{align}
where we used $\tilde R\approx (R+2|\tilde t|)$. Using \eqref{bd:energyaFiscbfR2} again, with \eqref{bd:nonradiativeforcingmainproof100} and $\tilde R\approx (R+2|\tilde t|)$ so that $| \cbf [\tilde t,\tilde R]|_{\tilde R}\approx | \cbf [\tilde t,\tilde R]|_{R+2|\tilde t|}$ we get:
\begin{equation}\label{bd:nonradiativeforcingmainproof101}
\| \vec a_F[\cbf [\tilde t,\tilde R]](0)\|_{\mathcal H_{\tilde R}}\lesssim \tilde R^{-\min(\kappa,\kappa_N)}R^{\min(\kappa,\kappa_N)-\kappa} q(f,R)
\end{equation}
Reinjecting \eqref{bd:nonradiativeforcingmainproof2} and \eqref{bd:nonradiativeforcingmainproof101} in \eqref{bd:nonradiativeforcingmainproof4} shows:
\begin{equation} \label{bd:nonradiativeforcingmainproof300}
 \| \vec u(\tilde t)\|_{\mathcal H_{\tilde R}}\lesssim  \tilde R^{-\min(\kappa,\kappa_N)} R^{\min(\kappa,\kappa_N)-\kappa} q(f,R).
\end{equation}
The inequalities \eqref{bd:nonradiativeforcingmainproof3} and \eqref{bd:nonradiativeforcingmainproof300} imply $\| u\|_{W^{\min(\kappa,\kappa_n)}_R}\lesssim R^{\min(\kappa,\kappa_N)-\kappa}q(f,R)$ ending the proof of \eqref{bd:nonradiativeforcingmain1odd} and \eqref{bd:nonradiativeforcingmain1even}.

\end{proof}

As a by-product of the strategy of the proof of Propositions \ref{pr:nonradiativeforcingmainodd} and \ref{pr:nonradiativeforcingmaineven}, we can refine the channels of energy estimate of Proposition \ref{pr:channels}:

\begin{lemma}[Weighted channels of energy]

Assume $N\geq 4$ is even. For any $0\leq \kappa<1$, there exists $C>0$ such that the following holds true for all $R>0$. Assume $u$ is a solution to \eqref{eq:inhomwave} with $f\in L^1L^2 \cap W^{'\kappa}_R$ that is non-radiative for $|x|>R+|t|$. Assume moreover that $u$ and $f$ are even in time if $N\equiv 4\mod 4$ and odd in time if $N\equiv 6\mod 4$. Then:
\begin{equation} \label{bd:weightedchannels}
\| u\|_{\tilde W^\kappa_R}\leq \left\{ \begin{array}{l l} C \| f\|_{W^{'\kappa}_R}+R^\kappa \| \Pi_{\dot H^1,R}u_0\|_{\dot H^1_R} \qquad \mbox{if }N\equiv 4 \mod 4,\\
C \| f\|_{W^{'\kappa}_R}+R^\kappa \| \Pi_{L^2,R} u_1\|_{L^2_R} \qquad \mbox{if }N\equiv 6 \mod 4.
\end{array} \right.
\end{equation}

\end{lemma}

\begin{proof}

The proof of the Lemma follows the same strategy as that above of Propositions \ref{pr:nonradiativeforcingmainodd} and \ref{pr:nonradiativeforcingmaineven}. It is actually much simpler since by definition the norm $\tilde W^\kappa_R$ only concerns the half part of $\vec u$ for which Corollary \ref{co:nonradiativeforcing} can be applied.

The first step is to show that if $f\in W^{'\kappa}_R$, then $u$ be given by Corollary \ref{co:nonradiativeforcing} enjoys:
\begin{equation}\label{construction:bd:weightedchannelstech}
\forall \tilde R>R, \quad \| (\tilde u_0,\tilde u_1)\|_{\mathcal H_R}\lesssim R^{-\kappa} \| f\|_{W^{'\kappa}_R}.
\end{equation}
The proof of \eqref{construction:bd:weightedchannelstech} is exactly as that of Lemma \ref{lem:nonradiativeforcinggain}, but simpler since $0<\kappa<\kappa_N=1$. Then, one proceeds exactly as in the proof of Propositions \ref{pr:nonradiativeforcingmainodd} and \ref{pr:nonradiativeforcingmaineven}, but using only Corollary \ref{co:nonradiativeforcing} and \eqref{construction:bd:weightedchannelstech}. We omit the details.

\end{proof}

\subsection{Construction of small nonlinear non-radiative solutions}

In this section we construct non-radiative solutions to
\begin{equation} \label{eq:nonlinearwave2}
\left\{ 
\begin{array}{l l} \pa_t^2 u-\Delta u=\varphi(u),\\
\vec u(0)=(u_0,u_1)\in \mathcal H
\end{array}
\right.
\end{equation}
with small energy in the exterior cone $|x|>1+|t|$ as perturbations of the non-radiative free waves $a_F$ given by \eqref{id:aF}. In \eqref{eq:nonlinearwave2} we write $\varphi(u)=\varphi(x,u)$ to ease notations, where $\varphi$ is given by \eqref{eq:nonlinearityanalytic} or \eqref{eq:nonlinearitypower}. We recall $\delta=1$ (resp. $\delta=\frac{N-2}{4}$) if $\varphi $ is of the form \eqref{eq:nonlinearityanalytic} (resp. of the form \eqref{eq:nonlinearitypower}).

\begin{proposition} \label{pr:constructionnonradia}

Let $\varphi$ be given by either \eqref{eq:nonlinearityanalytic}, or \eqref{eq:nonlinearitypower} if $N$ is odd or $N=4,6$. There exists $\epsilon>0$ such that, for any $\cbf \in \mathbb R^{m_0+1}$ with $|\cbf|\leq \epsilon$, there exists $a[\cbf]=a_F[\cbf]+\tilde a[\cbf]$ a solution to \eqref{eq:nonlinearwave2} that is non-radiative for $|x|>1+|t|$ and such that
\begin{align} 
\label{bd:tildeaWkappaN}
& \| ( a (0),\pa_t a(0))\|_{\mathcal H}\lesssim |\cbf|,\quad \| \tilde a \|_{W^{\kappa_N}_1}\lesssim |\cbf|^{1+\delta}.\\
\label{id:Pia0} & \Pi_{\mathcal H,1} \vec a (0)=\vec a_F[\cbf](0),
\end{align}
Moreover, the map $\cbf \mapsto (\tilde a(0),\pa_t \tilde a(0),\pa_r \pa_t \tilde a(0))$ if $N\equiv 4\mod 4$ (resp. the map $\cbf \mapsto (\tilde a(0),\pa_t \tilde a(0),\pa_{tt} \tilde a(0))$ if $N\equiv 6\mod 4$) is differentiable from a neighbourhood of the origin in $\mathbb R^{m_0+1}$ to $\dot H^1_1\times L^2_1\times L^2_1$ and the norm of its differential is $\lesssim |\cbf|^\delta$.

\end{proposition}

\begin{remark}

In view of \eqref{add_decay_aF} and \eqref{bd:tildeaWkappaN}, we have $\| a[\cbf]\|_{W^{\kappa_N}_1}\lesssim |\cbf|$, implying the gain of decay $|a |\lesssim |\cbf| |x|^{1-N/2-\kappa_N}$ for $|x|>1+|t|$ by \eqref{bd:SobolevW}, in comparison with the scaling invariant decay $|x|^{1-N/2}$.

\end{remark}

\begin{definition}
\label{def:constructionnonradia}
Proposition \ref{pr:constructionnonradia} allows us to construct small non-radiative solutions for $|x|>R+|t|$ for any $R>0$ by a scaling argument. Namely, for all $R>0$ and $\cbf \in \mathbb R^{m_0+1}$ such that $|\cbf|_{R}\leq \epsilon$, we define for $|x|>R+|t|$:
\begin{equation}\label{id:defageneral}
a[\cbf,R](t,r)=(a[\cbf_R])_{(R)}(t,r), \quad \tilde{a}[\cbf,R](t,r)=(\tilde{a}[\cbf_R])_{(R)}(t,r).
\end{equation}
where $\cbf_R=(R^{-\frac N2+1+m}c_m)_{0\leq m <\frac{N}{2}-1}$ and  $a[\cbf_R]$, $\tilde{a}[\cbf_R]$ are given by Proposition \ref{pr:constructionnonradia}. Note that $|\cbf_R|=|\cbf|_R$ and that $a_F[\cbf]=(a_F[\cbf_R])_{(R)}$.

\end{definition}

\begin{proof}

For $\eta>0$ to be fixed later, we define the quantities for $f\in L^1L^2$ with $\pa_t f\in L^1L^2$ and $b \in C(\mathbb R,\mathcal H)$:
\begin{align*}
& \| f\|_{Y_1}= \| f\|_{L^1L^2}+\| \pa_t f\|_{L^1L^2}, \qquad \| f\|_{Y_2}=\|\pa_t f\|_{W^{'1+\kappa_N+\eta}_1}+\sup_{t\in \mathbb R}\sup_{R>1+|t|} R^{1+\kappa_N+\eta}\| f(t)\|_{L^2_R},\\
& \| b \|_{X_1}= \| b\|_{L^\infty \mathcal H}, \qquad \| b \|_{X_2}=\| b \|_{W^{\kappa_N}_1}+\left| \begin{array}{l l} \| \pa_t b(0)\|_{\dot H^1} \ \mbox{ if }N\equiv 4\mod 4,\\ \| b(0)\|_{\dot H^2} \ \mbox{ if }N\equiv 6\mod 4, \end{array}\right.
\end{align*}
and let $X$ and $Y$ be the Banach spaces associated to the norms $\| b\|_X=\| b\|_{X_1}+\| b\|_{X_2}$ and $\| f\|_Y=\| f\|_{Y_1}+\| f\|_{Y_2}$ respectively, requiring in addition that $\vec b \in C(\mathbb R,\mathcal H)$ for $b\in X$. The proof relies on the Banach fixed point argument in $X$.\\

\noindent \textbf{Step 1}. \emph{Formulation of the fixed point problem}. Let $\chi$ be a smooth one-dimensional cut-off with $\chi(\rho)=1$ for $\rho>1$ and $\chi(\rho)=0$ for $\rho<\frac 12$. We define $a^0$ as the solution to $\pa_t^2a^0-\Delta a^0=0$ with initial data
\begin{align}
\label{construction:id:defa00} & (a^0_0,a^0_1)(r)=\chi (r)\vec a_F[\cbf](0,r)
\end{align}
Notice that by finite speed of propagation
\begin{equation}\label{construction:id:a0isaF}
\forall |x|>1+|t|, \qquad a^0(t,x)=a_F[\cbf](t,x).
\end{equation}
Using this and \eqref{add_decay_aF}, and energy conservation, we get
\begin{equation}\label{bd:constructionnonradia1}
\| a^0 \|_{X}\lesssim |\cbf| .
\end{equation}
We consider the mapping $\Phi$ defined as follows. For $\tilde b\in X$, we let $(b_0,b_1)$ be the initial data of the solution $\bar b$ to
\begin{equation}\label{id:constructionbarb}
\left\{
\begin{array}{l l}
\pa_t^2 \bar b-\Delta \bar b= f, \\
(\bar b (0),\pa_t \bar b (0))=(b_0,b_1)\in \mathcal H,
\end{array} \right. \qquad f=f[\cbf,\tilde b]= \chi (|x|-|t|) \varphi (a^0+\tilde b)
\end{equation}
that is non-radiative for $|x|>1+|t|$ given by Proposition \ref{pr:nonradiativeforcingmainodd} if $N$ is odd or Proposition \ref{pr:nonradiativeforcingmaineven} if $N$ is even, with $R=1$. We relegate to Step 2 the proof of all necessary estimates to ensure the applicability of Propositions \ref{pr:nonradiativeforcingmainodd} and \ref{pr:nonradiativeforcingmaineven} to show that $\Phi$ is well-defined in Step 3. We then let $b$ be the solution to
\begin{equation} \label{id:constructionb}
\left\{
\begin{array}{l l}
\pa_t^2 b-\Delta  b= \indic(|x|>1+|t|) f(\tilde b), \\
\vec b(0)=(b_0,b_1).
\end{array} \right.
\end{equation}
As $\bar b=b$ for $|x|>1+|t|$ by finite speed of propagation, $b$ is non-radiative for $|x|>1+|t|$ since $\bar b$ is.

Let us comment that the auxiliary function $\bar b$ is solely used to ensure the differentiability of the forcing term in its equation in order to be able to apply Proposition \ref{pr:nonradiativeforcingmaineven}. We then let $\Phi(\tilde b)=b$.\\

\noindent \textbf{Step 2}. \emph{Bounds on $f$}. Let $B(K|\cbf|^{1+\delta})$ denote the ball of $X$ of radius $K|\cbf|^{1+\delta}$ centered at the origin. We claim that for any $K>1$, if $|\cbf|$ is small enough, then for all $\tilde b,\tilde b' \in B(K|\cbf|^{1+\delta})$ the following holds true. First, if $\varphi$ is of the form \eqref{eq:nonlinearityanalytic} or \eqref{eq:nonlinearitypower} then:
\begin{align}
\label{bd:constructionf1} & \sup_{R>1+|t|} R^{1+\kappa_N+\eta}\| f[\tilde b](t)\|_{L^2_R}\leq C |\cbf|^{1+\delta},\\
\label{bd:constructionf2} & \sup_{R>1+|t|} R^{1+\kappa_N+\eta}\| (f[\tilde b]-f[\tilde b'])(t)\|_{L^2_R}\leq C |\cbf|^{\delta}\| \tilde b-\tilde b'\|_{X},\\
\label{bd:constructionf3} & \| \pa_t f[\tilde b]\|_{W^{'1+\kappa_N+\eta}_1}\leq C |\cbf|^{1+\delta},
\end{align}
where $\eta=\min (\kappa_N,\frac{4}{N-2}\kappa_N )$. Second, if $\varphi$ is of the form \eqref{eq:nonlinearityanalytic} or of the form \eqref{eq:nonlinearitypower} and $3\leq N\leq 6$, then
\begin{equation}\label{bd:constructionf4} 
\| \pa_t f[\tilde b]-\pa_t f[\tilde b']\|_{W^{'1+\kappa_N+\eta}_1}\leq C |\cbf| \| \tilde b-\tilde b'\|_{X}.
\end{equation}
Third, if $\varphi$ is of the form \eqref{eq:nonlinearityanalytic} or \eqref{eq:nonlinearitypower}:
\begin{equation}\label{bd:constructionf5} 
\| f[\tilde b]\|_{L^1L^2}+\| \pa_t f[\tilde b]\|_{L^1L^2}<\infty.
\end{equation}

We now prove all the above estimates. First, by \eqref{bd:constructionnonradia1} and \eqref{bd:SobolevW}:
\begin{equation}\label{bd:existencetechnical1}
\forall |x|\geq 1+|t|,\qquad |a^0+\tilde b| \lesssim \frac{|\cbf|+K|\cbf|^{1+\delta}}{|x|^{\frac{N}{2}-1+\kappa_N}}\leq \frac{C |\cbf |}{|x|^{\frac{N}{2}-1+\kappa_N}}
\end{equation}
for some universal $C>0$, for any $K>0$, if $|\cbf|$ has been chosen small enough.

We first assume that $\varphi $ is of power-type form \eqref{eq:nonlinearitypower}. Then by \eqref{bd:existencetechnical1}, for $|x|>1+|t|$:
$$
 | f[\tilde b]|\leq \left(\frac{C |\cbf |}{|x|^{\frac{N}{2}-1+\kappa_N}} \right)^{\frac{N+2}{N-2}}\leq \frac{C |\cbf|^{1+\delta}}{|x|^{\frac N2 +1 +\kappa_N+\eta}}.
$$
Next, using $|\varphi (u)-\varphi (v)|\lesssim (|u|+|v|)^{\frac{4}{N-2}}|u-v|$ and \eqref{bd:existencetechnical1} we get that for $|x|>1+|t|$:
$$
| f[\tilde b]-f[\tilde b']|\leq \left(\frac{C |\cbf |}{|x|^{\frac{N}{2}-1+\kappa_N}} \right)^{\frac{4}{N-2}}|\tilde b'-\tilde b|\leq \frac{C |\cbf|^{\delta}}{|x|^{2+\eta}}|\tilde b-\tilde b'|.
$$
The two inequalities above imply \eqref{bd:constructionf1} and \eqref{bd:constructionf2} after a direct computation. For $|x|>1+|t|$, using $\pa_t \varphi(u)=\frac{N+2}{N-2} |u|^{\frac{4}{N-2}}u_t$ and that $\chi(|x|-|t|)=1$ we get:
$$
| \pa_t f[\tilde b]|\leq  \left(\frac{C |\cbf |}{|x|^{\frac{N}{2}-1+\kappa_N}} \right)^{\frac{4}{N-2}} (|\pa_t a^0|+|\pa_t \tilde b|)\leq \frac{C |\cbf|^{\delta}}{|x|^{2+\eta}}(| \pa_t a^0|+|\pa_t \tilde b|).
$$
We deduce by an explicit computation that for $R\geq 1$:
\begin{align}
\nonumber \| \pa_t f[\tilde b] \|_{L^1L^2(|x|\geq R+|t|)}& \lesssim |\cbf|^\delta \int_0^\infty \frac{dt}{(R+t)^{2+\eta}}(\| \pa_t a^0(t)\|_{L^2_{R+t}}+\| \pa_t \tilde b(t)\|_{L^2_{R+t}}) \\
\nonumber & \lesssim |\cbf|^\delta \int_0^\infty \frac{dt}{(R+t)^{2+\eta+\kappa_N}}(\| a^0\|_{X}+\| \tilde b\|_{X}) \\
 \label{bd:existencetechnical2}& \lesssim |\cbf|^{1+\delta} \frac{1}{R^{1+\kappa_N+\eta}}
\end{align}
which proves \eqref{bd:constructionf3}.

Second, assume $\varphi$ is of the form \eqref{eq:nonlinearityanalytic}. Then, using \eqref{bd:existencetechnical1}, for all $|x|>1+|t|$:
$$
| f[\tilde b]|\leq \sum_{k\geq 2} \frac{\tau^{-k}C^k|\cbf|^k}{|x|^{\frac{N}{2}+1+k\kappa_N}}\leq  \frac{C |\cbf|^{1+\delta}}{|x|^{\frac N2 +1 +2\kappa_N}}
$$
(the series converging if $\epsilon$ is small enough), and similarly
$$
| f[\tilde b]-f[\tilde b']|\leq \sum_{k\geq 2} \frac{k\tau^{-k}C^{k-1}|\cbf|^{k-1}}{|x|^{2+(k-1)\kappa_N}}|\tilde b-\tilde b'|\leq  C |\cbf |\frac{|\tilde b-\tilde b'|}{|x|^{2+\kappa_N}}.
$$
These two inequalities and a direct computation imply \eqref{bd:constructionf1} and \eqref{bd:constructionf2}. Using the identity $\pa_t \varphi[a^0+\tilde b]-\pa_t \varphi[a^0+\tilde b']=F_1+F_2$ where
\begin{align*}
F_1&=\sum_{k\geq 2} k\varphi_k |x|^{(k-1)(\frac N2-1)}(a^0+\tilde b)^{k-1}(\pa_t \tilde b-\pa_t \tilde b'), \\
F_2&=\sum_{k\geq 2} k\varphi_k |x|^{(k-1)(\frac N2-1)}((a^0+\tilde b)^{k-1}-(a^0+\tilde b')^{k-1})(\pa_t a^0+\pa_t \tilde b),
\end{align*}
and performing a computation very similar to \eqref{bd:existencetechnical2}, one obtains the bound \eqref{bd:constructionf3}.

We have thus proved \eqref{bd:constructionf1}, \eqref{bd:constructionf2} and \eqref{bd:constructionf3} in both cases $\varphi$ of the form \eqref{eq:nonlinearityanalytic} and $\varphi$ of the form \eqref{eq:nonlinearitypower}. The proof of \eqref{bd:constructionf4} is so similar to that of \eqref{bd:constructionf2} and \eqref{bd:constructionf3} that we safely omit it.

Third, let us show \eqref{bd:constructionf5}. By Sobolev and \eqref{bd:constructionnonradia1}, we have for $\frac 12+|t|\leq |x|\leq 1+|t|$ that
\begin{equation}\label{bd:existencetechnical3}
|a^0|+|\tilde b|\lesssim \frac{|\cbf|}{(1+|t|)^{\frac N2-1}}.
\end{equation}
Hence, if $\varphi$ is either of the form \eqref{eq:nonlinearityanalytic} or \eqref{eq:nonlinearitypower} then for $\frac 12+|t|\leq |x|\leq 1+|t|$ one has $|f[\tilde b]|\lesssim (1+|t|)^{-\frac N2-1}$ and thus by a direct estimate:
$$
\| \indic (\frac 12+|t|<|x|<1+|t|)f[\tilde b]\|_{L^1L^2} \lesssim \int_0^\infty \frac{dt}{(1+|t|)^2}<\infty.
$$
combining the above inequality with \eqref{bd:constructionf1} shows the first inequality in \eqref{bd:constructionf5}. The second inequality in \eqref{bd:constructionf5} can be proved using \eqref{bd:existencetechnical3} and performing a computation similar to \eqref{bd:existencetechnical2}, we omit the details.\\

\noindent \textbf{Step 3}. \emph{Well definition and contraction property of $\Phi$}. We claim that there exists $K>0$ such that for $\epsilon$ small enough, $\Phi$ is well-defined and is a contraction on the set $B(K|\cbf|^{1+\delta})=\{\tilde b\in X, \ \| \tilde b\|_{X}\leq K|\cbf|^{1+\delta}\}$. 

Indeed, by \eqref{bd:constructionf1}, \eqref{bd:constructionf3} and \eqref{bd:constructionf5} we deduce $f\in Y$ with $\| f\|_{Y_2}\lesssim |\cbf|^{1+\delta}$. We can apply Proposition \ref{pr:nonradiativeforcingmainodd} if $N$ is odd (using $\| g\|_{W_1^{'\kappa_N+\eta}}\lesssim \sup_{R>1+|t|}R^{1+\kappa_N+\eta} \| g(t)\|_{L^2_R}\leq \| g\|_{Y_2}$ which follows from a direct computation) and Proposition \ref{pr:nonradiativeforcingmaineven} if $N$ is even to the equation \eqref{id:constructionbarb} and obtain that $\bar b$, $b_0$ and $b_1$ are well-defined with:
\begin{equation}\label{bd:constructionfixed1}
\| (b_0,b_1) \|_{\mathcal H} +\| \bar b \|_{W_1^{\kappa_N}}\leq C|\cbf|^{1+\delta}
\end{equation}
Then, we can solve for $b$ in Equation \eqref{id:constructionb} using \eqref{bd:constructionfixed1}, \eqref{bd:constructionf1} and standard energy estimates and obtain:
\begin{equation}\label{bd:constructionfixed2}
\| \vec b\|_{L^\infty \mathcal H}\leq C|\cbf|^{1+\delta}.
\end{equation}
Furthermore, since $\bar b=b$ for $|x|>1+|t|$ we have from \eqref{bd:constructionfixed1} that
\begin{equation}\label{bd:constructionfixed3}
\| b \|_{W_1^{\kappa_N}}\leq C|\cbf|^{1+\delta}.
\end{equation}
Combining \eqref{bd:constructionfixed2} and \eqref{bd:constructionfixed3} shows that $b\in X$ with $\| b\|_{X}\leq C|\cbf|^{1+\delta}$. Thus $b\in B(K|\cbf|^{1+\delta})$ if $K$ has been chosen large enough. This shows that $\Phi$ maps $B(K|\cbf|^{1+\delta})$ onto itself. The proof that $\Phi$ is a contraction is similar, using \eqref{bd:constructionf2} and \eqref{bd:constructionf4}.\\

\noindent \textbf{Step 4}. \emph{End of the proof}. By Step 3, $\Phi$ admits a fixed point $b\in B(K|\cbf|^{1+\delta})$. That is, by \eqref{id:constructionb} $b$ solves
\begin{equation}\label{id:constructionb2}
\left\{
\begin{array}{l l}
\pa_t^2 b-\Delta b= \indic (|x|>1+|t|)\varphi (a^0+b), \\
\vec b (0)=(b_0,b_1)\in \mathcal H,
\end{array} \right.
\end{equation}
is non-radiative for $|x|>1+|t|$, and satisfies by definition of $X$:
\begin{equation}\label{bd:constructionfixed4}
\| b\|_{L^\infty \mathcal H}+\| b\|_{W^{\kappa_N}_1}\lesssim |\cbf|^{1+\delta}.
\end{equation}
In addition, by \eqref{id:nonradiativeforcingortho2odd} and \eqref{id:nonradiativeforcingortho2even}:
\begin{equation}\label{bd:constructionfixed5}
\Pi_{\mathcal H_1} \vec b(0)= (0,0).
\end{equation}
Consider now $\bar a=a^0+b$. Then it solves using \eqref{id:constructionb2}:
$$
\left\{
\begin{array}{l l}
\pa_t^2 \bar a-\Delta \bar a= \indic (|x|>1+|t|) \varphi (\bar a), \\
(\bar a(0),\pa_t \bar a(0))=(a^0_0+b_0,a^0_1+b_1),
\end{array} \right.
$$
We eventually let $a$ be the solution to \eqref{eq:nonlinearwave2} with initial datum $(\bar a(0),\pa_t \bar a (0))$. Then by finite speed of propagation, we have $a(t,x)=\tilde a (t,x)=a^0+b$ for all $|x|>1+|t|$. The desired estimates and orthogonality \eqref{id:Pia0} and \eqref{bd:tildeaWkappaN} for $a$ are then direct consequences of \eqref{construction:id:a0isaF}, \eqref{construction:id:defa00}, \eqref{construction:id:defa00}, \eqref{bd:constructionfixed5} and \eqref{bd:constructionfixed4}.

Finally, let us show the differentiability property at initial time in even dimensions. We have the identity $\Phi=\Phi_3\circ (\Phi_2\circ \Phi_1,\Phi_1)$ where :
$$
\begin{array}{l l l l}
\Phi_1 :& \mathbb R^{m_0+1} \times X & \rightarrow & Y  \\
& (\cbf,\tilde b) & \mapsto & f 
\end{array}
\qquad \begin{array}{l l l l l l l l}
\Phi_2 :& Y &  \rightarrow & X  \\
& f &\mapsto & \bar b
\end{array} 
\qquad 
\begin{array}{l l l l l l l l}
\Phi_3 :& X \times Y & \rightarrow & X \\
 &(\bar b,f)&\mapsto &b
\end{array}
$$
where $f$, $\bar b$ and $b$ are defined in Step 1. Then, for $N$ even, by computations that are very similar to the ones establishing \eqref{bd:constructionf2}, \eqref{bd:constructionf4} and \eqref{bd:constructionf5}, the map $\Phi_1$ is differentiable from $\mathbb R^{m_0+1}\times X$ into $Y$, and at a point $(\cbf,\tilde b)$ with $\| \tilde b\|_{X}\lesssim |\cbf|\ll 1$ its differential $J\Phi_1[\cbf,\tilde b]$ enjoys the estimate $\| J\Phi_1[\cbf,\tilde b](\Delta \cbf,\Delta \tilde b)\|_{Y_2}\lesssim |\cbf|^\delta \| (\Delta \cbf,\Delta \tilde b) \|_{\mathbb R^{m_0+1}\times X}$. The linear mappings $\Phi_2$ and $\Phi_3$ are continuous by Proposition \ref{pr:nonradiativeforcingmaineven} and standard energy estimates, with in particular the estimates $\| \Phi_2 (f)\|_{X_2}\lesssim \| f\|_{Y_2}$ and $\| \Phi_3(\bar b,f)\|_{X_1}\lesssim \| \bar b\|_{X_2}+\| f\|_{Y_2}$; hence both are differentiable.

Combining, and using that $\bar b=b$ for $|x|>1+|t|$, we deduce that $\Phi$ is differentiable from $\mathbb R^{m_0+1}\times X$ towards $X$ with $J\Phi=\Phi_3\circ (\Phi_2 \circ J\Phi_1,J\Phi_1)$ and that for $\| \tilde b\|_X\lesssim |\cbf|\ll 1$ there holds the estimate 
\begin{align*}
\| J\Phi[\cbf,\tilde b](\Delta \cbf,\Delta \tilde b)\|_{X} & \lesssim \| \Phi_2 \circ J \Phi_1 [\cbf,\tilde b](\Delta \cbf,\Delta \tilde b) \|_{X_2}+ \| J \Phi_1 [\cbf,\tilde b](\Delta \cbf,\Delta \tilde b) \|_{Y_2}\\
&\lesssim  |\cbf|^{1+\delta} \| (\Delta \cbf,\Delta \tilde b) \|_{\mathbb R^{m_0+1}\times X}.
\end{align*}
By a standard argument, this implies that the fixed point map $\cbf\mapsto b$ is differentiable from a neighbourghood of the origin into $X$, and that its differential has norm $\lesssim |\cbf|^\delta$. Hence the desired differentiability properties of the map $\cbf \mapsto \tilde a[\cbf]$.

\end{proof}

\section{Construction of an approximate nonlinear resonance in even dimensions} \label{sec:resonance}

We assume throughout this section that $N$ is even. In this section we first define properly the resonance \eqref{id:defphiN/2-1}, and then show that quadratic nonlinearities are non-resonant, a cancellation that will be crucial in Section \ref{sec:uniquenesseven}, for the proof of uniqueness.

We recall from \eqref{id:freenonrad1} and \eqref{id:freenonrad2} the identity:
\begin{equation}\label{id:BoxNresonance}
(\pa_t^2-\Delta) \left(\frac{1}{r^{N/2-1}}f(\frac{t}{r})\right)=-\frac{1}{r^{N/2+1}}\left(\mathcal L_{N} f(\frac{t}{r}\right)
\end{equation}
where, writing $\mathcal L_N=\mathcal L_{N,N/2-1}$ to ease notations,
$$
\mathcal L_{N}=-(1-\sigma^2)\pa_\sigma^2+\sigma\pa_\sigma -(\frac N2-1)^2.
$$

\begin{lemma}[Linear resonance] \label{lem:resonancelinear}

Let $N\geq 4 $ be even. Then if $N\equiv 4 \mod 4$ there exists a unique odd polynomial $p_{2l_1+3}$ of degree $2l_1+3=\frac N2-1$ with $\pa_\sigma p_{2l_1+3}(0)=1$, and if $N\equiv 6 \mod 4$ there exists a unique even polynomial  $p_{2l_0+2}$ of degree $2l_0+2=\frac N2-1$ with $p_{2l_0+2}(0)=1$, such that the following holds true. There hold $\mathcal L_N p_{2l_1+3}=0$ and $\mathcal L_N p_{2l_0+2}=0$. The function
\begin{equation}\label{id:defphim0+1}
\phi_{m_0+1} (t,x)=\left\{ \begin{array}{l l} \frac{1}{r^{N/2-1}} p_{2l_1+3}\left(\frac{t}{r} \right) \qquad \mbox{if }N\equiv 4 \mod 4, \\  \frac{1}{r^{N/2-1}} p_{2l_0+2}\left(\frac{t}{r} \right) \qquad \mbox{if }N\equiv 6 \mod 4, \end{array} \right.
\end{equation}
satisfies $\pa_t^2 \phi_{m_0+1} -\Delta \phi_{m_0+1}=0$ for all $t\in \mathbb R$ and $x\neq 0$, and
$$
\vec{\phi}_{m_0+1}(0,x)= \left\{ \begin{array}{l l} \left( 0,\frac{1}{r^{N/2 }}\right)  \qquad \mbox{if }N\equiv 4 \mod 4, \\ \left( \frac{1}{r^{N/2-1}},0\right)  \qquad \mbox{if }N\equiv 6 \mod 4. \end{array} \right.
$$

\end{lemma}

\begin{proof}

The equation $\mathcal L_{N}p=0$ is $-(1-\sigma^2)\pa_\sigma^2p+\sigma\pa_\sigma p -(\frac N2-1)^2p=0$ which is Chebyshev's equation. It admits as a particular solution the Chebyshev polynomials of the first kind $T_{\frac{N}{2}-1}$ which we rename in this article as $p_{N/2 -1}$ (see e.g. \cite{olver2010nist}), and renormalise so that $p_{N/2 -1}(0)=1$ if $\frac N2-1 $ is even, and $\pa_\sigma p_{N/2-1}(0)=1$ if $\frac N2-1$ is odd. Its uniqueness follows from standard arguments.

\end{proof}

The linear resonance $\phi_{m_0+1}$ critically fails to belong to the energy space. In our proof of the classification of non-radiative solutions in Section \ref{sec:uniquenesseven}, we will use localised projections on this resonance, in order to obtain dispersive estimates. The argument will require to add corrections to $\phi_{m_0+1}$ to make it a better approximate solution of the nonlinear problem \eqref{eq:nonlinearwaveintro}. Of the most importance are the quadratic terms. For an analytic nonlinearity, it is $r^{\frac{N-6}{2}}\phi_{m_0+1}^2$. The function $\phi_{m_0+1}^2$ is even in time, while $\phi_{m_0+1}$ is odd in time if $N\equiv 4 \mod 4$ and even in time if $N\equiv 6 \mod 4$. Thus, if $N\equiv 4 \mod 4$ quadratic effects are non-resonant due to this difference in time symmetry. Somewhat surprisingly, they are also non-resonant if $N\equiv 6 \mod 4$ due to an orthogonality property of Chebyshev polynomials:

\begin{proposition}[Quadratic correction for analytic nonlinearities] \label{pr:resonancequadra}

Assume $N\geq 4$ is even and let $\varphi_2\in \mathbb R$. There exists for $N\equiv 4 \mod 4$ a unique even polynomial $\tilde p=\tilde p[N,\varphi_2]$ of degree $N-2$, with $\mathcal L_N\tilde p=-\varphi_2 p_{2l_1+3}^2$, and for $N\equiv 6 \mod 4$, a unique polynomial $\tilde{p}=\tilde{p}[N,\varphi_2]$ with $\mathcal L_N\tilde p=-\varphi_2 p_{2l_0+2}^2$ and $\tilde p(0)=0$ . The function
\begin{equation}\label{rigidity|u|u:id:defphi2}
\psi (t,x)=\frac{1}{r^{N/2-1}} \tilde p\left(\frac{t}{r} \right)
\end{equation}
satisfies for all $t\in \mathbb R$ and $x\neq 0$ :
\begin{equation}\label{rigidity|u|u:id:eqA2}
\pa_t^2 \psi -\Delta \psi= r^{\frac{N-6}{2}}\varphi_2 \phi_{m_0+1}^2.
\end{equation}

\end{proposition}

The proof is done below. A power-type nonlinearity \eqref{eq:nonlinearitypower} is quadratic in the $N=6$ dimensional case. Then $|\phi_{m_0+1}|\phi_{m_0+1}$ is no longer orthogonal to $\phi_{m_0+1}$. However, its projection enjoys a crucial sign property, and we are still able to construct the quadratic correction.

\begin{proposition}[Quadratic correction for power nonlinearity]  \label{rigidity|u|u:pr:A}

Assume $N=6$ (so $m_0+1=2$). There exists a function $\tilde p\in C^2([-1,1])$ (which is not a polynomial)\footnote{We keep the same notation of $\tilde p$ and $\psi$ as in Proposition \ref{pr:resonancequadra} to be able to treat simultaneously power and analytic nonlinearities later on.} which is even in $\sigma$ and satisfies $\tilde p(0)=0$, and $\beta^*>0$, so that for any\footnote{Notice that in comparison with Proposition \ref{pr:resonancequadra}, $\psi$ is no longer invariant under the scaling transformation $f\mapsto f_{(\lambda)}$ hence the introduction of a scale parameter $R$.} $R>0$, the function $\psi=\psi[R]$ given by
\begin{equation}\label{rigidity|u|u:id:defphi}
\psi(t,x)=\beta^*  \log (\frac rR) \phi_{2}(t,x)+\frac{1}{r^{2}}\tilde p(\frac{t}{r})
\end{equation}
satisfies for all $t\in \mathbb R$ and $x\neq 0$ :
\begin{equation}\label{rigidity|u|u:id:eqA}
\pa_t^2 \psi -\Delta \psi= | \phi_{2}|\phi_{2} .
\end{equation}

\end{proposition}

The proof is done below. To treat both nonlinearities simultaneously, we let:
\begin{equation}\label{def:beta}
\beta=\left\{\begin{array}{l l l} \beta^* & \mbox{ for }\varphi \mbox{ given by }\eqref{eq:nonlinearitypower}\mbox{ and }N=6,\\
0 & \mbox{ for all other cases}.
\end{array} \right.
\end{equation}
Let $\tilde c\in \mathbb R$ and $R>0$. We define\footnote{We slightly abuse notations, as the superscript $2$ in $\hat c^2$ does not denote a square.} $\hat c^2=\hat c^2[\tilde c]$ by
\begin{equation}\label{def:hatc}
\hat c^2=\left\{\begin{array}{l l l} |\tilde c|\tilde c & \mbox{ for }\eqref{eq:nonlinearitypower}\mbox{ and }N=6,\\
\tilde c^2 & \mbox{ for all other cases}.
\end{array} \right.
\end{equation}
We define the nonlinearly corrected resonance $\tilde \phi=\tilde \phi[\tilde c,R]$ by:
\begin{equation}\label{def:tildephi}
\tilde \phi=    \tilde c\phi_{m_0+1}+\hat{c}^2 \psi
\end{equation}
where $\psi$ is given by Proposition \ref{pr:resonancequadra} for $\varphi$ given by \eqref{eq:nonlinearityanalytic}, or is given by Proposition \ref{rigidity|u|u:pr:A} for $\varphi$ given by \eqref{eq:nonlinearitypower} and $N=6$, or $\psi=0$ by convention for \eqref{eq:nonlinearitypower} and $N=4$, respectively. To sum up, we have the condensed notation:
\begin{align} \label{even:id:tidephigeneralised} 
\tilde \phi (t,x) & =\frac{1}{r^{N/2-1}} \left( \left(\tilde c+\beta \hat c^2 \ln \frac rR \right)p_{\frac N2-1}(\frac{t}{r})+\hat c^2 \tilde p (\frac{t}{r})\right) ,\\
\label{even:id:tidephigeneralised2}  \psi (t,x)&=\frac{1}{r^{N/2-1}}\left( \beta \ln \frac rR \ p_{\frac N2-1} (\frac{t}{r})+ \tilde p (\frac{t}{r})\right),
\end{align}
so that 
$\tilde \phi(t,x)=\frac{\tilde c}{r^{\frac N2 -1}} p_{\frac{N}{2}-1}\left( \frac{t}{r} \right)+\hat{c}^2 \psi(t,x)$.
We have that $\tilde \phi$ satisfies for $t\in \mathbb R$ and $x\neq 0$:
\begin{equation}\label{even:id:boxtildephi}
\pa_t^2\tilde \phi-\Delta \tilde \phi-\varphi(  \tilde \phi)=
 \left\{ \begin{array}{l l l} \displaystyle \frac{\varphi_2 }{r^{N/2+1}}\left(-2\tilde c^3 p_{N/2-1}\tilde p-\tilde c^4 \tilde p^2\right)(\frac{t}{r})-\varphi_c(\tilde \phi) & \mbox{for } \eqref{eq:nonlinearityanalytic},\\
\displaystyle |\tilde c \phi_{2}|\tilde c\phi_{2}-|\tilde \phi|\tilde \phi  & \mbox{for }\eqref{eq:nonlinearitypower} \mbox{ and }N=6,\\
\displaystyle \frac{1}{r^{3}} \tilde c^3 |p_{1}|^2p_{1}(\frac{t}{r}) & \mbox{for }\eqref{eq:nonlinearitypower} \mbox{ and }N=4.
\end{array} \right.
\end{equation}
where $\varphi_c(u)=\sum_{k\geq 3}^\infty \varphi_k |x|^{(k-1)(\frac N2-1)-2}u^k$ for \eqref{eq:nonlinearityanalytic}. Note that all right-hand sides in \eqref{even:id:boxtildephi} are cubic in $\tilde c$. In addition, there holds:
\begin{equation}\label{eq:vectildephi0}
\left(\tilde \phi(0,x),\partial_t \tilde \phi(0,x)\right)= \left\{ \begin{array}{l l l} \displaystyle \left( \frac{ \tilde c^2 \tilde p(0)}{r^{N/2-1 }} \ , \ \frac{\tilde c}{r^{N/2}}\right) & \mbox{for }N\equiv 4\mod 4, \\
\displaystyle \left( \frac{ \tilde c}{r^{N/2 -1}} \ , \ 0\right)  &\mbox{for }N\equiv 6\mod 4  \mbox{ and } \eqref{eq:nonlinearityanalytic}\\
 \displaystyle  \left( \frac{ \tilde c+\beta| \tilde c|\tilde c \ln \frac rR }{r^{2}} \ , \ 0\right)  &\mbox{for }N= 6  \mbox{ and } \eqref{eq:nonlinearitypower}
 \end{array} \right.
\end{equation}
where by convention $\tilde p(0)=0$ for $\varphi$ given by \eqref{eq:nonlinearitypower} and $N=4$.

\begin{proof}[Proof of Proposition \ref{pr:resonancequadra}]

For any even $N$, by Lemma \ref{lem:resonancelinear}, $p_{N/2-1}^2$ ($=p_{2l_1+3}^2$ or $=p_{2l_0+2}^2$ if $N\equiv 4\mod 4$ or $N\equiv 6 \mod 4$ respectively) is an even polynomial of degree $N-2$. We compute for any $\alpha\in \mathbb N$ that $\mathcal L_{N}(\sigma^{2\alpha})=\mu_{N,\alpha}\sigma^{2\alpha}-2\alpha (2\alpha-1)\sigma^{2\alpha-2}$ with $\mu_{N,\alpha}=4\alpha^2-(\frac N2-1)^2$.

First, assume that $N\equiv 4+4n$ for some integer $n\geq 0$. Then $\mu_{N,\alpha}=4\alpha^2-(2n+1)^2\neq 0$ for all $\alpha\in \mathbb N$. The operator $\mathcal L_N$ then defines a linear isomorphism on the set of even polynomials of degree $N-2=2+4n$. There then exists a unique even polynomial $\tilde p$ of degree $N-2$ such that $\mathcal L_N\tilde p=-p_{2l_1+3}^2$ and the result of the Lemma follows.

Second, assume that $N=6+4n$ for some integer $n\geq 0$ in which case
$$
\mathcal L_N=(1-\sigma^2)^{\frac 12}\pa_\sigma ((1-\sigma^2)^{\frac 12}\pa_\sigma )-(2+4n)^2.
$$
As standard properties of Chebyshev polynomials (see \cite{olver2010nist}), we have that the operator $\mathcal L_N$ is self-adjoint for the scalar product $\langle f,g\rangle=\int_0^1 (1-\sigma^2)^{-\frac 12}f(\sigma)g(\sigma)d\sigma$ with compact resolvent, that its spectrum is $\{\alpha^2-(2+4n)^2, \ \alpha \in \mathbb N\}$, and that the associated eigenfunctions $T_\alpha$ are the Chebyshev polynomials. Therefore, there exists at least one polynomial $\bar p$ of degree $N-2$ that solves $\mathcal L_N \bar p=-p_{2l_0+2}^2$ if and only if
\begin{equation}\label{id:resonancequadra1}
\int_{-1}^1 (1-\sigma^2)^{-\frac 12} T_{\alpha}^3(\sigma)d\sigma =0
\end{equation}
for $\alpha=2+4n$, and then all other solutions are of the form $\tilde p=\bar p+c p_{2l_0+2}$ for $c\in \mathbb R$. We will show that \eqref{id:resonancequadra1} holds true for all $\alpha\geq 1$. Indeed, changing variables, setting $s=\arccos(\sigma)$ there holds $\frac{ds}{d\sigma}=-\frac{1}{(1+\sigma^2)^{\frac 12}}$ and $T_{\alpha}(\sigma)=T_{\alpha}(\cos s)=\cos (\alpha s)$  and hence
$$
\int_{-1}^1 (1-\sigma^2)^{-\frac 12} T_{\alpha}^3(\sigma)d\sigma =\int_{0}^\pi \cos^3(\alpha s)ds=0.
$$
This shows the existence of $\bar p$, and choosing $c=-\bar p(0)$ ensures by Lemma \ref{lem:resonancelinear} that this is the only solution such that $\tilde p(0)=0$.

\end{proof}

The proof of Proposition \ref{rigidity|u|u:pr:A}, relies on an explicit formula for solving $\mathcal L_Nu=f$ (only stated our case of concern $N=6$) given by the following Lemma. Notice that for $N=6$ the polynomial involved in the definition of the resonance of Lemma \ref{lem:resonancelinear} is $p_2=1-2\sigma^2$.

\begin{lemma}[Inverting $\mathcal L_N$ for $N=6$] \label{rigidity|u|u:lem:mathcalL}
Assume $N=6$. There exists a solution $\tilde p_2$ (which is not a polynomial) to $\mathcal L_6 \tilde p_2=0$ with $\tilde p_2(0)=0$ and $\pa_\sigma \tilde p_2(0)=1$. This function is smooth on $(-1,1)$, is odd in $\sigma$, satisfies the Wronskian relation:
\begin{equation}\label{rigidity|u|u:id:wronskian}
\pa_\sigma \tilde p_2p_2-\pa_\sigma p_2 \tilde p_2=-\frac 12 (1-\sigma^2)^{-\frac 12},
\end{equation}
and admits the asymptotic behaviour close to $1$ for some $c_1,c_2\in \mathbb R$ as $\sigma \uparrow 1$:
\begin{align} 
& \label{rigidity|u|u:id:proptildep1}\tilde p_2(\sigma)=c_1-\sqrt 2(1-\sigma)^{\frac 12}+O((1-\sigma)),\\
& \label{rigidity|u|u:id:proptildep2} \pa_\sigma \tilde p_2(\sigma)= \frac{1}{ (2-2\sigma)^{\frac 12}}+c_2 +O((1-\sigma)^{\frac12}),\\
& \label{rigidity|u|u:id:proptildep3} \pa_\sigma^2 \tilde p_2(\sigma)= \frac{1}{(2-2\sigma)^{\frac 32}}+ \frac{4\sqrt 2}{(1-\sigma)^{\frac12}}+O(1).
\end{align}
Given any continuous and even function $f$ on $(-1,1)$, there is a unique $C^2(-1,1)$ solution $u$ to $\mathcal L_6 u=f$ on $(-1,1)$ that is even in $\sigma$ and such that $u(0)=0$. It is given by the formula:
\begin{equation}\label{rigidity|u|u:id:inversionL}
u(\sigma)=p_2(\sigma)\int_0^\sigma \tilde p_2(\tilde \sigma)(1-\tilde \sigma^2)^{-\frac 12}f(\tilde \sigma)d\tilde \sigma-\tilde p_2(\sigma)\int_0^\sigma p_2(\tilde \sigma)(1-\tilde \sigma^2)^{-\frac 12}f(\tilde \sigma)d\tilde \sigma.
\end{equation}

\end{lemma}

\begin{proof}[Proof of Lemma \ref{rigidity|u|u:lem:mathcalL}]
We have $\mathcal L_{6}=-(1-\sigma^2)\pa_\sigma^2+\sigma\pa_\sigma -4$. Let $\tilde p_2$ be a solution to $\mathcal L_6\tilde p_2=0$ with $\tilde p_2(0)=0$ and $\pa_\sigma \tilde p_2(0)=1$, and define the Wronskian $W=\pa_\sigma \tilde p_2 p_2-\pa_\sigma p_2\tilde p_2$. As $p_2(0)=1$ by Lemma \ref{lem:resonancelinear}, we have $W(0)=\pa_\sigma \tilde p_2(0) p_2(0)=1$. Moreover, it satisfies the ODE:
$$
\pa_\sigma W=\frac{\sigma}{1-\sigma^2}W
$$
so that $W= (1-\sigma^2)^{-\frac 12}$. For $\sigma_0>2^{-\frac 12}$, reintegrating the Wronskian relation we have that there exists $c\in \mathbb R$ such that:
\begin{equation}\label{rigidity|u|u:id:astildep1}
\tilde p_2(\sigma)=c p_2(\sigma)+ p_2(\sigma)\int_{\sigma_0}^{\sigma} (1-\tilde \sigma^2)^{-\frac 12}p_2^{-2}(\tilde \sigma)d\tilde \sigma
\end{equation}
which, using $p_2=1-2\sigma^2$, gives the desired asymptotic behaviour \eqref{rigidity|u|u:id:proptildep1}, \eqref{rigidity|u|u:id:proptildep2} and \eqref{rigidity|u|u:id:proptildep3}. The second part of the lemma then follows as a standard consequence from linear second order ODEs theory.

\end{proof}

\begin{proof}[Proof of Proposition \ref{rigidity|u|u:pr:A}]

Recall $m_0+1=2$ for $N=6$. Using the identity \eqref{id:BoxNresonance} and the fact that $\Box \phi_{2}=0$ we obtain:
$$
\Box\left( \beta \log r \phi_{2}\right)(t,r)=\frac{2\beta}{r^4}(\sigma \pa_\sigma p_2)(\frac tr), \qquad \Box\left(\frac{1}{r^2}\tilde p(\frac tr)\right)=\frac{1}{r^4}(-\mathcal L_6 \tilde p)(\frac tr),
$$
so that \eqref{rigidity|u|u:id:eqA} is satisfied if and only if:
\begin{equation}\label{rigidity|u|u:id:intertildeA}
\mathcal L_6 \tilde p=2\beta \sigma \pa_\sigma p_2-|p_2|p_2=:f.
\end{equation}
We fix $\beta$ by requiring the orthogonality condition $\int_{0}^{1} f(\sigma)p_2(\sigma)(1-\sigma^2)^{-1/2}d\sigma=0$. Since $p_2(\sigma)=1-2\sigma^2$, we have $\sigma\pa_\sigma p_2=-4\sigma^2=2p_2-2$. Using $\int_{0}^1 p_2(\sigma) (1-\sigma^2)^{-1/2}d\sigma=0$, we infer:
$$
\int_{0}^1  \sigma \pa_\sigma p_2(\sigma) p_2(\sigma)(1-\sigma^2)^{-1/2}d\sigma= 2 \int_{0}^1  (p_2(\sigma))^2 (1-\sigma^2)^{-1/2}d\sigma>0.
$$
We thus choose:
$$
\beta=\beta^*=\frac{\int_{0}^1 |p_2(\sigma)|^3(1-\sigma^2)^{-1/2}d\sigma}{2 \int_{0}^1  \sigma \pa_\sigma p_2(\sigma)p_2(\sigma)(1-\sigma^2)^{-1/2}d\sigma}>0.
$$
Solving Equation \eqref{rigidity|u|u:id:intertildeA} using Lemma \ref{rigidity|u|u:lem:mathcalL} and $\int_{0}^1 f(\sigma)p_2(\sigma)(1-\sigma^2)^{-1/2}=0$, we obtain $\tilde p=\bar p+cp_2$ for
$$
\bar p(\sigma)=p_2(\sigma)\int_0^\sigma \tilde p_2(\tilde \sigma)(1-\tilde \sigma^2)^{-\frac 12}f(\tilde \sigma)d\tilde \sigma+\tilde p_2(\sigma)\int_\sigma^1 p_2(\tilde \sigma)(1-\tilde \sigma^2)^{-\frac 12}f(\tilde \sigma)d\tilde \sigma,
$$
and $c\in \mathbb R$. The above identity defines a $C^2$ function on $(-1,1)$. It remains to check the differentiability at $\sigma=1$ of $\bar p$ (implying that at $\sigma=-1$ by even symmetry). Differentiating twice, using \eqref{rigidity|u|u:id:wronskian}, \eqref{rigidity|u|u:id:proptildep1}, \eqref{rigidity|u|u:id:proptildep2} and \eqref{rigidity|u|u:id:proptildep3}:
\begin{align*}
\pa_{\sigma}^2 \bar p(\sigma) & =\pa_\sigma^2 p_2(\sigma)\int_0^\sigma \tilde p_2(1-\tilde \sigma^2)^{-\frac 12}fd\tilde \sigma+\pa_\sigma^2 \tilde p_2(\sigma)\int_\sigma^1 p_2(1-\tilde \sigma^2)^{-\frac 12}f(\tilde \sigma)d\tilde \sigma -(1-\sigma^2)^{-1}f(\sigma), \\
&=C+O((1-\sigma)^{1/2})\\
& +\left(\frac{4\sqrt 2}{(1-\sigma)^{\frac 12}}+\frac{1}{(2-2\sigma)^{\frac 32}}+O(1)\right)\left(  (1-\sigma)^{\frac 12}f(1)+C'(1-\sigma)^{\frac 32}+O(1-\sigma)^{\frac 52}\right)\\
&-\frac{f(1)}{2(1-\sigma)}+C''+O((1-\sigma))\\
&=C'''+O(1-\sigma)^{1/2}
\end{align*}
where $C,C',C'',C'''$ denote constants depending on $f(1)$ and $f'(1)$. Hence $\pa_{\sigma}^2\bar p$ is continuous up to $\sigma=1$, so that $\bar p$ is $C^2$ up to $\sigma=1$. We finally choose $c\in \mathbb R$ so that $\tilde p(0)=0$, which ends the proof of the Proposition.
\end{proof}

\section{Uniqueness of non-radiative solutions in odd space-dimension}
In this section we prove the uniqueness part of Theorem \ref{th:main} when $N\geq 3$ is odd. We divide the proof into two cases. In the first case, when $N\in\{3,5\}$ with nonlineariry $|u|^{\frac{4}{N-2}}u$ or when $N\geq 5$ with analytic nonlinearity, the straightforward proof relies on the linear channels of energy estimates, Corollary \ref{co:nonradiativeforcing} and Strichartz inequalities (see Subsection \ref{sub:uniqueness1}). In this case we prove indeed a stronger uniqueness result, which is not restricted to nonradiative solutions. In the second case, treated in Subsection \ref{sub:uniqueness2}, when $N\geq 7$ and the nonlinearity is $|u|^{\frac{4}{N-2}}u$, we need an extra ingredient, the a priori decay of nonradiative solutions in odd dimensions, proved in \cite{DuKeMe21a} and recalled in the preliminary Subsection \ref{sub:Strichartz}.
\subsection{Uniqueness for general solutions}
\label{sub:uniqueness1}
In this Subsection we assume:
\begin{equation}
 \label{H1}
\text{$N\geq 5$ is odd and $\varphi$ of the form \eqref{eq:nonlinearityanalytic}, or $N\in \{3,5\}$ and $\varphi$ is of the form \eqref{eq:nonlinearitypower}. }
\end{equation}

\begin{proposition}
\label{pr:uniqueness1}
Assume \eqref{H1}. Then  
there exists $\ep>0$ such that, for any $R\geq 0$ and any radial solutions $u$ and $v$ of \eqref{eq:nonlinearwave2} for $|x|> R+|t|$ such that 
\begin{equation}
\label{smalluv}
\left\|\vec{u}(0)\right\|_{\Hc_R}+\left\|\vec{v}(0)\right\|_{\Hc_R}\leq \ep,\quad \Pi_{\Hc_R}\vec{u}(0)=\Pi_{\Hc_R}\vec{v}(0) 
\end{equation} 
and 
\begin{equation}
\label{u-v}
\sum_{\pm}\lim_{t\to\pm\infty}\int_{|x|>R+|t|} |\nabla_{t,x}(u(t,x)-v(t,x))|^2dx=0,
\end{equation} 
one has 
$$ u(t,x)=v(t,x),\quad |x|>R+|t|.$$
\end{proposition}
We first check that Proposition \ref{pr:uniqueness1} implies the uniqueness part of Theorem \ref{th:main} when \eqref{H1} holds. We thus assume \eqref{H1} and let $u$ be a nonradiative solution such that $\int |\nabla_{t,x}u(0)|^2dx\leq {\ep'}^2$, where $\ep'$ such that $0<\ep'\leq \ep$ is to be specified. We choose $\cbf \in \Rb^{\lfloor \frac{N-1}{2} \rfloor}$ such that 
$$ \Pi_{\Hc_R}\vec{a}[\cbf](0)=\Pi_{\Hc_R}\vec{u}(0),$$
and we denote to lighten notation $a=a[\cbf]$. Then by the ``existence'' part of Theorem \ref{th:main}, taking $\ep'$ small enough, we have 
$$ \int |\nabla_{t,x}a(0)|^2dx\leq {\ep}^2.$$
By Proposition \ref{pr:uniqueness1}, we obtain immediately $u=a$, which concludes the proof of the uniqueness part of Theorem \ref{th:main} when (H1) holds.
\begin{proof}[Proof of Proposition \ref{pr:uniqueness1}]
We first consider the case where $N\in \{3,5\}$, and $\varphi(u)=|u|^{\frac{4}{N-2}}u$. Let $w=u-v$. Then
\begin{equation}
 \label{eqw}
\partial_t^2w-\Delta w=|u|^{\frac{4}{N-2}}u-|v|^{\frac{4}{N-2}}v.
 \end{equation} 
 By Strichartz estimate, finite speed of propagation, the smallness condition in \eqref{smalluv} and standard small data theory,
 \begin{equation} 
  \label{StrichartzBound}
  \left\|\indic(|x|>R+|t|) u\right\|_{L^{\frac{N+2}{N-2}}L^{\frac{2(N+2)}{N-2}}}+\left\|\indic(|x|>R+|t|) v\right\|_{L^{\frac{N+2}{N-2}}L^{\frac{2(N+2)}{N-2}}}\lesssim \ep.
 \end{equation} 
 Noting that
 \begin{equation}
  \label{bound_NL}
  \left||u|^{\frac{4}{N-2}}u-|v|^{\frac{4}{N-2}}v\right|\leq C|u-v|\left( |u|^{\frac{4}{N-2}}+|v|^{\frac{4}{N-2}} \right),
 \end{equation} 
we obtain, by equation \eqref{eqw}, finite speed of propagation, Strichartz and H\"older estimates.
\begin{multline*}
 \big\|\indic(|x|>R+|t|)w\big\|_{L^{\frac{N+2}{N-2}}L^{\frac{2(N+2)}{N-2}}}\lesssim \left\|\indic(|x|>R+|t|) w \left( |u|^{\frac{4}{N-2}}+|v|^{\frac{4}{N-2}}\right)\right\|_{L^1L^2}+\|\vec{w}(0)\|_{\Hc_R}\\
 \lesssim \ep^{\frac{4}{N-2}}\left\|\indic(|x|>R+|t|) w\right\|_{L^{\frac{N+2}{N-2}}L^{\frac{2(N+2)}{N-2}}}+\|\vec{w}(0)\|_{\Hc_R},
\end{multline*}
and hence, taking $\ep$ small enough,
\begin{equation}
\label{Strichartz_w}
\big\|\indic(|x|>R+|t|)w\big\|_{L^{\frac{N+2}{N-2}}L^{\frac{2(N+2)}{N-2}}}\lesssim  \|\vec{w}(0)\|_{\Hc_R}.
\end{equation} 
Since $w$ satisfies equation \eqref{eqw},  we obtain by Proposition \ref{pr:nonradiativeforcing} and the second assumption in \eqref{smalluv} and assumption \eqref{u-v}
\begin{multline*}
 \|\vec{w}(0)\|_{\Hc_R}\lesssim\left\|\indic(|x|>R+|t|) w \left( |u|^{\frac{4}{N-2}}+|v|^{\frac{4}{N-2}}\right)\right\|_{L^1L^2}\\
 \lesssim \ep^{\frac{4}{N-2}}\left\|\indic(|x|>R+|t|) w\right\|_{L^{\frac{N+2}{N-2}}L^{\frac{2(N+2)}{N-2}}}\lesssim \ep^{\frac{4}{N-2}}\|\vec{w}(0)\|_{\Hc_R}.
\end{multline*}
Taking $\ep$ small enough, we obtain $\|\vec{w}(0)\|_{\Hc_R}=0$, concluding the proof in the case where $N\in \{3,5\}$ and the nonlinearity is power-like.

The proof is almost the same in the case where $N\geq 5$ and $\varphi$ is an analytic nonlinearity of the form \eqref{eq:nonlinearityanalytic}. Indeed, by standard small data theory and the smallness assumption in \eqref{smalluv}, we obtain 
$$ \forall t\in \Rb,\quad \|u\|_{\Hc_{R+|t|}}+\|v\|_{\Hc_{R+|t|}}\lesssim \ep.$$
Using \eqref{NLLipsch2} and the weighted Strichartz estimate of Lemma \ref{lem:weighted_Strichartz}, the end of the proof goes among the same lines as the proof in the case where $\varphi$ is power-like. We omit the details.
\end{proof}

\subsection{Uniqueness of non-radiative solutions for power nonlinearity in high dimensions}
\label{sub:uniqueness2}
In this subsection we prove the uniqueness part of Theorem \ref{th:main}, in the case where $N\geq 7$ and $\varphi(u)=|u|^{\frac{4}{N-2}}u$. Unlike the proofs of the two other subsections, we need an additional decay estimate, proved in \cite{DuKeMe19Pb}.

Let $u$ be a non-radiative solution, and $C_1$, $R_1$ be given by Proposition \ref{pr:aprioribound}, so that $u(t,r)$ is well defined for $r\geq t+R_1$ and 
\begin{equation}
\label{decayu}
\forall R\geq R_1+|t|,\quad \|\vec{u}(t)\|_{\Hc_R}\leq \frac{C_1}{R^{1/2}},\quad |u(t,R)|\leq \frac{C_1}{R^{\frac{N-1}{2}}}. 
\end{equation} 
Taking a larger $R_1$ if necessary, we can assume 
\begin{equation}
 \label{largeR1}
 \frac{C_1}{R_1^{1/2}}\leq \ep,
\end{equation}
Since $\|\vec{u}(0)\|_{\Hc_R}\leq \ep$, taking $\ep$ small enough, we can find $\cbf \in \Rb^{\lfloor \frac{N-1}{2}\rfloor}$ such that 
$$ \Pi_{\Hc_{R_1}}\vec{a}[\cbf](0)=\Pi_{\Hc_{R_1}}\vec{u}(0).$$
We will denote $a=a[\cbf]$ to ligthen notation. Since by \eqref{decayu}, \eqref{largeR1}, $|\cbf|_{R_1}\approx \left\|\pi_{\Hc_{R_1}}\vec{a}(0)\right\|_{\Hc_{R_1}}\leq \ep$, we have $\|\vec{a}(0)\|_{\Hc_{R_1}}\lesssim \ep$. Also, by the inequality \eqref{bound_a_aF},
\begin{equation}
 \label{decaya1}
 \forall t,\; \forall R\geq R_1+|t|,\quad \|\vec{a}(t)-\vec{a}_F(t)\|_{\Hc_R}\lesssim \ep^{1+\delta}\frac{R_1^{1/2}}{R^{1/2}},
\end{equation}
where $a_F$ is the solution of the free wave equation with initial data $\vec{a}(0)$. By the decay estimate \eqref{add_decay_aF}, we see that for $R\geq R_1+|t|$, $\left\|a_F(t,R)\right\|_{\Hc_R}\lesssim \ep\left(\frac{R_1}{R}\right)^{1/2}$. Combining with \eqref{decaya1}, radial Sobolev, \eqref{decayu} and \eqref{largeR1} we obtain te following uniform decay bound for $a$ and $u$
\begin{equation}
 \label{unif_bound_ua}
 \forall r\geq R_1+|t|,\quad |a(t,r)|+|u(t,r)|\lesssim \ep \frac{R_1^{1/2}}{R^{\frac{N-1}{2}}}. 
\end{equation} 
We next consider $w=u-a$, which satisfies 
\begin{equation}
 \label{eqwbis}
 \partial_t^2w-\Delta w=|u|^{\frac{4}{N-2}}u-|a|^{\frac{4}{N-2}}a.
\end{equation} 
By the endpoint Strichartz estimate and finite speed of propagation,
\begin{equation}
 \label{StrichartzwNlarge}
\left\|\indic(|x|>R_1+|t|) w\right\|_{L^2L^{\frac{2N}{N-3}}}\lesssim \|\vec{w}(0)\|_{\Hc_{R_1}}+\left\|\indic(|x|>R_1+|t|)\left( |u|^{\frac{4}{N-2}} u-|a|^{\frac{4}{N-2}}a\right)\right\|_{L^1L^2}.
 \end{equation} 
By H\"older's inequality and the uniform bound \eqref{unif_bound_ua}
\begin{multline*}
 \left\|\indic(|x|>R_1+|t|)\left( |u|^{\frac{4}{N-2}} u-|a|^{\frac{4}{N-2}}a\right)\right\|_{L^1L^2}\\
 \lesssim \left(\ep R_1^{1/2}\right)^{\frac{4}{N-2}} \|\indic(|x|>R_1+|t|) w\|_{L^2L^{\frac{2N}{N-3}}}\left\| \indic(|x|>R_1+|t|)r^{-2(N-1)/(N-2)}\right\|_{L^2L^{\frac {2N}3}}.
\end{multline*} 
By direct computation, $\left\| \indic(|x|>R_1+|t|)r^{-2(N-1)/(N-2)}\right\|_{L^2L^{\frac{2N}3}}\approx R_1^{-2/(N-2)}$, and thus 
\begin{equation}
\label{boundNLw} 
\left\|\indic(|x|>R_1+|t|)\left( |u|^{\frac{4}{N-2}} u-|a|^{\frac{4}{N-2}}a\right)\right\|_{L^1L^2}\lesssim \ep^{\frac{4}{N-2}} 
\left\|\indic(|x|>R_1+|t|) w\right\|_{L^2L^{\frac{2N}{N-3}}}.
\end{equation} 
Going back to \eqref{StrichartzwNlarge}, we deduce, taking $\ep$ small enough
\begin{equation}
\label{boundStrichartzwNlarge}
\left\|\indic(|x|>R_1+|t|) w\right\|_{L^2L^{\frac{2N}{N-3}}}\lesssim \|\vec{w}(0)\|_{\Hc_{R_1}}. 
\end{equation} 
The end of the proof is the same as in the other cases, using Proposition \ref{pr:nonradiativeforcing} together with \eqref{boundNLw} and \eqref{boundStrichartzwNlarge} to prove that $w=0$ for $|x|>R_1+|t|$. By Proposition 3.8 of \cite{DuKeMe19Pb}, we obtain that $u=a$ for $|x|>R+|t|$, for any $R$ so that $u$ and $a$ are defined on this set.

\section{Uniqueness of non-radiative solutions in even dimensions} \label{sec:uniquenesseven}

In this section we prove:

\begin{proposition}[Uniqueness of small non-radiative nonlinear waves in even dimensions] \label{pr:uniquenesseven}

Assume $N\geq 4$ is even. Then there exists $\epsilon'>0$ such that if $u$ is any solution to \eqref{eq:nonlinearwaveintro} that is non-radiative for $|x|>R_0+|t|$ with $\int_{|x|> R_0} |\nabla_{t,x}u(0)|^2dx\leq \epsilon^{'2}$, then there exists $\cbf \in \mathbb R^{m_0+1}$ with $|\cbf|_{R_0}\lesssim \epsilon'$ such that
$$
\forall r>R_0+|t|, \qquad u(t,r)=a[\cbf,R_0](t,r)
$$
where $a[\cbf,R_0]$ is given by \eqref{id:defageneral}.

\end{proposition}

The proof will use several intermediate results contained in the following subsections.

\begin{proof}

Applying Lemma \ref{lem:criticalnondispersivestabilitycontrad} if $N\equiv 6\mod 4$, or Lemma \ref{lemN=4:criticalnondispersivestabilitycontrad} if $N\equiv 4 \mod 4$, there exists $ \bar R \in [R_0,256R_0]$ and $\bar{\cbf}\in \mathbb R^{m_0+1}$ with $|\bar{\cbf}|_{\bar R}\lesssim \epsilon$ such that $u=a[\bar{\cbf},\bar R]$ on $\{r\geq \bar R+|t|\}$. Using Proposition \ref{pr:constructionnonradia} and the local inverse Theorem, there exists $\cbf \in \mathbb R^{m}$ with $|\cbf|_{R_0}\lesssim \epsilon$ such that $\Pi_{\mathcal H_{\bar R}}\vec a[\cbf,R_0]=\Pi_{\mathcal H_{\bar R}}\vec a[\bar{\cbf},\bar{R}]$, so that $a[\cbf,R_0]=a[\bar{\cbf},\bar{R}]$ on $\{r>\bar R+|t|\}$ by Lemmas \ref{lem:criticalnondispersivestabilitycontrad} and \ref{lemN=4:criticalnondispersivestabilitycontrad}. Hence $u=a[\cbf,R_0]$ on $\{r\geq \bar R+|t|\}$. Applying Theorem \ref{th:maximal}, whose proof is independent of that of Proposition \ref{pr:uniquenesseven}, we obtain $u=a[\cbf,R_0]$ on $\{r\geq R_0+|t|\}$ as desired.

\end{proof}

\subsection{Proof of Proposition \ref{pr:uniquenesseven} in the case $N\equiv 6\mod 4$} \label{subsec:N=6}

In this subsection we fix $N\equiv 6\mod 4$ and prove Proposition \ref{pr:uniquenesseven} in that case, which is the most difficult, see the comments in Subsection \ref{subsec:N=4}. In what follows, all statements are valid simultaneously for $\varphi$ given by \eqref{eq:nonlinearityanalytic}, or \eqref{eq:nonlinearitypower} and $N=6$, unless explicitly mentioned.

We now fix once for all for this section $u$ a solution to \eqref{eq:nonlinearwaveintro} that is non-radiative for $|x|>R_0+|t|$ with $\int_{|x|> R_0} |\nabla_{t,x}u(0)|^2dx\leq \epsilon^{'2}$. We have the a priori estimates on $u$
\begin{equation}\label{even:bd:aprioriu}
\| \vec u \|_{L^\infty \mathcal H_{R_0+|t|}}+\sup_{r>R_0+|t|}r^{\frac{N-2}{2}}|u(t,r)|+\| r^{\frac{N-6}{4}}u\|_{L^2L^4_{R_0+|t|}}\lesssim \epsilon'
\end{equation}

First, we establish a gain of regularity for $u$. We let $R_1=16 R_0$ and $\varphi'=\frac{\pa}{\pa u} \varphi $.

\begin{lemma}[Gain of regularity] \label{even:lem:gainreg}

The functions $\indic (r> R_1+|t|)\pa_r\pa_t u $, $\indic (r>R_1+|t|)\pa_{tt} u$ and $\indic (r>R_1+|t|)\Delta u $ all belong to $ \mathcal CL^2 \cap L^\infty L^2$. In addition, $\pa_t u$ is a non-radiative solution for $r>R_1+|t|$ to
\begin{equation}\label{even:id:Boxpatu}
\left\{ \begin{array}{l l} \Box \pa_t u=\varphi'(u)\pa_t u,\\ (\pa_t u(0),\pa_{tt}u(0))=(u_1,\Delta u_0+\varphi(u_0),\end{array} \right. 
\end{equation}
and $\varphi'(u)\pa_t u\in L^1L^2_{R_1+|t|}$. It satisfies:
\begin{equation}\label{even:aprioripattuL21}
\forall R >R_1, \qquad \| \pa_{t} u(0)\|_{\dot H^1_{R}}+\| \pa_{tt} u(0)\|_{L^2_{R}}\lesssim \frac{\epsilon' }{ R}.
\end{equation}

\end{lemma}

\begin{proof}

Let $R\geq R_0$, and $v$ and $w$ be the solutions to
$$
\left\{\begin{array}{l l} \Box v=\chi(r-2R-|t|)\varphi(u)=:f,\\ \vec v(0)=(0,0), \end{array}  \right. 
$$
where $\chi$ is a smooth one-dimensional cut-off function with $\chi(s)=0$ for $s\leq -R$ and $\chi(s)=1$ for $s\geq 0$, and
$$
 \left\{\begin{array}{l l} \Box w= 0 ,\\ \vec w(0,r)=\left(u_0(r)\indic(r\geq R)+u_0(R)\indic(r<R),u_1(r)\indic(r\geq R)\right), \end{array}  \right.
$$
Note that by \eqref{even:bd:aprioriu}, since $\vec u(0)\in \mathcal H_{R_0}$ and $R\geq R_0$, then $\vec w(0)\in \mathcal H$ with $\| \vec w(0) \|_{\mathcal H}\lesssim \| \vec u(0)\|_{\mathcal H_{R_0}}\lesssim \epsilon'$. Applying Proposition \ref{construction:pr:radiation}, $w$ has radiation profiles $G_\pm[w]$ as $t\to \pm \infty$, that satisfy
\begin{equation}\label{even:gainreg:bd:G+w1}
\| G_{\pm}[w]\|_{L^2(\mathbb R)}\lesssim \epsilon'.
\end{equation}
Note that $f\in L^1L^2$ by \eqref{even:bd:aprioriu} and H\"older (see Lemma \ref{lem:lipschitz}), hence $v$ admits radiation profiles $G_\pm[v]\in L^2(\mathbb R)$ as $t\to \pm \infty$ by Corollary \ref{cor:inhomogeneousradiation}. Since $u=v+w$ for $r>2R+|t|$ by finite speed of propagation, and $u$ is non-radiative for $r>R+|t|$, there holds
\begin{equation}\label{even:gainreg:id:G+w=G+v}
G_+ [w](\rho)=-G_+ [v](\rho) \quad \mbox{and}\quad G_- [w](\rho)=-G_- [v](\rho)
\end{equation}
for all $\rho \geq 2 R$. In addition, $f$ is differentiable for $t\neq 0$ with $\pa_t f=f_1+f_2$ where
$$
 f_1=-\textup{sgn}(t)\chi'(r-|t|-2R)\varphi(u), \quad f_2=\chi(r-2R-|t|)\varphi'(u)\pa_t u.
$$
We have that $\chi'(r-|t|-2R)$ is zero for $r$ outside $[R+|t|,2R+|t|]$, and by \eqref{even:bd:aprioriu} that $|\varphi(u)|\lesssim  r^{-N/2-1}\epsilon^{'2}$ (see \eqref{NLLipsch1}), so that $|f_1|\lesssim \epsilon^{'2} (R+|t|)^{-N/2-1}\indic(R+|t|\leq r\leq 2R+|t|)$. Thus $\| f_1\|_{L^1L^2}\lesssim \epsilon^{'2}R^{-1}$ by a direct computation using \eqref{even:controle:bd:universal5}. By \eqref{even:bd:aprioriu} again we have $|\varphi'(u)|\lesssim  r^{-2}\epsilon^{'}$ (see \eqref{NLLipsch1}). Hence by H\"older and \eqref{even:controle:bd:universal5}:
\begin{equation}\label{even:gainreg:bd:upatuL1L2}
\| \varphi'(u)\pa_t u\|_{L^1L^2_{R+|t|}}\lesssim  \| \pa_t u \|_{L^\infty L^2_{R+|t|}} \| \varphi'(u)\|_{L^1L^\infty_{R+|t|}}\lesssim \epsilon' \| \frac{\epsilon'}{(R+|t|)^2}\|_{L^1}\lesssim \frac{\epsilon^{'2}}{R}.
\end{equation}
Hence $\|f_2\|_{L^1L^2}\lesssim \epsilon^{'2}R^{-1}$. Combining, we get $ \pa_t f\in L^1L^2$ with $\| \pa_t f\|_{L^1L^2}\lesssim \epsilon^{'2}R^{-1}$. From \cite{DuKeMaMe22} this implies that $G_+ [v],G_- [v] \in H^1(\mathbb R)$ with $\| \pa_\rho G_+[v]\|_{L^2(\mathbb R)}+\| \pa_\rho G_-[v]\|_{L^2(\mathbb R)}\lesssim \epsilon^{'2}R^{-1}$. Hence by \eqref{even:gainreg:id:G+w=G+v} $G_+ [w],G_- [w]\in H^1(\rho\geq 2R)$ with
\begin{equation}\label{even:gainreg:bd:G+w2}
\| \pa_\rho G_+[w]\|_{L^2(\rho\geq 2R)}+\| \pa_\rho G_-[w]\|_{L^2(\rho\geq 2R)}\lesssim \epsilon^{'2}R^{-1}.
\end{equation}
Since $\vec w(0)=\vec u(0)$ for $r\geq R$, applying Lemma \ref{lem:gainregradiation2} to $w$ using \eqref{even:gainreg:bd:G+w1} and \eqref{even:gainreg:bd:G+w2}, we obtain the gain of regularity for the initial data $(u_1 ,\Delta u_0)\in \mathcal H_{16 R}$ with
\begin{align}
\nonumber \| (u_1 ,\Delta u_0)\|_{\mathcal H_{16 R}} &=\| (w_1 ,\Delta w_0)\|_{\mathcal H_{16 R}}\\
\label{even:gainreg:bd:G+w3}& \lesssim \| \pa_\rho G_+[w]\|_{L^2(\rho\geq 2R)}+\| \pa_\rho G_-[w]\|_{L^2(\rho\geq 2R)}+\frac{1}{R}\| G_+[w]\|_{L^2(\mathbb R)} \ \lesssim \frac{\epsilon'}{R}.
\end{align}
We have $(\pa_t u(0),\pa_{tt}u(0))=(u_1 ,\Delta u_0+\varphi(u_0))$. We have by \eqref{even:bd:aprioriu} that $|\varphi(u_0)|\lesssim \epsilon^{'2} r^{-N/2-1}$ (see \eqref{NLLipsch1}) hence $\|\varphi(u_0)\|_{L^2_{R}}\lesssim \epsilon^{'2}R^{-1}$ by \eqref{even:controle:bd:universal1}. Combining this and \eqref{even:gainreg:bd:G+w3} shows \eqref{even:aprioripattuL21}. Taking $R=R_0$ we get $(\pa_t u(0),\pa_{tt}u(0))\in \mathcal H_{R_1}$ and the rest of the Lemma follows combining a classical propagation of regularity argument, energy estimates with \eqref{even:gainreg:bd:upatuL1L2}, and finite speed of propagation.

\end{proof}

Second, we show that the derivative gained by the previous lemma enjoys a gain of decay.

\begin{lemma} \label{even:lem:gainreg2}

There holds $\pa_t u \in \tilde W^{1}_{R_1}$ with
\begin{equation}\label{even:aprioripattuL2}
\forall R >R_1+|t|, \qquad \| \pa_{tt} u(t)\|_{L^2_{R}}\lesssim \frac{\epsilon' }{ R}.
\end{equation}

\end{lemma}

\begin{proof}

First, consider the case $R>R_1+2| t|$. Applying standard energy estimates to \eqref{even:id:Boxpatu}, using \eqref{even:aprioripattuL21}, \eqref{even:gainreg:bd:upatuL1L2} and finite speed of propagation:
$$
\| (\pa_{t} u(t),\pa_{tt}u(t))\|_{\mathcal H_R}\lesssim \| (\pa_{t} u(0),\pa_{tt}u(0))\|_{\mathcal H_{R-|t|}}+\| \indic(r>R-|t|+|t'|)\varphi'(u(t'))\pa_t u(t')\|_{L^1L^2}\lesssim \frac{\epsilon'}{R}
$$ 
where we used that $R\approx R-|t|$. This shows \eqref{even:aprioripattuL2} in that case.

Second, consider the case $R_1+|t|\leq R\leq R_1+2|t|$. Let $v(\bar t)=\pa_t u(t+\bar t)$. We have $\pa_{\bar t}v(0)=\pa_t^2 u(t)$. By Lemma \ref{even:lem:gainreg}, $v$ is a non-radiative solution for $r>R+|\bar t|$ of $(\pa_{\bar t}^2-\Delta) v=\varphi'(u(t+\bar t))\pa_t u(t+\bar t)$. Its radiation profiles then satisfy $G_+[v](\rho)=G_-[v](\rho)=0$ for all $\rho\geq R$. Applying the channels of energy estimate \eqref{bd:channels2} to $v$, using \eqref{even:gainreg:bd:upatuL1L2}, we obtain:
\begin{equation}\label{even:regdecay:bd:tech2}
\| \Pi_{L^2,R}^\perp \pa_{tt} u(t)\|_{L^2_{R}}=\| \Pi_{L^2,R}^\perp \pa_{t} v(0)\|_{L^2_{R}}\lesssim \| \varphi'(u(t+\cdot))\pa_t u(t+\cdot )\|_{L^1L^2_{R+|\bar t-t|}} \lesssim \frac{\epsilon'}{R}.
\end{equation}
Then using the identity
$$
\Pi_{L^2,R} \pa_{tt} u(t)= \pa_{tt} u(t)-\Pi_{L^2,R}^\perp \pa_{tt} u(t),
$$
the bounds \eqref{even:regdecay:bd:tech2} and \eqref{even:aprioripattuL2} for $R=R_1+2|t|$ imply $\| \Pi_{L^2,R} \pa_{tt} u(t)\|_{L^2_{R_1+2|t|}}\lesssim \epsilon' R^{-1}$. Note that $\| \Pi_{L^2,R} \pa_{tt} u(t)\|_{L^2_{R}}\approx \| \Pi_{L^2,R} \pa_{tt} u(t)\|_{L^2_{R_1+2|t|}}$ by finite dimensionality, since $\Pi_{L^2,R} \pa_{tt} u(t)\in \textup{Span}(r^{-N+2l+2})_{0\leq l \leq l_1}$ and $R\approx R_1+|t|$. Therefore, 
\begin{equation}\label{even:regdecay:bd:tech4}
\| \Pi_{L^2,R} \pa_{tt} u(t)\|_{L^2_{R}}\lesssim \frac{\epsilon'}{R}.
\end{equation}
Combining \eqref{even:regdecay:bd:tech2} and \eqref{even:regdecay:bd:tech4} implies the desired result \eqref{even:aprioripattuL2} in this second case as well.

\end{proof}

We recall that our aim is to show $u=a[\cbf]$ for $r>R_0+|t|$ for some $\cbf$. Our strategy will be to introduce an approximate non-radiative solution $u_{ap}[\cbf,\tilde c]$ defined below in \eqref{even:id:defuap}, with an extra parameter $\tilde c$ along the resonant direction $\phi_{m_0+1}$ given by \eqref{id:defphim0+1}, such that $u_{ap}[\cbf,0]=a[\cbf]$. Adding this extra parameter will enable us to control the difference $v=u-u_{ap}[\cbf,\tilde c]$, eventually leading to show $v=0$ and $\tilde c=0$, so that $u=a[\cbf]$ as desired.

For any $\cbf \in \mathbb R^{m_0+1}$, $\tilde c\in \mathbb R$ and $R>R_0$ we define the approximate solution:
\begin{equation}\label{even:id:defuap}
u_{ap}[\cbf,\tilde c,R]=a[\cbf,R]+\tilde \phi[\tilde c,R]\chi_{M R}
\end{equation}
where $a$ is given by Proposition \ref{pr:constructionnonradia}, $\tilde \phi$ is defined by \eqref{def:tildephi} and where for $\chi$ a smooth non-negative cut-off with $\chi(t,r)=1$ for $|r|+|t|\leq 1 $ and $\chi(t,r)=0$ for $|r|+|t|\geq 2 $, that is even in both time and space variables, we write $\chi_{L}(t,r)=\chi(\frac tL,\frac rL)$ for $L\in (0,\infty]$ with the convention that $\chi_\infty =1$. We introduce the error generated by $u_{ap}$:
\begin{equation}\label{even:id:defE}
E=-\Box u_{ap}+\varphi(u_{ap}).
\end{equation}

Recall the notation \eqref{def:beta} and \eqref{def:hatc}. The following lemma gathers estimates on $u_{ap}$, $a$, $\tilde \phi $ and $E$.

\begin{lemma} \label{even:lem:estimatesuap}

For all $R>0$, $\cbf \in \mathbb R^{m_0+1}$, $\tilde c\in \mathbb R$ with $|\cbf|_{R}+|\tilde c|\lesssim \epsilon'$ and $M>1$ with $\ln M\lesssim |\ln |\tilde c||$, $u_{ap}$ is well-defined for $r>R+|t|$ and satisfies the following estimates.
\begin{itemize}
\item \emph{Pointwise estimates}. For all $r>R+|t|$ one has:
\begin{align}
& \label{even:pointwisea}  |a|\lesssim \frac{R |\cbf|_{R}}{r^{N/2}} \\
& \label{even:pointwisetildephi} |\tilde \phi \chi_{RM}|\lesssim \frac{ |\tilde c|}{r^{N/2-1}}\indic(r+|t|\leq 2M R) \\
& \label{even:pointwiseuap} |u_{ap}|\lesssim \frac{\epsilon'}{r^{N/2-1}}
\end{align}
\item \emph{Averaged estimates in space}. For all $\bar R\geq R+|t|$:
\begin{align}
& \label{even:L2patuap} \| \frac{1}{r} \pa_t u_{ap}(t)\|_{L^2_{\bar R}}\lesssim \frac{R|\cbf|_{R}}{\bar R^2}+\frac{|\tilde c|}{\bar R}\\
& \label{even:L2E} \|E(t)\|_{L^2_{\bar R}}\lesssim \left( \frac{R|\cbf|_{R}|\tilde c|}{\bar R^2}+\frac{|\tilde c|^3(1+\beta|\ln |\tilde c||)}{\bar R}+\frac{|\tilde c|}{MR}\right)\indic (\bar R\leq 2MR) 
\end{align}
\item \emph{Averaged estimates in time and space}.
\begin{align}
& \label{even:L1L2E-} \| E_- \|_{L^1L^2_{R+|t|}}\lesssim |\cbf|_{R}|\tilde c|,\\
& \label{even:L2patE} \| \pa_t E(t)\|_{L^1L^2_{\bar R+|t-\bar t|}}\lesssim \frac{|\tilde c|\left(|\cbf|_{R}+|\tilde c|^2(1+\beta|\ln |\tilde c||)+\frac{1}{M^{1/2}}\right)}{\bar R^{1/2} R^{1/2}},\\
& \label{even:L2L4uap-} \| r^{\frac{N-6}{4}} u_{ap,-}\|_{L^2L^4_{R+|t|}}\lesssim \epsilon'.
\end{align}
\end{itemize}

\end{lemma}

The proof of Lemma \ref{even:lem:estimatesuap} is by direct computations, done in Appendix \ref{A:approximate}. We choose:
\begin{equation}\label{even:id:defM}
M=\frac{\nu}{|\cbf|_{\tilde R}^2+\tilde c^4(1+\beta |\ln |\tilde c||)^2+\left(\sup_{r>R+|t|}r^{N/2-1}|u(t,r)-u(-t,r)|\right)^4}
\end{equation}
for some $0<\nu\ll 1$ independent of $\tilde c$, $\cbf$ and $R$. This choice of $M$ is technical: the larger $M$ is, the smaller the error terms in $E$ given by \eqref{even:id:defE} due to boundary terms created by the cut-off $\chi_{M R}$ are; but the larger $M$ is, the larger the nonlinear interactions with $\tilde \phi \chi_{M \tilde R}$ are in the proof of the forthcoming Lemma \ref{even:lem:approximation} (since $\tilde \phi$ is not in the energy space). Optimising these two constraints leads to the choice \eqref{even:id:defM}.

We approximate $u$ for $|x|>R+|t|$ by $u_{ap}[\cbf,\tilde c,R]$ whose parameters are chosen as follows:

 \begin{lemma} \label{lem:defctildec}

For all $R>R_1$, there exists $\cbf[R]\in \mathbb R^{m_0+1}$ and $\tilde c[R]\in \mathbb R$ with
\begin{equation}\label{even:bd:cbftildectildeR}
|\cbf[R]|_{R}+|\tilde c[R]|\lesssim \epsilon' \quad \mbox{and}\quad \lim_{R\to \infty} |\cbf[R]|_{R}+|\tilde c[R]|=0
\end{equation}
such that, $u_{ap}=u_{ap}[\cbf[R],\tilde c[R],R]$ satisfies:
\begin{align}
\label{even:id:orthogonalityatR}& u(0,R)-u_{ap}(0,R)=0,\\
\label{even:id:orthogonalitypatv}& \Pi_{L^2,R}\left(\pa_t u(0)-\pa_t u_{ap}(0)\right)=0,\\
\label{even:id:orthogonalityDeltav}& \Pi_{L^2,R}\left(\pa_{tt} u(0)-\pa_{tt} u_{ap}(0)\right)=0.
\end{align}

 \end{lemma}

\begin{proof}

Fix $R>R_1$ and let $\Psi=(\Psi_0,\Psi_1)$ be the map defined by:
$$
\Psi_0(f)=\left(f(0,R),\Big(\int_{r>R}\pa_{tt}f(0,r)r^{2l+1}dr\Big)_{0\leq l\leq l_1}\right), \quad \Psi_1(f)= \Big(\int_{r>R}\pa_{t}f(0,r)r^{2l+1}dr\Big)_{0\leq l\leq l_1}.
$$
By Cauchy-Schwarz, for $R=1$, for any $f$ with $\vec f(0)\in \mathcal H_1$ and $\pa_{tt} f(0)\in L^2_1$, $\Psi$ is well-defined with
\begin{equation}\label{even:bd:linearctildectech}
|\Psi_0(f)|\lesssim |f(0,1)|+\| \pa_{tt}f(0)\|_{L^2_1} \quad \mbox{and}\quad |\Psi_1(f)|\lesssim \| \pa_{t}f(0)\|_{L^2_1}.
\end{equation}
To prove the Lemma, we will find $(\cbf,\tilde c)$ satisfying \eqref{even:bd:cbftildectildeR} and $\Psi(u_{ap})=\Psi(u)$.\\

\noindent \textbf{Step 1}. \emph{The linear projection}. Let $\Phi(\cbf,\tilde c)=a_{F}[\cbf]+\tilde c \phi_{m_0+1}$. We prove in this first step that $\Psi \circ \Phi$ is a linear invertible map on $\mathbb R^{m_0+2}$, and that for any $f$ with $\vec f(0)\in \mathcal H_R$ and $\pa_{tt} f(0)\in L^2_R$, $(\cbf,\tilde c)=(\Psi \circ \Phi)^{-1}(\Psi (f))$ satisfies
\begin{align}
\label{even:bd:linearctildec}& |\tilde c|\lesssim R^{N/2-1}|f(0,R)|+R\| \pa_{tt}f(0)\|_{L^2_R}, \\
\label{even:bd:linearctildec2}& |\cbf|_R\lesssim R^{N/2-1}|f(0,R)|+\|\pa_t f(0)\|_{L^2_R}+R\| \pa_{tt}f(0)\|_{L^2_R}.
\end{align}

By scaling, it suffices to show this result in the case $R=1$. We decompose $\Phi(\cbf)=\Phi_0(\cbf_0,\tilde c)+\Phi_1(\cbf_1)$ where $\Phi_0(\cbf_0,\tilde c)= \sum_{0\leq l\leq l_0} c_{0,l}\phi_{2l}+\tilde c \phi_{m_0+1}$ and $\Phi_1(\cbf_1)= \sum_{0\leq l\leq l_1} c_{1,l}\phi_{2l+1}$. We have $\Psi\circ \Phi=(\Psi_0\circ \Phi_0,\Psi_1\circ \Phi_1)$. Using \eqref{id:aF} and \eqref{id:defphim0+1} we obtain the identities:
\begin{align*}
& (\Phi(\cbf,\tilde c))(0,1)=\sum_{0\leq l\leq l_0}c_{0,l}+\tilde c,\\
& \left(\pa_{t}\Phi(\cbf,\tilde c) \right)(0,r)=\sum_{0\leq l\leq l_1}c_{1,l}r^{-N+2l},\\
&  \left(\pa_{tt}\Phi(\cbf,\tilde c) \right)(0,r)=\sum_{1\leq l\leq l_0}2l(-N+2l+2)c_{0,l}r^{-N+2l}-(\frac N2-1)^2 r^{-N+2l_0+2},
\end{align*}
and deduce, since $l_0=l_1$ when $N\equiv 6\mod 4$, that both map $ \Psi_0\circ\Phi_0$ and $ \Psi_1\circ\Phi_1$ are invertible maps. Hence $\Psi\circ \Phi$ is invertible, and hence, using \eqref{even:bd:linearctildectech}, the estimates \eqref{even:bd:linearctildec} and \eqref{even:bd:linearctildec2}.\\

\noindent \textbf{Step 2}. \emph{The non-linear projection}. Let $ \tilde \Phi(\cbf,\tilde c)=\tilde a[\cbf,R]+\chi_{MR} \tilde \phi[\tilde c,R] -\tilde c \phi_{m_0+1}$, where $\tilde{a}[\cbf,R]$ is defined in Proposition \ref{pr:constructionnonradia} and Definition \ref{def:constructionnonradia}. By the differentiability property of Proposition \ref{pr:constructionnonradia} and the radial Sobolev inequality, we infer that for any $\cbf,\cbf'\in \mathbb R^{m_0+1}$ with $|\cbf|_R+|\cbf'|_R$ small one has:
\begin{align}
 \label{even:bd:linearctildectech5} & R^{N/2-1}|(\tilde a[\cbf]-\tilde a[\cbf'])(0,R)|+\| \pa_t( \tilde a[\cbf]-\tilde a[\cbf']  )(0) \|_{L^2_R}+R\| \pa_{tt}( \tilde a[\cbf]-\tilde a[\cbf']  )(0)\|_{L^2_R} \\
\nonumber & \qquad \qquad \qquad \qquad \qquad \qquad \qquad \qquad  \qquad \qquad \qquad \qquad \qquad  \lesssim |\cbf-\cbf'|_R(|\cbf|_R+|\cbf '|_R).
\end{align}
We now claim that for any $\tilde c,\tilde c'\in \mathbb R$ and general $M,M'>1$ (i.e. not necessarily given by formula \eqref{even:id:defM}) there holds:
\begin{align}
\label{even:bd:linearctildectech2}& (\chi_{MR} \tilde \phi[\tilde c,R] -\tilde c \phi_{m_0+1})(0,R)=0,\\
\label{even:bd:linearctildectech3} & \pa_t(\chi_{MR} \tilde \phi[\tilde c,R] -\tilde c \phi_{m_0+1})(0,r)=0 \qquad \forall r>R,\\
\label{even:bd:linearctildectech4}& \| \pa_{tt}(\chi_{MR} \tilde \phi[\tilde c,R] -\tilde c \phi_{m_0+1})(0)-\pa_{tt}(\chi_{M'R} \tilde \phi[\tilde c'] -\tilde c' \phi_{m_0+1})(0) \|_{L^2_R}\\
\nonumber &\qquad \qquad \qquad \qquad \lesssim \frac{|\tilde c-\tilde c'|}{R}\left(|\tilde c|+|\tilde c'|+\frac 1M+\frac{1}{M'}\right)+\left|\frac{1}{M}-\frac{1}{M'}\right|\frac{|\tilde c|+|\tilde c'|}{R}.
\end{align}
Then, assuming temporarily this claim, we have that $\cbf $ and $\tilde c$ satisfy the conclusions of the Lemma if and only if they satisfy \eqref{even:bd:cbftildectildeR} and solve $\Psi ((\Phi+\tilde \Phi)(\cbf,\tilde c))=\Psi(u)$. By the result of Step 1, the latter equation is equivalent to the fixed point equation
\begin{equation}\label{even:id:fixedpointctildec}
(\cbf,\tilde c)=(\Psi \circ \Phi)^{-1}(\Psi(u)-\Psi(\tilde \Phi (\cbf,\tilde c))).
\end{equation}
By the estimate \eqref{even:bd:linearctildectech} for $\Psi$, the estimates \eqref{even:bd:linearctildectech5}, \eqref{even:bd:linearctildectech2}, \eqref{even:bd:linearctildectech3} and \eqref{even:bd:linearctildectech4} for $\tilde \Phi$, the choice \eqref{even:id:defM} for $M$ (along with $\sup_{r>R+|t|}r^{N/2-1}|u(t,r)-u(-t,r)|\lesssim \epsilon'$ by \eqref{even:bd:aprioriu}), the right-hand side of \eqref{even:id:fixedpointctildec}, for $\nu$, $\eps'$ small, is a contraction on the set $\{(\cbf,\tilde c)\in \mathbb R^{m_0+2} \mbox{ satisfying }\eqref{even:bd:cbftildectildeR} \}$. Hence, by the Banach fixed point Theorem, there exists a unique $(\cbf ,\tilde c)$ solving \eqref{even:id:fixedpointctildec} and satisfying the desired estimate \eqref{even:bd:cbftildectildeR}.

There remains to prove \eqref{even:bd:linearctildectech2}, \eqref{even:bd:linearctildectech3} and \eqref{even:bd:linearctildectech4}. First, by \eqref{def:tildephi}, \eqref{rigidity|u|u:id:defphi2}, \eqref{rigidity|u|u:id:eqA2} and $\tilde p(0)=0$ we have $\tilde \phi(0,R)=\tilde c\phi_{m_0+1}(0,R)$, and   \eqref{even:bd:linearctildectech2} follows. As $\chi$, $\tilde \phi$ and $\phi_{m_0+1}$ are even with time, we obtain \eqref{even:bd:linearctildectech3}. Next, we decompose using \eqref{def:tildephi}:
\begin{align}
\label{even:bd:linearctildectech10} &(\chi_{MR} \tilde \phi[\tilde c,R] -\tilde c \phi_{m_0+1})-(\chi_{M'R} \tilde \phi[\tilde c',R] -\tilde c' \phi_{m_0+1}) \\
\nonumber &\qquad \qquad = (\chi_{MR}-\chi_{M'R}) \tilde \phi[\tilde c] +\chi_{M'R}(\hat c^2-\hat c^{'2})\psi+(\chi_{M'R}-1 )(\tilde c-\tilde c') \phi_{m_0+1}.
\end{align}
We estimate the first term, assuming $M<M'$ without loss of generality. If $M'\geq 2M$ then $|M^{-1}-M^{'-1}|\approx M^{-1}$ and performing a direct estimate using \eqref{def:tildephi} and $|\ln M|+|\ln M'|\lesssim |\ln |\tilde c||$:
$$
\| \pa_{tt} ((\chi_{MR}-\chi_{M'R}) \tilde \phi[\tilde c])(0) \|_{L^2_R}^2\lesssim \int_{MR}^\infty \frac{\tilde c^2}{(r^{\frac N2+1})^2} r^{N-1}dr\lesssim \frac{\tilde c^2}{R^2M^2}\approx \frac{\tilde c^2}{R^2}\left|\frac{1}{M}-\frac{1}{M'} \right|^2.
$$
If $M<M'< 2M$ then $\textup{supp}(\chi_{MR}-\chi_{M'R})\subset\{M/2\leq |r|+|t|\leq 4M\}.$ By the mean value Theorem: $|\chi_{MR}(r)-\chi_{M'R}(r)|\lesssim \frac{r}{R}(M^{-1}-M^{'-1})\indic (MR/2\leq r \leq  4MR)$ (and similarly for higher order derivatives), so that performing a direct computation using \eqref{def:tildephi}:
$$
\| \pa_{tt} ((\chi_{MR}-\chi_{M'R}) \tilde \phi[\tilde c])(0) \|_{L^2_R}^2\lesssim \left|\frac{1}{M}-\frac{1}{M'} \right|^2\int_{MR/2}^{4MR} \frac{\tilde c^2r^2}{R^2(r^{\frac N2+1})^2} r^{N-1}dr\approx \frac{\tilde c^2}{R^2}\left|\frac{1}{M}-\frac{1}{M'} \right|^2.
$$
This shows the desired bound \eqref{even:bd:linearctildectech4} for the first term in the right-hand side of \eqref{even:bd:linearctildectech10}. The estimates for the remaining terms are easier, relying on direct computations using \eqref{def:tildephi} and $|\ln M|+|\ln M'|\lesssim |\ln |\tilde c||$; we omit them. This ends the proof of \eqref{even:bd:linearctildectech4}.

 \end{proof}

The key result of this Section is the following, estimating the difference $u-u_{ap}[\cbf,\tilde c,R]$ solely in terms of the projection parameters $\cbf$ and $\tilde c$.

\begin{lemma} \label{even:lem:approximation}

For all $R>R_1$, let $\cbf=\cbf[ R]$ and $\tilde c=\tilde c[R]$ be given by Lemma \ref{lem:defctildec}. Then:
\begin{align}
\label{even:bd:u-uap1} &\| \pa_t u(0)-\pa_t u_{ap}[\cbf,\tilde c,R](0)\|_{L^2_{R}} \lesssim |\tilde c|( |\cbf|_{R}+|\tilde c|),\\
\label{even:bd:u-uap2} & \| \pa_{tt} u(0)-\pa_{tt} u_{ap}[\cbf,\tilde c,R](0)\|_{L^2_{R}}\lesssim  \frac{|\tilde c|\left(|\cbf|_{R}+\tilde c^2 (1+\beta |\ln |\tilde c||)\right)}{R},\\
\label{even:bd:u-uap3} & |u(0,r)-u_{ap}[\cbf,\tilde c,R](0,r)| \lesssim  \frac{|\tilde c|\left(|\cbf|_{R}+\tilde c^2(1+\beta |\ln |\tilde c||) \right)}{r^{N/2-3/2}R^{1/2}}, \qquad \forall r\geq R.
\end{align}

\end{lemma}

\begin{proof}

To ease notations, we drop the dependence on the parameters $R$, $\cbf[R]$ and $\tilde c[R]$ in the notation, and just write $u_{ap}$ for $u_{ap}[\cbf[R],\tilde c[R],R]$ for example. We fix $R>R_1$ and let
\begin{equation}\label{even:id:defv}
v=u-u_{ap}.
\end{equation}
By \eqref{even:id:defE} it is a non-radiative solution for $r>R+|t|$ to
\begin{equation}\label{even:id:equationu-uap}
\left\{ \begin{array}{l l} \Box v=E+\varphi(u)-\varphi(u_{ap}),\\
(v(0),\pa_t v(0))=(u_0-u_{ap}(0),u_1-\pa_t u_{ap}(0)).
\end{array} \right.
\end{equation}
We decompose $v=v_++v_-$ and aim at estimating the quantities
\begin{align}
& \label{even:controle:id:q1} q_1 =\| v_-\|_{L^\infty \mathcal H_{R+|t|}}+\| r^{\frac{N-6}{4}}v_-\|_{L^2L^4_{R+|t|}},\\
& \label{even:controle:id:q2} q_2=R^{1/2}\| \pa_t v_+\|_{\tilde W_{R}^{1/2}},\\
& \label{even:controle:id:tildeq1} \tilde q_1 = \sup_{t\in \mathbb R}\sup_{r>R+|t|} r^{N/2-1}|v_-(t,r)| .
\end{align}
Note $\tilde q_1\lesssim q_1$ by Sobolev. Because of the estimates \eqref{even:bd:cbftildectildeR}, \eqref{even:L2L4uap-} and \eqref{even:bd:aprioriu}, of Proposition \ref{pr:constructionnonradia} and of Lemma \ref{even:lem:gainreg2} (that we apply to both $u$ and $a$), both $q_1$ and $q_2$ are finite with:
\begin{equation}\label{even:controle:bd:aprioriq1q2}
q_1+q_2\lesssim \epsilon'.
\end{equation}
By \eqref{even:id:defuap} we have $u_{ap,-}=a_-$, hence $u_-=a_-+v_-$. By \eqref{even:pointwisea}, \eqref{even:controle:id:tildeq1} and the radial Sobolev embedding we obtain on the one hand that $\sup_{ r>R+|t|}r^{N/2-1}|u_-(t,r)|\lesssim |\cbf|_R+\tilde q_1$ and on the other that $\tilde q_1\lesssim \sup_{r>R+|t|}r^{N/2-1}|u_-(t,r)|+ |\cbf|_R$. Hence, as $|\cbf|_R\lesssim \epsilon'\ll1$, $M$ defined by \eqref{even:id:defM} satisfies
\begin{equation}\label{even:bd:approxM}
M\approx \frac{\nu}{|\cbf|_{R}^2+\tilde c^4(1+\beta |\ln |\tilde c||)^2+\tilde q_1^4}\gg 1.
\end{equation}

We will use several times the following inequalities for $x,y\in \mathbb R$ with $|x|,|y|\ll r^{-N/2+1}$:
\begin{align}
\label{even:bd:varphi1}  & |\varphi (x+y)-\varphi(x)|\lesssim r^{\frac{N-6}{2}}|y|(|x|+|y|),\\
\label{even:bd:varphi2}  & |\varphi (x+y)-\varphi(x)-\varphi(y)|\lesssim r^{\frac{N-6}{2}}|x||y|,\\
\label{even:bd:varphi3} & |\varphi'(x)|\lesssim r^{\frac{N-6}{2}}|x| ,\\
\label{even:bd:varphi4} & |\varphi'(x+y)-\varphi'(x)|\lesssim r^{\frac{N-6}{2}}|y| .
\end{align}
We omit their proof, which can be done by standard arguments.

\noindent \textbf{Step 1}. \emph{Preliminary estimates}. In this step we control $v_+$ and $v_-$ by the quantities $q_1$, $q_2$, $\tilde q_1$, $M$ and $\epsilon'$. We claim that for all $r>R+|t|$:
\begin{align} 
& \label{even:bdv} |v(t,r)|\lesssim \frac{\epsilon'}{r^{N/2-1}},\\
& \label{even:bdv-} |v_-(t,r)|\lesssim \frac{\tilde q_1}{r^{N/2-1}},\\
& \label{even:bdv+} |v_+(t,r)|\lesssim  \frac{q_2+|\tilde c|\left(|\cbf|_{R}+\tilde c^2(1+\beta |\ln |\tilde c||)+\frac{1}{M^{1/2}}\right)}{r^{\frac{N-3}{2}}R^{1/2}} +\frac{\epsilon' q_1}{r^{N/2-1}},
\end{align}
that for all $r>R$:
 \be
\label{even:bdv+t=0} |v_+(0,r)|\lesssim \frac{q_2+|\tilde c|\left(|\cbf|_{R}+\tilde c^2(1+\beta |\ln |\tilde c||)+\frac{1}{M^{1/2}}\right)}{r^{\frac{N-3}{2}}R^{1/2}},
\end{equation}
and that for all $\bar R>R+|t|$:
\begin{align}
& \label{even:bdv+L2}  \| \frac{v_+(t)}{r^2}\|_{L^2_{\bar R}}\lesssim \frac{q_2+|\tilde c|\left(|\cbf|_{R}+\tilde c^2(1+\beta |\ln |\tilde c||)+\frac{1}{M^{1/2}}\right)}{\bar R^{1/2}R^{1/2}},\\
&\label{even:bdpatv+L2}  \| \frac{\pa_t v_+(t)}{r^2}\|_{L^2_{\bar R}}\lesssim \frac{q_2}{\bar R^{3/2}R^{1/2}} .
\end{align}

\noindent \underline{Proof of \eqref{even:bdv} and \eqref{even:bdv-}:} The bound \eqref{even:bdv} is a direct consequence of \eqref{even:id:defv}, \eqref{even:pointwiseuap} and \eqref{even:bd:aprioriu}. The bound \eqref{even:bdv-} is by definition of $\tilde q_1$.
\smallskip

\noindent \underline{Proof of \eqref{even:bdv+t=0}:} Using \eqref{even:id:equationu-uap}, \eqref{even:id:defv} and $v(\pm t)=v_+(t)\pm v_-(t)$, $v_+$ satisfies for any $t\in \mathbb R$:
\begin{align} \label{even:id:Deltav+}
\Delta v_+(t) &= (\pa_{tt}v+\varphi(u_{ap})-\varphi(u_{ap}+v)-E)_+(t)\\ \notag
&= \pa_{tt}v_+(t)+(\varphi(u_{ap})-\varphi(u_{ap}+v_+))_+(t)\\ \notag
&\qquad \qquad+(\varphi(u_{ap}+v_+)-\varphi(u_{ap}+v_++v_-))_+(t)-E_+(t).
\end{align}
We let $t\in \mathbb R$ and $\bar R\geq R+|t|$, and estimate all terms. By definition of the $\tilde W^{1/2}_{R}$ norm \eqref{id:deftildeWkappa}:
\begin{equation}\label{even:controle:tech1}
\| \pa_{tt}v_+(t)\|_{L^2_{\bar R}}\lesssim \frac{ q_2}{\bar R^{1/2}R^{1/2}}.
\end{equation}
Let $r>\bar R$. Using \eqref{even:bd:varphi1}, \eqref{even:pointwiseuap} and \eqref{even:bdv},
\begin{equation}\label{even:controle:tech2}
|\varphi(u_{ap})-\varphi(u_{ap}+v_+)|\lesssim r^{\frac{N-6}{2}}|v_+|(|v_+|+|u_{ap}|)\lesssim \frac{\epsilon'}{r^2}|v_+|.
\end{equation}
Using \eqref{even:L2E}, $\bar R\geq R$ and $\bar R\leq 2MR$ on the support of the right-hand side of \eqref{even:L2E} we get:
\begin{equation}\label{even:controle:tech5}
\| E(t) \|_{L^2_{\bar R}}\lesssim \frac{|\tilde c|\left(|\cbf|_{R}+\tilde c^2(1+\beta|\ln|\tilde c||)+\frac{1}{M^{1/2}}\right) }{\bar R^{1/2} R^{1/2}}.
\end{equation}
At the initial time $t=0$, injecting \eqref{even:controle:tech1}, \eqref{even:controle:tech2}, $(\varphi(u_{ap}+v_+)-\varphi(u_{ap}+v_++v_-))(0)=0$ since $v_-(0)=0$ and \eqref{even:controle:tech5} in \eqref{even:id:Deltav+} we get:
\begin{equation}\label{even:controle:tech4}
\| \Delta v_+(0)\|_{L^2_{\bar R}}\lesssim \frac{q_2+|\tilde c|\left(|\cbf|_{R}+\tilde c^2(1+\beta |\ln |\tilde c||)+\frac{1}{M^{1/2}}\right)}{\bar R^{1/2}R^{1/2}} +\epsilon' \| \frac{v_+(0)}{r^2}\|_{L^2_{\bar R}}.
\end{equation}
Since $v_+(0,R)=0$ by \eqref{even:id:orthogonalityatR}, injecting \eqref{even:controle:tech4} in the weighted Hardy inequality \eqref{bd:hardyoutside} with $\kappa=1/2$ shows:
\begin{align*}
\sup_{\bar R\geq R} \bar R^{\frac 12} \| \frac{v_+(0)}{r^2}\|_{L^2_{\bar R}}+\sup_{r\geq {R}} r^{\frac{N-3}{2}} |v_+(0,r)| & \lesssim \frac{q_2+|\tilde c|\left(|\cbf|_{R}+\tilde c^2(1+\beta |\ln |\tilde c||)+\frac{1}{M^{1/2}}\right)}{R^{1/2}}\\
&\qquad +\epsilon' \sup_{\bar R\geq R}\bar R^{\frac 12} \| \frac{v_+(0)}{r^2}\|_{L^2_{\bar R}} .
\end{align*}
This implies \eqref{even:bdv+t=0} for $\epsilon'$ small enough.

\smallskip

\noindent \underline{Proof of \eqref{even:bdv+L2} and \eqref{even:bdpatv+L2}:} As $v_+$ is even in time, $\pa_t v_+(t)=\int_0^t \pa_{tt}v_+(t')dt'$. Hence \eqref{even:bdpatv+L2} follows from \eqref{even:controle:tech1}, $|t|\leq \bar R$ and $r>\bar R$. Similarly, $v_+(t)=v_+(0)+\int_0^{t}\pa_{t}v_+(t')dt'$. Hence \eqref{even:bdv+L2} follows from \eqref{even:bdv+t=0} and \eqref{even:controle:bd:universal1}, and \eqref{even:bdpatv+L2} with $|t|\leq \bar R$.

\smallskip

\noindent \underline{Proof of \eqref{even:bdv+}:} Let $t\in \mathbb R$ and $R(t)=R+|t|$. In view of \eqref{even:id:Deltav+}, we decompose
\begin{equation}\label{even:controle:tech6}
v_+(t,r)=\frac{c(t)}{r^{N-2}}+v_+^{(1)}(r)+v_+^{(2)}(r)
\end{equation}
where $c(t)\in \mathbb R$, and $v_+^{(1)}[t]$ and $v_+^{(2)}[t]$ are the unique solutions that decay as $r\to \infty$ of
\begin{align*}
& \left\{ \begin{array}{l l} \Delta v_+^{(1)}(r)=\pa_{tt}v_+(t,r)+(\varphi(u_{ap})-\varphi(u_{ap}+v_+))_+(t,r)-E_+(t,r) \qquad \forall r>R(t),\\ v_+^{(1)}(R(t))=0 ,\end{array} \right. \\
& \left\{ \begin{array}{l l}  \Delta v_+^{(2)}=(\varphi(u_{ap}+v_+)-\varphi(u_{ap}+v_++v_-))_+(t,r)  \qquad \forall r>R(t),\\ v_+^{(2)}(R(t))=0,   \end{array} \right. 
\end{align*}
(which exist, by using a rescaled version of the formula \eqref{bd:hardyoutsidetech4}). For any $\bar R\geq R(t)$, combining \eqref{even:controle:tech1}, \eqref{even:controle:tech2}, \eqref{even:bdv+L2} and \eqref{even:controle:tech5} we have that 
$$
\| \pa_{tt}v_+(t)+(\varphi(u_{ap})-\varphi(u_{ap}+v_+))_+-E_+(t)\|_{L^2_{\bar R}}\lesssim \frac{q_2+|\tilde c|\left(|\cbf|_{R}+\tilde c^2(1+\beta |\ln |\tilde c||)+\frac{1}{M^{1/2}}\right)}{\bar R^{1/2}R^{1/2}}.
$$
Hence by \eqref{bd:hardyoutside} with $\kappa=1/2$, for all $r>R(t)$:
\begin{equation}\label{even:controle:tech7}
|v_+^{(1)}(r)|\lesssim \frac{q_2+|\tilde c|\left(|\cbf|_{R}+\tilde c^2(1+\beta |\ln |\tilde c||)+\frac{1}{M^{1/2}}\right)}{r^{(N-3)/2}R^{1/2}} .
\end{equation}
By \eqref{even:bd:varphi1}, \eqref{even:bdv-} with $\tilde q_1\lesssim q_1$, \eqref{even:pointwiseuap} and \eqref{even:bdv}, for $r>R(t)$:
$$
|\varphi(u_{ap}+v_+)-\varphi(u_{ap}+v_++v_-)|\lesssim r^{\frac{N-6}{2}} |v_-|(|u_{ap}|+|v_+|+|v_-|)\lesssim \frac{\epsilon' q_1}{r^{N/2+1}}.
$$
Hence for any $\bar R\geq R(t)$, $\| \varphi(u_{ap}+v_+)-\varphi(u_{ap}+v_++v_-)\|_{L^2_{\bar R}}\lesssim \epsilon' q_1/ \bar R$ by \eqref{even:controle:bd:universal1}. Using this and \eqref{bd:hardyoutside} with $\kappa=1$ we get for all $r>R(t)$:
\begin{equation}\label{even:controle:tech8}
|v^{(2)}_+(r)|\lesssim \frac{\epsilon' q_1}{r^{N/2-1}}.
\end{equation}
By \eqref{even:controle:tech6} we have $ c(t)r^{-N}=r^{-2}(v_+(t)-v_+^{(1)}+v_+^{(2)})$. Taking the $L^2_{R(t)}$ norm, using \eqref{even:bdv+L2}, \eqref{even:controle:tech7}, \eqref{even:controle:tech8} and \eqref{even:controle:bd:universal1} we obtain:
\begin{align*}
\| \frac{c(t)}{r^N}\|_{L^2_{R(t)}} & \lesssim \| \frac{|v_+|}{r^2}\|_{L^2_{R(t)}}+\| \frac{q_2+|\tilde c|(|\cbf|_{R}+\tilde c^2(1+\beta|\ln |\tilde c||)+\frac{1}{M^{1/2}})}{r^{N/2+1/2}R^{1/2}}\|_{L^2_{R(t)}}+\|\frac{\epsilon' q_1}{r^{N/2+1}}\|_{L^2_{R(t)}} \\
&\lesssim \frac{q_2+|\tilde c|(|\cbf|_{R}+\tilde c^2+\frac{1}{M^{1/2}})}{R(t)^{1/2}R^{1/2}}+\frac{q_1 \epsilon'}{ R(t)}.
\end{align*}
Hence using $\| r^{-N}\|_{L^2_{R(t)}}\approx R(t)^{-N/2}$ we deduce
\begin{equation}\label{even:controle:tech9}
|c(t)|\lesssim \frac{R(t)^{N/2-1/2}\left(q_2+|\tilde c|\left(|\cbf|_{R}+\tilde c^2(1+\beta|\ln |\tilde c||)+\frac{1}{M^{1/2}}\right)\right)}{R^{1/2}}+q_1 \epsilon' R(t)^{N/2-1}
\end{equation}
For $r>R(t)$, reinjecting \eqref{even:controle:tech7}, \eqref{even:controle:tech8} and \eqref{even:controle:tech9} in \eqref{even:controle:tech6} we obtain \eqref{even:bdv+}, ending Step 1.\\

\noindent \textbf{Step 2}. \emph{Control of $v_-$}. In this step we show an improved estimate for $q_1$. We have, using \eqref{even:id:equationu-uap}, $v=v_++v_-$ and $(\varphi (v_+))_-=0$ that $v_-$ solves:
\begin{align*}
 \Box v_-  & = (E+\varphi(u_{ap}+v)-\varphi(u_{ap}))_-\\
& = E_-+(\underbrace{\varphi (u_{ap}+v_++v_-)-\varphi(u_{ap}+v_+)}_{=I})_-+(\underbrace{\varphi(u_{ap}+v_+)-\varphi (u_{ap})-\varphi (v_+)}_{=II})_-.
\end{align*}
For $I$, using \eqref{even:bd:varphi1}, and then $u=u_{ap}+v_++v_-$, for $r>R+|t|$:
$$
|I|\lesssim r^{\frac{N-6}{2}}|v_-|(|u+v_++v_-|+|u+v_+|) \lesssim r^{\frac{N-6}{2}} |v_-|(|u|+|v_-|+|u_{ap}|).
$$
For the first two terms by H\"older, \eqref{even:bd:aprioriu}, \eqref{even:controle:id:q1} and \eqref{even:controle:bd:aprioriq1q2}:
\begin{align*}
\| r^{\frac{N-6}{2}} |v_-|(|u|+|v_-|)\|_{L^1L^2_{R+|t|}} & \lesssim \| r^{\frac{N-6}{4}} v_-\|_{L^2L^4_{R+|t|}} \left(\| r^{\frac{N-6}{4}} (|u|+ | v_-|)\|_{L^2L^4_{R+|t|}}  \right) \\
 &\lesssim q_1(\epsilon' +q_1) \ \lesssim \epsilon'q_1.
\end{align*}
For the second, using \eqref{even:bdv-} and $\tilde q_1\lesssim q_1$, \eqref{even:id:defuap}, \eqref{even:pointwisea} and \eqref{even:pointwisetildephi}, we get for $r>R+|t|$:
$$
|v_-  u_{ap}|\lesssim q_1\left(\frac{R|\cbf|_{R}}{r^{N-1}}+\frac{|\tilde c|}{r^{N-2}}\indic (r+|t|\leq 2MR) \right).
$$
This implies using \eqref{even:controle:bd:universal4}, \eqref{even:bd:cbftildectildeR}, $\ln M \lesssim |\ln |\tilde c||$ (by \eqref{even:bd:approxM}) and \eqref{even:bd:cbftildectildeR} again:
\begin{align*}
\| r^{\frac{N-6}{2}}v_- u_{ap}\|_{L^1L^2_{R+|t|}} & \lesssim q_1\left( \| \frac{|\cbf|_R R}{r^{N/2+2}}\|_{L^1L^2_{R+|t|}}+\| \frac{\tilde c}{r^{N/2+1}}\indic (r+|t|\leq 2MR) \|_{L^1L^2_{R+|t|}}\right) \\
& \lesssim q_1( \frac{|\cbf|_RR}{R}+|\tilde c|\log M) \lesssim q_1( \epsilon' +|\tilde c| |\ln |\tilde c||) \lesssim q_1 \epsilon' |\ln \epsilon '|.
\end{align*}
Combining, we get that for $I$:
\begin{equation}\label{even:controle:tech10} 
\| I\|_{L^1L^2_{R+|t|}}\lesssim q_1 \epsilon' |\ln \epsilon'|.
\end{equation}
To estimate $II$ we first use \eqref{even:bd:varphi2}, then \eqref{even:id:defuap}, \eqref{even:pointwisea} and \eqref{even:pointwisetildephi} and get for $r>R+|t|$:
$$
|II|\lesssim r^{\frac{N-6}{2}}|v_+||u_{ap}|\lesssim \frac{1}{r^{2}}|v_+| \left(\frac{|\cbf|_{R}R}{r}+ |\tilde c| \indic ( r+|t|\leq 2MR) \right)
$$
Then, by \eqref{even:bdv+L2}, for any $t\in \mathbb R$:
\begin{align*}
& \|II(t)\|_{L^2_{R+|t|}} \lesssim \| \frac{1}{r^2} v_+\|_{L^2_{R+|t|}}  \left(\frac{|\cbf|_{R}R}{R+|t|}+ |\tilde c| \indic ( R+|t|\leq 2MR) \right) \\
\lesssim & \frac{q_2+|\tilde c|\left(|\cbf|_{R}+\tilde c^2(1+\beta |\ln |\tilde c||)+\frac{1}{M^{1/2}}\right)}{R^{1/2}}\left(\frac{|\cbf|_{R}R}{(R+|t|)^{3/2}}+\frac{|\tilde c|}{(R+|t|)^{1/2}}\indic ( R+|t|\leq 2MR)\right)
\end{align*}
so that using \eqref{even:controle:bd:universal5} and \eqref{even:controle:bd:universal6}:
\begin{align}
\nonumber \|II\|_{L^1L^2_{R+|t|}} & \lesssim  \frac{q_2+|\tilde c|\left(|\cbf|_{R}+\tilde c^2(1+\beta |\ln |\tilde c||)+\frac{1}{M^{1/2}}\right)}{R^{1/2}}\left(\frac{|\cbf|_{R}R}{R^{1/2}}+|\tilde c|M^{1/2}R^{1/2} \right) \\
\nonumber &=q_2\left( |\cbf|_{R}+|\tilde c|M^{1/2} \right)+|\tilde c||\cbf|_{R}\left(|\cbf|_{R}+\tilde c^2(1+\beta |\ln |\tilde c||)+\frac{1}{M^{1/2}}\right) \\
\nonumber & \qquad +\tilde c^2\left(M^{1/2}(|\cbf|_{R}+\tilde c^2(1+\beta|\ln |\tilde c||))+1\right)
 \\
 \label{even:controle:tech11} & \lesssim  q_2\left( |\cbf|_{R}+|\tilde c|M^{1/2} \right) +|\tilde c||\cbf|_{R}+\tilde c^2
\end{align}
where we used $|\cbf|_{R}+\tilde c^2(1+\beta |\ln |\tilde c||)+M^{-1/2}\ll 1$ from \eqref{even:bd:cbftildectildeR} and \eqref{even:bd:approxM}, and $M^{1/2}(|\cbf|_{R}+\tilde c^2(1+\beta|\ln |\tilde c||))\lesssim \sqrt{\nu}\lesssim 1$ from \eqref{even:bd:approxM}.

Combining \eqref{even:controle:tech10}, \eqref{even:controle:tech11} and \eqref{even:L1L2E-} we infer that:
$$
 \| \Box v_- \|_{L^1L^2_{R+|t|}}\lesssim q_1 \epsilon'|\ln \epsilon'|+ q_2\left( |\cbf|_{R}+|\tilde c|M^{1/2} \right) +|\tilde c||\cbf|_{R}+\tilde c^2
$$
As $v_-$ is non-radiative for $r>R+|t|$, its radiation profiles satisfy $G_+[v_-](\rho)=G_-[v_-](\rho)=0$ for all $\rho \geq R$. Recalling the definition \eqref{even:controle:id:q1} of $q_1$, then applying the channels of energy estimate \eqref{bd:channels2} using \eqref{even:id:orthogonalitypatv}, and the energy and Strichartz estimate \eqref{Weight_StrichartzR} to $v_-$ we infer:
\begin{align*}
q_1 & =\| v_-\|_{L^\infty \mathcal H_{R+|t|}}+\| r^{\frac{N-6}{4}} v_-\|_{L^2L^4_{R+|t|}} \lesssim  \| \Box v_- \|_{L^1L^2_{R+|t|}} \\
&\qquad \lesssim q_1 \epsilon'|\ln \epsilon'|+ q_2\left( |\cbf|_{R}+|\tilde c|M^{1/2} \right) +|\tilde c||\cbf|_{R}+\tilde c^2
\end{align*}
Therefore for $\epsilon'$ small enough:
\begin{equation}\label{even:controle:bdq1}
q_1\lesssim q_2\left( |\cbf|_{R}+|\tilde c|M^{1/2} \right) +|\tilde c||\cbf|_{R}+\tilde c^2.
\end{equation}

\noindent \textbf{Step 3}. \emph{Control of $v_+$}. The equation satisfied by $v_+$ is using \eqref{even:id:equationu-uap} and $u=u_{ap}+v$:
\begin{align*}
& \Box v_+  = (E+\varphi (u_{ap}+v)-\varphi (u_{ap}))_+
\end{align*}
so that the equation satisfied by $\pa_t v_+$ is, using $\pa_t f_+=(\pa_t f)_-$, $u=u_{ap}+v$, $v=v_++v_-$ and $(\varphi'(v_+)\pa_t v_-)_-=0$ (as a product of even in time functions):
\begin{align*}
 \Box \pa_t  v_+& =\pa_t E_++\left(\varphi'(u_{ap}+v)(\pa_t u_{ap}+\pa_t v)-\varphi'(u_{ap})\pa_t u_{ap}\right)_-\\
&= \pa_t E_++ (\underbrace{\varphi'(u)\pa_t v_+}_{=I})_-+ (\underbrace{(\varphi'(u_{ap}+v_++v_-)-\varphi'(v_+))\pa_t v_-}_{=II})_-\\
& \qquad + \Bigl(\underbrace{\left(\varphi'(u_{ap}+v_++v_-)-\varphi'(u_{ap})\right)\pa_t u_{ap}}_{=III} \Bigr)_-.
\end{align*}
% We fix $\bar R\geq R+|t|$ and estimate all terms. For $I$, using \eqref{even:bd:varphi3} and \eqref{even:bd:aprioriu} we have $|I|\lesssim \epsilon' r^{-2}|\pa_t v_+|$ for $r>R+|t|$. Using this and \eqref{even:bdpatv+L2} shows $\| I(t)\|_{L^2_{\tilde R}}\lesssim \epsilon' q_2 R^{-1/2} \tilde R^{-3/2}$ for $\tilde R> R+|t|$. 
We have by \eqref{even:bd:aprioriu}, \eqref{even:bd:varphi3}, $|I|\leq \epsilon' r^{-2}|\partial_t v_+(t,r)|$. Fix $\widetilde{R}$, $\tilde{t}$ with $\widetilde{R}>R+|\tilde{t}|$, so that $r>\widetilde{R}+|t-\tilde{t}|\Longrightarrow r>R+|t|$. Using \eqref{even:bdpatv+L2} with $\bar{R}=\widetilde{R}+|t-\tilde{t}|\geq R+|t|$, and \eqref{even:controle:bd:universal5}, we then obtain that
\begin{equation} \label{even:controle:tech13}
\| I\|_{L^1L^2_{\widetilde R+|t-\tilde t|}} \lesssim \frac{\epsilon' q_2}{R} \int_{\mathbb R}\frac{d t}{(\widetilde R+|t-\tilde t|)^{3/2}} \lesssim \frac{\epsilon' q_2 }{R^{1/2}\widetilde{R}^{1/2}}.
\end{equation}
For $II$, using \eqref{even:bd:varphi4} we get $|II|\lesssim r^{\frac{N-6}{2}}(|u_{ap}|+|v_-|)|\pa_t v_-|$ for $r>R+|t|$. Combining \eqref{even:bdv-}, \eqref{even:pointwisea} and \eqref{even:pointwisetildephi} shows $ r^{\frac{N-6}{2}}(|u_{ap}|+|v_-|)\lesssim r^{-2}(|\cbf|+|\tilde c|+\tilde q_1)$ for $r>R+|t|$. Hence by the definition \eqref{even:controle:id:q1} of $q_1$, for $t\in \mathbb R$ and $\widetilde R\geq R+|\tilde t|$, $\bar{R}=\widetilde{R}+|t-\tilde{t}|$,
$$
\| II(t)\|_{L^2_{\bar R}}  \lesssim  \|r^{\frac{N-6}{2}}(u_{ap,-}+v_-)\|_{L^\infty_{\bar R}}  \| \pa_t v_-\|_{L^2_{\bar R}}  \lesssim \frac{|\cbf|_{R}+|\tilde c|+\tilde q_1}{\bar R^2} q_1.
$$
Thus, using \eqref{even:controle:bd:universal5}:
\begin{equation}\label{even:controle:tech15}
\| II\|_{L^1L^2_{\bar R+|t-\bar t|}} \lesssim \frac{q_1(|\cbf|_{R}+|\tilde c|+\tilde q_1)}{\widetilde R}\leq \frac{1}{R^{1/2}\widetilde{R}^{1/2}}q_1\left( |c|_R+|\tilde{c}|+\tilde{q}_1 \right).
\end{equation}
For $III$, we first use \eqref{even:bd:varphi4} and obtain for $r>R+|t|$:
\begin{equation}\label{even:controle:tech140}
|III|\lesssim r^{\frac{N-6}{2}}|v_+||\pa_t u_{ap}|+ r^{\frac{N-6}{2}}|v_-||\pa_t u_{ap}|.
\end{equation}
For the first term, using H\"older with \eqref{even:L2patuap} and \eqref{even:bdv+} we obtain for any $\bar R\geq R+|t|$:
\begin{align*}
\| r^{\frac{N-6}{2}} v_+(t)\pa_t u_{ap}(t)\|_{L^2_{\tilde R}} & \lesssim \| r^{\frac{N-4}{2}}v_+(t)\|_{L^\infty_{\tilde R}} \| \frac{1}{r}\pa_t u_{ap}(t)\|_{L^2_{\tilde R}}\\
&\lesssim \frac{q_2+|\tilde c|(|\cbf|_{R}+\tilde c^2(1+\beta |\ln |\tilde c||)+\frac{1}{M^{1/2}})+\epsilon'q_1}{\bar R^{1/2}R^{1/2}} \left(\frac{R|\cbf|_{R}}{\bar R^2}+\frac{|\tilde c|}{\bar R}\right)\\
&\quad \lesssim  \frac{\epsilon' q_2+|\tilde c| (|\cbf|_{R}+|\tilde c|^2(1+\beta |\ln |\tilde c||)+\frac{1}{M^{1/2}}) +\epsilon' (|\cbf|_{R}+|\tilde c|) q_1}{\bar R^{3/2}R^{1/2}}
\end{align*}
where we used \eqref{even:bd:cbftildectildeR} for the last inequality. Hence, using \eqref{even:controle:bd:universal5} if $\widetilde{R}\geq R+|\tilde{t}|$, $\bar{R}=\widetilde{R}+|t-\tilde{t}|$:
\begin{equation}\label{even:controle:tech14}
\| r^{\frac{N-6}{2}} v_+\pa_t u_{ap} \|_{L^1L^2_{\bar R+|t-\bar t|}}\lesssim  \frac{\epsilon' q_2+|\tilde c| (|\cbf|_{R}+|\tilde c|^2(1+\beta |\ln |\tilde c||)+\frac{1}{M^{1/2}}) +\epsilon' (|\cbf|_{R}+|\tilde c|) q_1}{\widetilde R^{1/2}R^{1/2}} .
\end{equation}
For the second term, with \eqref{even:bdv-}, \eqref{even:L2patuap} and $\tilde q_1\lesssim q_1$ we obtain for $\tilde R\geq R+|t|$:
$$
\| r^{\frac{N-6}{2}}v_-(t)\pa_t u_{ap}(t)\|_{L^2_{\bar R}} \lesssim \| r^{\frac{N-4}{2}}v_-(t)\|_{L^\infty_{\bar R}}  \|\frac{1}{r}\pa_t u_{ap,-}(t) \|_{L^2_{\bar R}} \lesssim \frac{q_1(|\cbf|_{R}+|\tilde c|)}{\bar R^2}.
$$
Hence using \eqref{even:controle:bd:universal5}:
\begin{equation}\label{even:controle:tech16}
\| r^{\frac{N-6}{2}}v_-\pa_t u_{ap}\|_{L^1L^2_{\bar R+|t-\bar t|}}\lesssim \frac{q_1(|\cbf|_{R}+|\tilde c|)}{\widetilde R}\leq \frac{q_1(|c_R|+|\tilde{c}|}{\widetilde{R}^{1/2}R^{1/2}}.
\end{equation}
Injecting \eqref{even:controle:tech14} and \eqref{even:controle:tech16} in \eqref{even:controle:tech140} we find:
\begin{equation}\label{even:controle:tech160}
\| III \|_{L^1L^2_{\bar R+|t-\bar t|}}\lesssim  \frac{\epsilon' q_2+|\tilde c| (|\cbf|_{R}+|\tilde c|^2(1+\beta |\ln |\tilde c||)+\frac{1}{M^{1/2}}) +(|\cbf|_{R}+|\tilde c|) q_1}{\widetilde R^{1/2}R^{1/2}} .
\end{equation}

Gathering \eqref{even:L2patE} (using $\|f_+\|_{L^1L^2}\lesssim\| f\|_{L^1L^2}$), \eqref{even:controle:tech13}, \eqref{even:controle:tech15} and \eqref{even:controle:tech160} we obtain:
$$
 \| \Box \pa_t v_+ \|_{W^{'1/2}_{R}}\lesssim  \frac{|\tilde c|\left(|\cbf|_{R}+\tilde c^2(1+\beta| \ln |\tilde c||)+ \frac{1}{M^{1/2}}\right)+\epsilon' q_2+q_1\left(|\cbf|_{R}+|\tilde c|+\tilde q_1\right)}{R^{1/2}}
$$
Recalling the definition \eqref{even:controle:id:q2} of $q_2$, and then using the weighted channels of energy estimate \eqref{bd:weightedchannels} with \eqref{even:id:orthogonalityDeltav} we infer
\begin{align*}
q_2& =R^{1/2}\| \pa_t v_+\|_{\tilde W_{R}^{1/2}} \lesssim  \| \Box \pa_t v_+ \|_{W^{'1/2}_{R}} \\
&\lesssim |\tilde c|\left(|\cbf|_{R}+\tilde c^2(1+\beta| \ln |\tilde c||)+ \frac{1}{M^{1/2}}\right)+\epsilon' q_2+q_1\left(|\cbf|_{R}+|\tilde c|+\tilde q_1\right)
\end{align*}
so that
\begin{equation}\label{even:controle:bdq2}
q_2 \lesssim |\tilde c| \left(|\cbf|_{R}+\tilde c^2(1+\beta| \ln |\tilde c||)+ \frac{1}{M^{1/2}}\right)+q_1\left(|\cbf|_{R}+|\tilde c|+\tilde q_1\right).
\end{equation}

\noindent \textbf{Step 4}. \emph{End of the proof}. We inject \eqref{even:controle:bdq2} in \eqref{even:controle:bdq1}, use \eqref{even:controle:bd:aprioriq1q2} and \eqref{even:bd:cbftildectildeR}, and obtain:
\begin{align*}
q_1 & \lesssim |\tilde c|(|\cbf|_R+|\tilde c|M^{1/2})(|\cbf|_R+\tilde c^2(1+\beta |\ln |\tilde c||)+\frac{1}{M^{1/2}})\\
&\qquad +q_1(|\cbf|_R+|\tilde c| M^{1/2})(|\cbf|_R+|\tilde c|+\tilde q_1)+|\tilde c||\cbf|_{R}+\tilde c^2 \\
& \lesssim  \tilde c^2M^{1/2} (|\cbf|_R+\tilde c^2(1+\beta |\ln |\tilde c||))+q_1(\epsilon'+|\tilde c| M^{1/2}( |\cbf|_R+|\tilde c|+\tilde q_1))+|\tilde c||\cbf|_{R}+\tilde c^2 
\end{align*}
Using \eqref{even:bd:approxM} we have $M^{1/2} (|\cbf|_R+\tilde c^2(1+\beta |\ln |\tilde c||))\lesssim \sqrt{\nu}\lesssim 1$. Using $|\tilde c| ( |\cbf|_R+|\tilde c|+\tilde q_1)\lesssim \tilde c^2+|\cbf|_R^2+\tilde q_1^2$ and \eqref{even:bd:approxM} we have $|\tilde c| M^{1/2}( |\cbf|_R+|\tilde c|+\tilde q_1))\lesssim \sqrt{\nu}$.

Therefore, for $\nu$ small enough, but independently of the other parameters as announced, the above inequality implies:
\begin{equation}\label{even:controle:bdq1final}
q_1 \lesssim |\tilde c||\cbf|_{R}+\tilde c^2.
\end{equation}
Injecting \eqref{even:controle:bdq1final} in \eqref{even:bd:approxM}, using $\tilde q_1\lesssim q_1$ by Sobolev, we find
\begin{equation}\label{even:bd:approxM2}
M \approx \frac{\nu}{|\cbf|_R^2+\tilde c^4(1+\beta |\ln |\tilde c||)^2}.
\end{equation}

Injecting $\tilde q_1\lesssim q_1$, \eqref{even:bd:approxM2} and \eqref{even:controle:bdq1final} in \eqref{even:controle:bdq2} we get:
\begin{equation}\label{even:controle:bdq2final}
q_2 \lesssim |\tilde c|\left( |\cbf|_{R}+\tilde c^2(1+\beta |\ln |\tilde c||)\right).
\end{equation}
The two bounds \eqref{even:controle:bdq1final} and \eqref{even:controle:bdq2final}, because of the very definitions \eqref{even:controle:id:q1} and \eqref{even:controle:id:q2} of $q_1$ and $q_2$, imply the result of the Lemma.

\end{proof}

Thanks to Lemma \ref{even:lem:approximation}, one can estimate how the parameters $\cbf$ and $\tilde c$ depend on $R$:

 \begin{lemma} \label{even:lem:dyadicdifferences}

For all $R_1\leq R_2\leq R_3\leq 10 R_2$, there holds:
\begin{align}
\label{even:bd:cbf2-cbf1} & \cbf[R_3]=\cbf[R_2]+O_{|\cdot |_{R_2}}(|\tilde c[R_2]|(|\tilde c[R_2]|+|\cbf [R_2]|)),\\
\label{even:bd:tildec2-tildec1}& \tilde c[R_3]=\tilde c[R_2]+\beta \ln \frac{R_3}{R_2} |\tilde c[R_2]|\tilde c[R_2]+O\left(|\tilde c[R_2]|\left(|\tilde c[R_2]|^2(1+\beta |\ln |\tilde c[R_2]||)+|\cbf [R_2]|\right)\right),
\end{align}
where, for $R,\eta>0$, above $O_{|\cdot |_{R}}(\eta)$ denotes an element of $\mathbb R^{m_0+1}$ whose $|\cdot|_{R}$ norm is $\lesssim \eta$.

\end{lemma}

\begin{proof}

We write $\cbf_i=\cbf[R_i]$, $\tilde c_i=\tilde c[R_i]$ and $M_i=M[R_i]$ for $i=2,3$. By \eqref{even:id:tidephigeneralised} we have $\tilde \phi [\tilde c_2,R_2]-\tilde \phi[\tilde c_2,R_3]=\beta |\tilde c_2|\tilde c_2 \ln \frac{R_3}{R_2} \phi_{m_0+1}$ (using $\beta \hat c_2^2=\beta |\tilde c_2|\tilde c_2$). We decompose using this, \eqref{even:id:defuap} and \eqref{def:tildephi}:
\begin{align}
\label{even:id:diffcR2-cR1} & a_F[\cbf_3-\cbf_2]+\left(\tilde c_3-\tilde c_2-\beta \tilde c_2^2 \ln \frac{R_3}{R_2}\right)\phi_{m_0+1} \\
\nonumber =& \underbrace{u_{ap}[\cbf_3,\tilde c_3,R_3]-u}_{=f_3}+\underbrace{u-u_{ap}[\cbf_2,\tilde c_2,R_2]}_{=-f_2}+\underbrace{\tilde a[\cbf_2]-\tilde a[\cbf_3]}_{=g}+\underbrace{(\chi_{M_2R_3}-\chi_{M_3R_3}) \tilde \phi[\tilde c_3,R_3]}_{=h_3} \\
\nonumber &+\underbrace{(\chi_{M_2R_2}-\chi_{M_2R_3})\tilde \phi[\tilde c_2,R_3]}_{=-h_2}+\underbrace{\chi_{M_2R_3}\tilde \phi[\tilde c_2,R_3]-\tilde c_2 \phi_{m_0+1}-(\chi_{M_2R_3} \tilde \phi[\tilde c_3,R_3] -\tilde c_3 \phi_{m_0+1})}_{=j}\\
\nonumber &+\underbrace{\beta \tilde c_2^2 \ln \frac{R_3}{R_2} (\chi_{M_2R_2}-1)\phi_{m_0+1}}_{=k} \quad =:F
\end{align}
Recall the definition of the mappings $\Psi $ and $\Phi$ in the proof of Lemma \ref{lem:defctildec}. We have by Step 1 in the proof of Lemma \ref{lem:defctildec} that \eqref{even:id:diffcR2-cR1} implies:
\begin{equation}\label{even:id:diffcR2-cR12}
\left(\cbf_3-\cbf_2-\beta \tilde c_2^2 \ln \frac{R_3}{R_2},\tilde c_3-\tilde c_2\right)=(\Psi\circ \Phi)^{-1}(\Psi (F) ).
\end{equation}

Let $Q=|\cbf_2|_{R_3}+|\cbf_3|_{R_3}+|\tilde c_2|+|\tilde c_3|$, $Q_1=(|\tilde c_2|+|\tilde c_3|)(|\cbf_2|_{R_3}+|\cbf_3|_{R_3}+|\tilde c_2|+|\tilde c_3|)$ and $Q_2=(|\tilde c_2|+|\tilde c_3|)(|\cbf_2|_{R_3}+|\cbf_3|_{R_3}+\tilde c_2^2(1+\beta |\ln |\tilde c_2||)+\tilde c_3^2(1+\beta |\ln |\tilde c_3||))$. We estimate at $r=R_3$ (so that $r>R_2$) the first two terms $f_2$ and $f_3$ via \eqref{even:bd:u-uap3}, the second $g$ by \eqref{even:bd:linearctildectech5}, and the last $j$ by \eqref{even:bd:linearctildectech2}, yielding:
\begin{align}
&\nonumber |(f_3-f_2+g+j)(0,R_3)| \\
\nonumber \lesssim  &\frac{|\tilde c_3|(|\cbf_3|_{R_3}+\tilde c_3^2(1+\beta |\ln |\tilde c_3||))+|\tilde c_2|(|\cbf_2|_{R_2}+\tilde c_2^2(1+\beta |\ln |\tilde c_2||))+|\cbf_3-\cbf_2|_{R_3}(|\cbf_2|_{R_3}+|\cbf_3|_{R_3})}{R^{N/2-1}_2}\\
 \label{even:bd:diffcR2-cR1:tech1} &\lesssim\frac{Q_2+Q|\cbf_3-\cbf_2|_{R_3}}{R^{N/2-1}_2} 
\end{align}
where we used $R_2\approx R_3$ so that $|\cdot|_{R_2}\approx |\cdot|_{R_3}$. Similarly, using \eqref{even:bd:u-uap1}, \eqref{even:bd:u-uap2}, \eqref{even:bd:linearctildectech5}, \eqref{even:bd:linearctildectech3} and \eqref{even:bd:linearctildectech4} (with \eqref{even:bd:approxM2} so that $M_i^{-1}\lesssim |\cbf_i|_{R_3}^2+|\tilde c_i|^4(1+\beta |\ln |\tilde c_i||)^2$ for $i=2,3$)
\begin{align}
\label{even:bd:diffcR2-cR1:tech2} & \| \pa_t(f_3-f_2+g+j)(0)\|_{L^2_{R_3}}\lesssim Q_1+Q|\cbf_3-\cbf_2|_{R_3} ,\\
\label{even:bd:diffcR2-cR1:tech3}& \| \pa_{tt}(f_3-f_2+g+j)(0)\|_{L^2_{R_3}}\lesssim \frac{Q_2+Q(|\cbf_3-\cbf_2|_{R_3}+|\tilde c_3-\tilde c_2|)}{R_3} .
\end{align}
For $h_i$ for $i=2,3$ and $k$, we have that $\textup{supp}(\chi_{M_2R_3}-\chi_{M_iR_3})\subset \{r+|t|\geq R_2 \min(M_2,M_3) \}$ and $\textup{supp}(\chi_{M_2R_2}-1)\subset \{r+|t|\geq M_2R_2 \}$, and that $\chi$, $\phi_{m_0+1}$ and $\tilde \phi$ are even in time. Hence:
\begin{equation}\label{even:bd:diffcR2-cR1:tech4}
(h_3-h_2+k)(0,R_3)=0 \quad \mbox{and}\quad \pa_t (h_3-h_2+k)(0,r)=0 \quad \forall r>0,
\end{equation}
and by a direct computation using \eqref{def:tildephi} and \eqref{id:defphim0+1}:
\begin{align}
\nonumber \| \pa_{tt}(h_3-h_2+k)(0)\|_{L^2_{R_3}}^2 & \lesssim  \int_{\min(M_2,M_3)R_2}^\infty \frac{\tilde c_1^2+\tilde c_2^2}{(r^{\frac N2+1})^2} r^{N-1}dr\\
\label{even:bd:diffcR2-cR1:tech5} & \lesssim \frac{(\tilde c_1^2+\tilde c_2^2)\max (M_2^{-2},M_3^{-2})}{R_2^2}\lesssim \frac{Q_2^2}{R_3^2}.
\end{align}
where we used $M_2^{-1}+M_3^{-1}\lesssim (|\cbf_2|^2+|\cbf_3|^2+|\tilde c_2|^4(1+\beta |\ln |\tilde c_2||)^2+|\tilde c_3|^4(1+\beta |\ln |\tilde c_3||)^2)$ by \eqref{even:bd:approxM2}. Injecting \eqref{even:bd:diffcR2-cR1:tech1}, \eqref{even:bd:diffcR2-cR1:tech2}, \eqref{even:bd:diffcR2-cR1:tech3}, \eqref{even:bd:diffcR2-cR1:tech4} and \eqref{even:bd:diffcR2-cR1:tech5} in \eqref{even:id:diffcR2-cR1} one finds:
\begin{align*}
& |F(0,R_3)|\lesssim \frac{Q_2+Q|\cbf_3-\cbf_2|_{R_3}}{R^{N/2-1}_2} ,\\
&\|\pa_tF(0)\|_{L^2_{R_3}}\lesssim  Q_1+Q|\cbf_3-\cbf_2|_{R_3} \\
& \|\pa_{tt}F(0)\|_{L^2_{R_3}}\lesssim \frac{ Q_2+Q(|\cbf_3-\cbf_2|_{R_3}+|\tilde c_3-\tilde c_2|)}{R_3}.
\end{align*}
Using the identity \eqref{even:id:diffcR2-cR12}, and injecting the three above estimates in the right-hand sides of \eqref{even:bd:linearctildec} and \eqref{even:bd:linearctildec2} with $R=R_3$ yields:
\begin{align*}
& |\cbf_3-\cbf_2|_{R_3}\lesssim Q_1+Q(|\cbf_3-\cbf_2|_{R_3}+|\tilde c_3-\tilde c_2|), \\
& \left|\tilde c_3-\tilde c_2 -\beta \ln \frac{R_3}{R_2}|\tilde c_2|\tilde c_2\right| \lesssim Q_2+Q(|\cbf_3-\cbf_2|_{R_3}+|\tilde c_3-\tilde c_2|).
\end{align*}
and hence since $QQ_1\lesssim Q_2$ and $Q\lesssim \epsilon'$ by \eqref{even:bd:cbftildectildeR}, we get $|\cbf_3-\cbf_2|_{R_3}\lesssim Q_1$ and $|\tilde c_3-\tilde c_2|\lesssim Q_2$. In turn, injecting these two inequalities in the definitions of $Q_1$ and $Q_2$ shows $Q_1\approx |\tilde c_2|(|\cbf_2|_{R_3}+|\tilde c_2|)$ and $Q_2\approx |\tilde c_2|(|\cbf_2|_{R_3}+\tilde c_2^2(1+\beta |\ln |\tilde c_2||))$. This eventually shows \eqref{even:bd:cbf2-cbf1} and \eqref{even:bd:tildec2-tildec1}

\end{proof}

\begin{lemma} \label{lem:criticalnondispersivestabilitycontrad}

There holds $\tilde c[R_1]=0$. In particular, as a consequence of Lemma \ref{even:lem:approximation} and of finite speed of propagation we get that for all $|x|>R_1+|t|$:
$$
u(t,x)=a[\cbf(R_1),R_1](t,x).
$$

\end{lemma}

\begin{proof}

Our aim is to interpret Lemma \ref{even:lem:dyadicdifferences} as a difference inequality for the energy in dyadic zones. We shall then show that if $\tilde c[R_1]\neq 0$ then the initial data has infinite energy, which by contradiction shows $\tilde c[R_1]=0$.

We introduce $R_k=2^{3+k}R_0$ for $k\geq 1$ and write $\cbf_k=\cbf (R_k)$ and $\tilde c_k=\tilde c(R_k)$.\\

\noindent \textbf{Step 1.} \emph{An $\ell^2$ bound from energy considerations}. We claim that $|\cbf_k|_{R_k},\tilde c_k\in \ell^2$ with
\begin{equation}\label{even:bd:aprioricktildeckl2}
\sum_{k\geq 1}^\infty |\cbf_k|_{R_k}^2+|\tilde c_k|^2 \lesssim \| \vec u(0)\|_{\mathcal H_{16R_0}}^2\lesssim \epsilon^{'2}.
\end{equation}
To prove it,  for each $k\geq 1$ we let 
\be \label{even:id:decompositioninitialdata}
\sum_{0\leq l\leq l_0} c_{2l}[R_k] r^{2l+2-N}+\tilde c_k r^{1-N/2}=u_0+\tilde u_0+\bar u_0
\ee
where:
$$
\tilde u_0=\sum_{0\leq l\leq l_0}\frac{c_{2l}[R_k]}{r^{N-2l-2}}+\frac{\tilde c_k}{r^{N/2-1}}-u_{ap}[\cbf_k,\tilde c_k,R_k](0),\quad \bar u_0=u_{ap}[\cbf_k,\tilde c_k,R_k](0)-u_0.
$$
Combining \eqref{even:id:defuap}, \eqref{id:aF0} and \eqref{eq:vectildephi0} we find $\tilde u_0=-\tilde a[\cbf_k,R_k](0)-\beta |\tilde c_k|\tilde c_k \ln (r/R_k) r^{1-N/2}$, so that, by using \eqref{bd:tildeaWkappaN} and the Hardy inequality, and $\| r^{1-N/2}\|_{L^2(R_k\leq r\leq R_{k+1})}\lesssim R_k$ we have:
\begin{equation}\label{even:bd:tildeu0dyadic}
\| \tilde u_0\|_{L^2(R_k\leq r\leq R_{k+1})}\lesssim R_k(|\cbf_k|_{R_k}^2+\tilde c_k^2).
\end{equation}
Using \eqref{even:bd:u-uap3} and $\| r^{3/2-N/2}\|_{L^2(R_k\leq r\leq R_{k+1})}\lesssim R_k^{3/2}$ we have:
\begin{equation}\label{even:bd:baru0dyadic}
\| \bar u_0\|_{L^2(R_k\leq r\leq R_{k+1})}\lesssim R_k|\tilde c_k|(|\cbf_k|_{R_k}+\tilde c_k^2(1+\beta |\ln \tilde |c_k||)).
\end{equation}
Injecting \eqref{even:bd:tildeu0dyadic} and \eqref{even:bd:baru0dyadic} in \eqref{even:id:decompositioninitialdata} we infer that:
\begin{equation} \label{pingouin}
\| \sum_{0\leq l\leq l_0} c_{2l}[R_k] r^{2l+2-N}+\tilde c_k r^{1-N/2} \|_{L^2(R_k\leq r\leq R_{k+1})} \lesssim  \| u_0 \|_{L^2(R_k\leq r\leq R_{k+1})} +R_k(|\cbf_k|_{R_k}+|\tilde c_k|)^2.
\end{equation}
By rescaling, and using that $(r^{2-N},...,r^{2l_0+2-N},r^{1-N/2})$ are linearly independent, we have:
\begin{align}
\label{marsupial} & \| \sum_{0\leq l\leq l_0} c_{2l}[R_k] r^{2l+2-N}+\tilde c_k r^{1-N/2} \|_{L^2(R_k\leq |x|\leq R_{k+1})} \\
\nonumber &=\| \sum_{0\leq l\leq l_0} R_k^{2l+2-\frac N2}c_{2l}[R_k] r^{2l+2-N}+R_k \tilde c_k r^{1-\frac N2} \|_{L^2(1\leq |x|\leq 2)} \approx \sum_{0\leq l \leq l_0} R_k^{2l+2-\frac N2} |c_{2l}[R_k] |+R_k |\tilde c_k|.
\end{align}
Combining \eqref{pingouin} and \eqref{marsupial} we get:
\begin{equation}\label{even:bd:energyu0byctildec}
\sum_{0\leq l \leq l_0} R_k^{2l+2-N/2} |c_{2l}[R_k] |+R_k |\tilde c_k| \lesssim \| u_0 \|_{L^2(R_k\leq |x|\leq R_{k+1})} +R_k(|\cbf_k|_{R_k}+|\tilde c_k|)^2
\end{equation}
Very similarly, using \eqref{even:bd:u-uap1}, \eqref{even:id:defuap}, \eqref{id:aF0}, \eqref{eq:vectildephi0} and \eqref{bd:tildeaWkappaN} we get
\begin{equation}\label{even:bd:energyu0byctildec2}
 \sum_{0\leq l \leq l_1} R_k^{2l+2-N/2} |c_{2l+1}[R_k] |\lesssim \| u_1 \|_{L^2(R_k\leq |x|\leq R_{k+1})} +(|\cbf_k|_{R_k}+|\tilde c_k|)^2.
\end{equation}
Recalling the definition of the $|\cdot |_{R_k}$ norm, then combining \eqref{even:bd:energyu0byctildec} and \eqref{even:bd:energyu0byctildec2} we obtain:
\begin{align} \label{even:bd:energyu0byctildec3}
|\cbf_k|_{R_k}+|\tilde c_k| & \approx \sum_{0\leq l \leq l_0} R_k^{2l+1-N/2} |c_{2l}[R_k] |+ \sum_{0\leq l \leq l_1} R_k^{2l+2-N/2} |c_{2l+1}[R_k] |+|\tilde c_k| \\
\nonumber & \lesssim \| \frac 1r u_0\|_{L^2(R_k\leq |x|\leq R_{k+1})}+\| u_1\|_{L^2(R_k\leq |x|\leq R_{k+1})}+(|\cbf_k|_{R_k}+|\tilde c_k|)^2
\end{align}
Hence $|\cbf_k|_{R_k}+|\tilde c_k| \lesssim  \|  u_0/r\|_{L^2(R_k\leq |x|\leq R_{k+1})}+\| u_1\|_{L^2(R_k\leq |x|\leq R_{k+1})}$ by \eqref{even:bd:cbftildectildeR}. Squaring, summing over $k$ and using the Hardy inequality shows \eqref{even:bd:aprioricktildeckl2}.\\

\noindent \textbf{Step 2.} \emph{An improved $\ell^1$ bound for $|\cbf_k|_{R_k}$}. We claim that $|c_k|_{R_k}\in \ell^1$ with
\begin{equation}\label{even:bd:apriorickl1}
\sum_{k\geq 1}^\infty |\cbf_k|_{R_k} \lesssim \epsilon^{'}.
\end{equation}
Let $k\geq 1$. We notice by \eqref{even:bd:cbf2-cbf1} and \eqref{even:bd:cbftildectildeR} that $|\cbf_{k+1}|_{R_{k+1}}=|\cbf_{k}|_{R_{k+1}}(1+O(\epsilon'))+O(|\tilde c_k|^2)$. We have $|\cbf|_{R_{k+1}}\leq  |\cbf |_{R_{k}}/2$ for any $\cbf\in \mathbb R^{m_0+1}$ by definition of the $|\cdot |_{R}$ norm and since $R_{k+1}=R_k/2$. Thus for $\epsilon'$ small enough: 
$$
|\cbf_{k+1}|_{R_{k+1}}\leq \frac{2}{3} |\cbf_{k}|_{R_{k}}+C |\tilde c_k|^2
$$
for some $C>0$. Therefore, $|\cbf_{k}|_{R_{k}}\leq (\frac 23)^{k-1} |\cbf_{1}|_{R_{1}}+C \sum_{j=1}^{k-1} (\frac 23)^{k-j-1} |\tilde c_j|^2$, so that $\sum_{k\geq 1}|\cbf_{k}|_{R_{k}}\lesssim |\cbf_{1}|_{R_{1}}+\sum_{k\geq 1} |\tilde c_k|^2$. Using \eqref{even:bd:aprioricktildeckl2} this shows \eqref{even:bd:apriorickl1}.\\

\noindent \textbf{Step 3.} \emph{A dichotomy result for a recurrence relation for sequences}. We claim the following result: for any non-negative sequences $(x_n)_{n\geq 1}$ and $(y_n)_{n\geq 1}$ such that $\sum_1^\infty y_n<1$ and $x_{n+1}\geq (1-y_n) x_n$ for all $n\geq 1$, then either $x_1=0$ or $x_n\geq c>0$ for all $n\geq 1$ for some $c>0$.

Indeed, assume $x_1>0$ and let $z_n=\ln x_n$. Then for all $n\geq 1$ we have $z_{n+1}\geq z_n+\ln (1-y_n)$ so that $z_n\geq z_1+\sum_0^{n-1} \ln (1-y_k)$. As $y_n$ is non-negative with $\sum_1^\infty y_n<1$ we have that $\sum_1^{n-1} \ln (1-y_k)$ converges as $n\to \infty$. Hence $z_{n}$ is bounded from below, establishing the result.\\

\noindent \textbf{Step 4.} \emph{End of the proof}. By \eqref{even:bd:tildec2-tildec1}, if $\beta =0$, we have that for all $k\geq 1$:
$$
|\tilde c_{k+1}|=\left|\tilde c_k+O(|\tilde c_k|(|\cbf_k|_{R_k}+|\tilde c_k|^2)) \right| \geq |\tilde c_k| \left(1-C(|\cbf_k|_{R_k}+|\tilde c_k|^2) \right),
$$
for some universal constant $C>0$, and if $\beta>0$, using \eqref{even:bd:cbftildectildeR}:
$$
|\tilde c_{k+1}|=\left|\tilde c_k+\beta \ln 2 |\tilde c_k|\tilde c_k+O(|\tilde c_k|(|\cbf_k|_{R_k}+|\tilde c_k|^2(1+\beta |\ln |\tilde c_k||))) \right| \geq |\tilde c_k| \left(1-C(|\cbf_k|_{R_k}) \right).
$$
In both cases, applying the result of Step 3 to $x_k=|\tilde c_k|$ and $y_k=C(|\cbf_k|_{R_k}+|\tilde c_k|^2)$, using \eqref{even:bd:aprioricktildeckl2} and \eqref{even:bd:apriorickl1}, we get that for $\epsilon'$ small enough there must hold $\tilde c_1=0$. This ends the proof of the Lemma.

\end{proof}

\subsection{Proof of Proposition \ref{pr:uniquenesseven} in the case $N\equiv 4\mod 4$} \label{subsec:N=4}

We fix $N\equiv 4\mod 4$ and prove Proposition \ref{pr:uniquenesseven} in that case. The proof is mostly similar to that in the case $N\equiv 6\mod 4$ performed in Subsection \ref{subsec:N=6}. There are two major simplifications. First, the nonlinear resonance $\tilde \phi=\tilde{c}\phi_{m_0+1}+\tilde{c}^2\psi$ is such that the correction $\psi $ does not involve logarithmic growth, see \eqref{even:id:tidephigeneralised2}. This is a consequence of a simple orthogonality condition: $\phi_{m_0+1}$ is odd in time, while $\phi_{m_0+1}^2$ is even, see Section \ref{sec:resonance}. Second, the nonlinearity $\varphi$ is $\mathcal C^2$. This permits us to differentiate twice in time equation \eqref{eq:nonlinearwaveintro}, allowing a much simpler proof of the main approximation Lemma \ref{evenN=4:lem:approximation}, in comparison with that of the analogue Lemma \ref{even:lem:approximation} for $N\equiv 6\mod 4$. Therefore, in what follows we shall only give proofs when they differ from Subsection \ref{subsec:N=6}.

We fix $u$ a solution to \eqref{eq:nonlinearwaveintro} that is non-radiative for $|x|>R_0+|t|$ with $\int_{|x|> R} |\nabla_{t,x}u(0)|^2dx\leq \epsilon^{'2}$. We keep the notation of Subsection \ref{subsec:N=6}. We have a first gain of regularity. We let $R_1=16R_0$, $\varphi'=\frac{\partial \varphi}{\partial u}$.

\begin{lemma}[Gain of regularity] \label{evenN=4:lem:gainreg}

The functions $\indic (r> R_1+|t|)\pa_r\pa_t u $, $\indic (r>R_1+|t|)\pa_{tt} u$ and $\indic (r>R_1+|t|)\Delta u $ all belong to $ \mathcal C^0L^2 \cap L^\infty L^2$. In addition, $\pa_t u$ is a non-radiative solution for $r>R_1+|t|$ to
\begin{equation}\label{evenN=4:id:Boxpatu}
\left\{ \begin{array}{l l} \Box \pa_t u=\varphi'(u)\pa_t u,\\ (\pa_t u(0),\pa_{tt}u(0))=(u_1,\Delta u_0+\varphi(u_0),\end{array} \right. 
\end{equation}
and $\varphi'(u)\pa_t u\in L^1L^2_{R_1+|t|}$. It satisfies:
\begin{equation}\label{evenN=4:aprioripattuL21}
\forall R >R_1, \qquad \| \pa_{t} u(0)\|_{\dot H^1_{R}}+\| \pa_{tt} u(0)\|_{L^2_{R}}\lesssim \frac{\epsilon' }{ R}.
\end{equation}

\end{lemma}

\begin{proof}

The proof is exactly that of Lemma \ref{even:lem:gainreg}, which is valid regardless of the dimension $N\geq 3$.

\end{proof}

\begin{lemma} \label{evenN=4:lem:gainreg2}

For all $R>R_1+|t|$ there holds
\begin{equation}\label{evenN=4:aprioripatudotH1}
\forall R >R_1+|t|, \qquad \| \pa_{t} u(t)\|_{\dot H^1_{R}}\lesssim \frac{\epsilon' }{ R}.
\end{equation}

\end{lemma}

As a corollary, by \eqref{evenN=4:aprioripatudotH1} and the radial Sobolev embedding, we obtain for all $r>R_1+|t|$:
\begin{equation}\label{evenN=4:bd:patu}
|\pa_t u(t,r)|\lesssim \frac{\epsilon'}{r^{N/2}}
\end{equation}

\begin{proof}

The proof is very similar to the proof of Lemma \ref{even:lem:gainreg2}. Let $t\in \mathbb R$ and $R>R_1+|t|$. Let $v(\bar t)=\pa_t u(t+\bar t)$. We have $v(0)=\pa_t u(t)$. By Lemma \ref{evenN=4:lem:gainreg}, $v$ is a non-radiative solution for $r>R+|\bar t|$ of $(\pa_{\bar t}^2-\Delta) v=\varphi'(u(t+\bar t))\pa_t u(t+\bar t)$. Applying the channels of energy estimate \eqref{bd:channels2} to $v$, using \eqref{even:gainreg:bd:upatuL1L2} (which is valid regardless of the dimension $N\geq 3$), we obtain:
\begin{equation}\label{evenN=4:regdecay:bd:tech2}
\| \Pi_{\dot H^1,R}^\perp \pa_{t} u(t)\|_{L^2_{R}}=\| \Pi_{\dot H^1,R}^\perp  v(0)\|_{L^2_{R}}\lesssim \| \varphi'(u(t+\cdot))\pa_t u(t+\cdot )\|_{L^1L^2_{R+|\bar t-t|}}\lesssim \frac{\epsilon'}{R}.
\end{equation}
This was the major difference between the current proof and that of Lemma \ref{even:lem:gainreg2}: that the application of the channels of energy estimate \eqref{bd:channels2} differs in dimensions $N\equiv 4\mod 4$ and $N\equiv 6\mod 4$. The rest of the proof, showing \eqref{evenN=4:aprioripatudotH1} using \eqref{evenN=4:aprioripattuL21} and \eqref{evenN=4:regdecay:bd:tech2}, is the same.

\end{proof}

We have a second gain of regularity, relying on the fact that $\varphi$ is $\mathcal C^2$. The following two Lemmas \ref{evenN=4:lem:gainreg3} and \ref{evenN=4:lem:gainreg4} are new in comparison with Subsection \ref{subsec:N=6} for $N\equiv 6\mod 4$. They will greatly simplify the proof of the forthcoming approximation Lemma \ref{evenN=4:lem:approximation}. Let $R_2=256R_0$ and $\varphi''=\frac{\pa^2}{\pa u^2}\varphi(u)$.

\begin{lemma}[Second gain of regularity] \label{evenN=4:lem:gainreg3}

The functions $\indic (r> R_2+|t|)\pa_r^{j}\pa_t^{3-j} u $ for $j=0,1,2,3$ all belong to $ \mathcal CL^2 \cap L^\infty L^2$. In addition, $\pa_{tt} u$ is a non-radiative solution for $r>R_2+|t|$ to
\begin{equation} \label{evenN=4:id:Boxpattu}
\left\{ \begin{array}{l l} \Box \pa_{tt} u=\varphi''(u)(\pa_t u)^2+\varphi'(u)\pa_{tt}u,\\ (\pa_{tt} u(0),\pa_{ttt}u(0))=(\Delta u_0+\varphi(u_0),\Delta u_1+\varphi'(u_0)u_1),\end{array} \right. 
\end{equation}
and above $\varphi''(u)(\pa_t u)^2+\varphi'(u)\pa_{tt}u\in L^1L^2_{R_2+|t|}$. It satisfies:
\begin{equation}\label{evenN=4:aprioripattuL212}
\forall R >R_2, \qquad \| \pa_{tt} u(0)\|_{\dot H^1_{R}}+\| \pa_{ttt} u(0)\|_{L^2_{R}}\lesssim \frac{\epsilon' }{ R^2}.
\end{equation}
\end{lemma}

\begin{proof}

The proof is very similar to that of Lemma \ref{even:lem:gainreg}. By Lemma \eqref{evenN=4:lem:gainreg}, $\pa_t u$ is a non-radiative solution of $ \Box \pa_t u=\varphi'(u)\pa_t u$. Let $R\geq R_2$ and consider $f=\chi(r-2R-|t|)\varphi'(u)\pa_t u$, where $\chi$ is a smooth one-dimensional cut-off function with $\chi(s)=0$ for $s\leq -R$ and $\chi(s)=1$ for $s\geq 0$. We claim that $\pa_t f\in L^1L^2$ with
\begin{equation}\label{evenN=4:gainreg2:bdpatf}
\| \pa_t f\|_{L^1L^2}\lesssim \frac{\epsilon^{'2}}{R^{2}}.
\end{equation}
Admitting the estimate \eqref{evenN=4:gainreg2:bdpatf}, the rest of proof of Lemma \ref{evenN=4:lem:gainreg2} then follows the same lines as that of Lemma \ref{even:lem:gainreg} (up to replacing $u$ and $f$ there by $\pa_t u$ and the current definition of $f$ respectively).

There remains to prove \eqref{evenN=4:gainreg2:bdpatf}. Note that $\varphi$, given either by \eqref{eq:nonlinearityanalytic} or \eqref{eq:nonlinearitypower} and $N=4$, is a $C^2$ function of the variable $u$ for $|u|\ll r^{1-N/2}$, with
\be
\label{even:bd:varphi5}  |\varphi''(u)|\lesssim r^{\frac{N-6}{2}}.
\end{equation}
Also, $\pa_t u$ is differentiable in time by Lemma \ref{evenN=4:lem:gainreg}. Hence $f$ is differentiable in time for $t\neq 0$ with $\pa_t f=f_1+f_2$ where
$$
 f_1=-\textup{sgn}(t)\chi'(r-|t|-2R)\varphi'(u)\pa_t u, \ \ f_2= \chi(r-2R-|t|)\left(\varphi''(u)(\pa_{t}u)^2+\varphi'(u) \pa_{tt}u\right).
$$
We estimate using the support property of $\chi$, \eqref{even:bd:varphi3}, \eqref{even:bd:aprioriu} and \eqref{evenN=4:bd:patu} that $| f_1|\lesssim \epsilon^{'2}(R+|t|)^{-N/2-2}\indic(R+|t|\leq r\leq 2R+|t|)$. Hence by direct estimates using \eqref{even:controle:bd:universal5}:
\begin{equation}\label{evenN=4:gainreg2:bdpatftech1}
\| f_1\|_{L^1L^2_{R+|t|}}\lesssim \frac{\epsilon^{'2}}{R^2}.
\end{equation}
By \eqref{even:bd:varphi5}, \eqref{evenN=4:bd:patu}, \eqref{even:bd:varphi3} and \eqref{even:bd:aprioriu} we have for $r>R_2+|t|$:
\begin{equation}\label{evenN=4:gainreg2:bdpatftech4}
\left|\varphi''(u)(\pa_{t}u)^2+\varphi'(u) \pa_{tt}u\right|\lesssim \frac{\epsilon^{'2}}{r^{N/2+3}}+\frac{\epsilon'|\pa_{tt}u|}{r^2}.
\end{equation}
Applying standard energy estimates to \eqref{evenN=4:id:Boxpatu}, using \eqref{evenN=4:aprioripattuL21} and \eqref{even:gainreg:bd:upatuL1L2} (which is valid regardless of the dimension $N\geq 3$) we obtain for all $t\in \mathbb R$:
\begin{equation}\label{evenN=4:gainreg2:bdpatftech2}
\| \pa_{tt} u(t)\|_{L^2_{R+|t|}}\lesssim \frac{\epsilon'}{R}.
\end{equation}
Hence by \eqref{even:controle:bd:universal2},  \eqref{even:controle:bd:universal5} and \eqref{evenN=4:gainreg2:bdpatftech2}:
\begin{equation}\label{evenN=4:gainreg2:bdpatftech3}
\| \varphi''(u)(\pa_{t}u)^2+\varphi'(u) \pa_{tt}u \|_{L^1L^2_{R+|t|}}\lesssim \frac{\epsilon^{'2}}{R^2}+\epsilon'\| \frac{1}{r^2}\|_{L^1L^\infty_{R+|t|}}  \| \pa_{tt} u\|_{L^\infty L^2}  \lesssim \frac{\epsilon^{'2}}{R^2}.
\end{equation}
Combining \eqref{evenN=4:gainreg2:bdpatftech1} and \eqref{evenN=4:gainreg2:bdpatftech3}, using the support properties of $\chi$, shows the desired estimate \eqref{evenN=4:gainreg2:bdpatf}.

\end{proof}

\begin{lemma} \label{evenN=4:lem:gainreg4}

For all $R >R_2+|t|$ there holds:
\begin{equation}\label{evenN=4:aprioripattudotH1}
\| \pa_{tt} u(t)\|_{\dot H^1_{R}}\lesssim \frac{\epsilon' }{ R^2}.
\end{equation}
Consequently, by the radial Sobolev embedding, for $r\geq R_2+|t|$:
\begin{equation}\label{evenN=4:bd:pattu}
|\pa_{tt} u(t,r)|\lesssim \frac{\epsilon'}{r^{N/2+1}}.
\end{equation}

\end{lemma}

\begin{proof}

We first treat the case $R>R_2+2| t|$. Applying standard energy estimates to \eqref{evenN=4:id:Boxpattu}, using \eqref{evenN=4:aprioripattuL212}, \eqref{evenN=4:gainreg2:bdpatftech3} and finite speed of propagation:
$$
\| \pa_{tt}u(t)\|_{\dot H^1_R}\lesssim \| (\pa_{tt} u(0),\pa_{ttt}u(0))\|_{\mathcal H_{R-|t|}}+\| \varphi''(u)(\pa_{t}u)^2+\varphi'(u) \pa_{tt}u\|_{L^1L^2_{R-|t|+|t'|}}\lesssim \frac{\epsilon'}{R^2}
$$ 
where we used that $R\approx R-|t|$. This shows \eqref{evenN=4:aprioripattudotH1} in that case. We now turn to the second case and let $R_2+|t|\leq R\leq R_2+2| t|$ and aim at estimating
\begin{equation}\label{evenN=4:gainreg4:defq}
q=\| \pa_{tt} u(t)\|_{\dot H^1_R}.
\end{equation}

\noindent \textbf{Step 1}. \emph{An priori estimate}. We claim that for all $t'\in \mathbb R$:
\begin{equation}\label{evenN=4:gainreg4:tech3}
\| \pa_{tt} u(t+t')\|_{L^2}\lesssim Rq+\frac{\epsilon'}{R}.
\end{equation}
For $R\leq r \leq R_2+2|t|$ we have $|\pa_{tt}u(t,r)|\lesssim qr^{1-N/2}$ by \eqref{evenN=4:gainreg4:defq} and Sobolev. For $r>R_2+2|t|$ we have $|\pa_{tt}u(t,r)|\lesssim \epsilon' r^{-1-N/2}$ from the estimate \eqref{evenN=4:aprioripattudotH1} that we have already proved for $R\geq R_2+2|t|$ and Sobolev. Combining, and using $R\approx R_2+|t|$, we obtain $\| \pa_{tt}u(t)\|_{L^2_R}\lesssim Rq+\epsilon'R^{-1}$. Using \eqref{evenN=4:aprioripatudotH1} this gives:
\begin{equation}\label{evenN=4:gainreg4:tech1}
\| (\pa_tu(t),\pa_{tt}u(t))\|_{\mathcal H_{R}}\lesssim Rq+\frac{\epsilon'}{R}.
\end{equation}
Let $v(t')=\pa_{t}u(t+t')$. Then $v$ solves $\Box v(t')=\varphi'(u(t+t'))\pa_t u(t+t')$ with $\vec v(0)=(\pa_tu(t),\pa_{tt}u(t))$. By \eqref{even:bd:varphi3}, \eqref{even:bd:aprioriu} and \eqref{evenN=4:bd:patu} we have $|\varphi'(u(t+t'))\pa_t u(t+t')|\lesssim \epsilon^{'2}r^{-N/2-2}$ for $r>R+|t|$. Hence
\begin{equation}\label{evenN=4:gainreg4:tech2}
\|\varphi'(u(t+t'))\pa_t u(t+t')\|_{L^1L^2_{R+|t'|}}\lesssim \frac{\epsilon^{'2}}{R}
\end{equation}
using \eqref{even:controle:bd:universal2}. Applying standard energy estimates to $v$, using finite speed of propagation, \eqref{evenN=4:gainreg4:tech1} and \eqref{evenN=4:gainreg4:tech2}, we get $\| \pa_t v(t')\|_{L^2_{R+|t'|}}\lesssim Rq+\epsilon' R^{-1}$ for all $t'\in \mathbb R$. As $\pa_t v(t')=\pa_{tt}u(t+t')$ this is the desired estimate \eqref{evenN=4:gainreg4:tech3}.\\

\noindent \textbf{Step 2}. \emph{End of the proof}. Let $w=\pa_{t'}v$. Then by Lemma \ref{evenN=4:lem:gainreg3} $w$ is a non-radiative solutions to $\Box w=\varphi''(u(t+t'))(\pa_t u(t+t'))^2+\varphi'(u(t+t'))\pa_{tt}u(t+t')=:f$. We have using \eqref{evenN=4:gainreg2:bdpatftech4}, \eqref{even:controle:bd:universal2} and \eqref{evenN=4:gainreg4:tech3}:
\begin{equation}\label{evenN=4:gainreg4:bdpatftech3}
\| f(t') \|_{L^1L^2_{R+|t'|}}\lesssim \frac{\epsilon^{'2}}{R^2}+\epsilon'\| \frac{1}{r^2}\|_{L^1L^\infty_{R+|t'|}}  \| \pa_{tt} u(t+t')\|_{L^\infty L^2_{R+|t'|}}  \lesssim \frac{\epsilon^{'2}}{R^2}+\epsilon' q.
\end{equation}
Applying the channels of energy estimate \eqref{bd:channels2} to $w$, using \eqref{evenN=4:gainreg2:bdpatftech3}, we obtain:
\begin{equation}\label{even:regdecay4:bd:tech222}
\| \Pi_{\dot H^1,R}^\perp \pa_{tt} u(t)\|_{\dot H^1_{R}}=\| \Pi_{\dot H^1,R}^\perp w(0)\|_{\dot H^1_{R}}\lesssim \| f(t')\|_{L^1L^2_{R+|t'|}}\lesssim \frac{\epsilon^{'2}}{R^2}+\epsilon' q.
\end{equation}
Then using the identity
$$
\Pi_{\dot H^1,R} \pa_{tt} u(t)= \pa_{tt} u(t)-\Pi_{\dot H^1,R}^\perp \pa_{tt} u(t),
$$
the bounds \eqref{even:regdecay:bd:tech2} and \eqref{evenN=4:aprioripattudotH1} for $R=R_2+2|t|$ (that has already been showed) we get $\| \Pi_{\dot H^1,R} \pa_{tt} u(t)\|_{\dot H^1_{R_2+2|t|}}\lesssim \epsilon' R^{-2}+\epsilon'q$. Note that $\| \Pi_{\dot H^1,R} \pa_{tt} u(t)\|_{\dot H^1_{R}}\approx \| \Pi_{\dot H^1,R} \pa_{tt} u(t)\|_{\dot H^1_{R_2+2|t|}}$ since $\Pi_{\dot H^1,R} \pa_{tt} u(t)\in \textup{Span}(r^{-N+2l+2})_{0\leq l \leq l_0}$ and $R\approx R_2+|t|$. Therefore, 
\begin{equation}\label{evenN=4:regdecay:bd:tech4}
\| \Pi_{\dot H^1,R} \pa_{tt} u(t)\|_{\dot H^1_{R}}\lesssim \frac{\epsilon'}{R^2}+\epsilon'q.
\end{equation}
Injecting \eqref{even:regdecay4:bd:tech222} and \eqref{evenN=4:regdecay:bd:tech4} in the definition \eqref{evenN=4:gainreg4:defq} of $q$ implies $q\lesssim \epsilon'R^{-2}+\epsilon'q$. Hence $q\lesssim \epsilon' R^{-2}$, implying the desired result \eqref{evenN=4:aprioripattudotH1} in this second case as well.
\end{proof}

The rest of the Subsection is now similar to Subsection \ref{subsec:N=6} for $N\equiv 6$, to which we refer for explanations. Recall the notation \eqref{even:id:defuap} and \eqref{even:id:defE} for the approximate solution $u_{ap}[\cbf,\tilde c,R]$ and the error $E$ it generates. The following lemma gathers estimates on $u_{ap}$ and $E$.

\begin{lemma} \label{evenN=4:lem:estimatesuap}

For all $R>0$, $\cbf \in \mathbb R^{m_0+1}$, $\tilde c\in \mathbb R$ with $|\cbf|_{R}+|\tilde c|\lesssim \epsilon'$ and $M>1$, $u_{ap}$ is well-defined for $r>R+|t|$ and satisfies the following estimates.
\begin{itemize}
\item \emph{Pointwise estimates}. For all $r>R+|t|$ one has:
\begin{equation}\label{evenN=4:bd:uap}
|u_{ap}|\lesssim \frac{\epsilon'}{r^{N/2-1}},
\end{equation}
\begin{equation}\label{evenN=4:bd:patuap}
|\pa_t u_{ap}|\lesssim \frac{\epsilon'}{r^{N/2}}.
\end{equation}
\item \emph{Averaged estimates in space}. For all $\bar R\geq R$:
\begin{align}
 \label{evenN=4:L2E} \|E(0)\|_{L^2_{\bar R}}\lesssim \frac{|\tilde c|\left( |\cbf|_R+\tilde c^2+\frac{1}{M^{1/2}}\right)}{\bar R^{1/2}R^{1/2}}.
\end{align}
For all $\bar R\geq R+|t|$:
\begin{equation}\label{evenN=4:bd:patuapmathcalH}
\| (\pa_t u_{ap}(t),\pa_{tt}u_{ap}(t))\|_{\mathcal H_{\bar R}} \lesssim \frac{\epsilon'}{\bar R},
\end{equation}
\begin{equation}\label{evenN=4:bd:patE}
\|\pa_t E(t)\|_{L^2_{\bar R}}\lesssim \frac{|\tilde c|(|\cbf|_R+\tilde c^2+\frac{1}{M^{1/2}})}{\bar R^{3/2}R^{1/2}}.
\end{equation}
\item \emph{Averaged estimates in time and space}.
\begin{equation}\label{evenN=4:bd:pattE}
\|\pa_{tt} E\|_{W_R^{'3/2}}\lesssim \frac{|\tilde c|(|\cbf|_R+\tilde c^2+\frac{1}{M^{1/2}})}{R^{1/2}}.
\end{equation}

\end{itemize}

\end{lemma}

\begin{remark}

Notice the following differences between Lemma \ref{evenN=4:lem:estimatesuap} and the analogue Lemma \ref{even:lem:estimatesuap} for $N\equiv 6\mod 4$. Lemma \ref{evenN=4:lem:estimatesuap} contains less estimates, and no time symmetrisation of the form $f_+$ or $f_-$ is performed, due to the simpler proof of the forthcoming Lemma \ref{evenN=4:lem:approximation}. Note that no requirement $\ln M\lesssim |\ln |\tilde c||$ is needed, and that no logarithmic corrections appear in the right-hand sides of \eqref{evenN=4:L2E}, \eqref{evenN=4:bd:patE} and \eqref{evenN=4:bd:pattE}, because the nonlinear resonance $\tilde \phi$ defined by \eqref{def:tildephi} does not contain logarithmic corrections if $N\equiv 4\mod 4$.
 
\end{remark}

\begin{proof}

The proof is done in Appendix \ref{A:approximate}.

\end{proof}

We choose:
\begin{equation}\label{evenN=4:id:defM}
M=\frac{1}{\tilde c^4}.
\end{equation}
Notice that this choice is easier than \eqref{evenN=4:id:defM}, due to a simpler proof of the forthcoming Lemma \ref{evenN=4:lem:approximation}. We approximate $u$ for $|x|>R+|t|$ by $u_{ap}[\cbf,\tilde c,R]$ whose parameters are chosen as follows:

 \begin{lemma} \label{lem:N=4defctildec}

For all $R>R_2$, there exists $\cbf[R]\in \mathbb R^{m_0+1}$ and $\tilde c[R]\in \mathbb R$ with
\begin{equation}\label{evenN=4:bd:cbftildectildeR}
|\cbf[R]|_{R}+|\tilde c[R]|\lesssim \epsilon' \quad \mbox{and}\quad \lim_{R\to \infty} |\cbf[R]|_{R}+|\tilde c[R]|=0
\end{equation}
such that, $u_{ap}=u_{ap}[\cbf[R],\tilde c[R],R]$ satisfies:
\begin{align}
\label{evenN=4:id:orthogonalityatR}& u(0,R)-u_{ap}(0,R)=0,\\
\label{evenN=4:id:orthogonalitypatv}& \Pi_{\mathcal H,R}\left(\pa_t u(0)-\pa_t u_{ap}(0),\pa_{tt} u(0)-\pa_{tt} u_{ap}(0)\right)=0.
\end{align}

 \end{lemma}

\begin{proof}

The proof is very similar to that of Lemma \ref{lem:defctildec}. Let $\Phi(\cbf,\tilde c)=a_{F}[\cbf]+\tilde c \phi_{m_0+1}$. Fix $R>R_2$ and let $\Psi=(\Psi_0,\Psi_1)$ be the map defined by:
$$
\Psi_0(f)=\left(f(0,R),(\int_{r>R}\pa_{tt}f(0,r)r^{2l+1}dr)_{0\leq l\leq l_1})\right), \quad \Psi_2(f)= (\int_{r>R}\pa_r \pa_{t}f(0,r)r^{2l}dr)_{0\leq l\leq l_0}.
$$
By Cauchy-Schwarz, for $R=1$, for any $f$ with $\vec f(0)\in \mathcal H_1$, $\pa_t f(0)\in \dot H^1_R$ and $\pa_{tt} f(0)\in L^2_1$, $\Psi$ is well-defined with
\begin{equation}\label{evenN=4:bd:linearctildectech}
|\Psi_0(f)|\lesssim |f(0,1)|+\| \pa_{tt}f(0)\|_{L^2_1} \quad \mbox{and}\quad |\Psi_1(f)|\lesssim \| \pa_{t}f(0)\|_{\dot H^1_1}.
\end{equation}

\noindent \underline{Claim:} We claim that $\Psi \circ \Phi$ is a linear invertible map on $\mathbb R^{m_0+2}$, and that for any $f$ with $\vec f(0)\in \mathcal H_R$, $\pa_t f(0)\in \dot H^1_R$ and $\pa_{tt} f(0)\in L^2_R$, $(\cbf,\tilde c)=(\Psi \circ \Phi)^{-1}(\Psi (f))$ satisfies
\begin{align}
\label{evenN=4:bd:linearctildec}& |\tilde c|\lesssim R \| \pa_t f(0)\|_{\dot H^1_R}, \\
\label{evenN=4:bd:linearctildec2}& |\cbf|_R\lesssim R^{N/2-1}|f(0,R)|+R \|\pa_t f(0)\|_{\dot H^1_R}+R\| \pa_{tt}f(0)\|_{L^2_R}.
\end{align}
We now prove this Claim. By scaling, it suffices to consider the case $R=1$. We decompose $\Phi(\cbf)=\Phi_0(\cbf_0)+\Phi_1(\cbf_1,\tilde c)$ where $\Phi_0(\cbf_0)= \sum_{0\leq l\leq l_0} c_{0,l}\phi_{2l}$ and $\Phi_1(\cbf_1,\tilde c)= \sum_{0\leq l\leq l_1} c_{1,l}\phi_{2l+1}+\tilde c \phi_{m_0+1}$. We have $\Psi\circ \Phi=(\Psi_0\circ \Phi_0,\Psi_1\circ \Phi_1)$. Using \eqref{id:aF} and \eqref{id:defphim0+1} we obtain the identities:
\begin{align*}
& (\Phi(\cbf,\tilde c))(0,1)=\sum_{0\leq l\leq l_0}c_{0,l},\\
& \left(\pa_{t}\Phi(\cbf,\tilde c) \right)(0,r)=\sum_{0\leq l\leq l_1}c_{1,l}r^{-N+2l}+\tilde c r^{-N/2},\\
&  \left(\pa_{tt}\Phi(\cbf,\tilde c) \right)(0,r)=\sum_{1\leq l\leq l_0}2l(-N+2l+2)c_{0,l}r^{-N+2l},
\end{align*}
and deduce, since $l_0=l_1+1$ when $N\equiv 4\mod 4$, that both map $ \Psi_0\circ\Phi_0$ and $ \Psi_1\circ\Phi_1$ are invertible maps. Hence $\Psi\circ \Phi$ is invertible. Using \eqref{evenN=4:bd:linearctildectech}, the estimates \eqref{evenN=4:bd:linearctildec} and \eqref{evenN=4:bd:linearctildec2} follow. This proves the claim.\\

The rest of the proof, showing Lemma \ref{lem:N=4defctildec} using the above Claim, is very similar to Step 2 of the proof of Lemma \ref{lem:defctildec} using Step 1. We omit the details.

 \end{proof}

We next control $u-u_{ap}$ in terms of the parameters $\cbf[R]$ and $\tilde c,R]$.

\begin{lemma} \label{evenN=4:lem:approximation}

For all $R>R_2$, let $\cbf=\cbf[ R]$ and $\tilde c=\tilde c[R]$ be given by Lemma \ref{lem:N=4defctildec}. Then:
\begin{align}
\label{evenN=4:bd:u-uap1} & R\| \pa_t u(0)-\pa_t u_{ap}[\cbf,\tilde c,R](0)\|_{\dot H^1_{R}} +R\| \pa_{tt} u(0)-\pa_{tt} u_{ap}[\cbf,\tilde c,R](0)\|_{L^2_{\tilde R}}\\
\nonumber &\qquad \qquad \qquad \qquad+R^{1/2}\sup_{r>R} r^{\frac{N-3}{2}}|u_0(r)-u_{ap}[\cbf,\tilde c,R](0,r)| \qquad  \lesssim |\tilde c|( |\cbf|_{R}+\tilde c^2).
\end{align}

\end{lemma}

\begin{proof}

We let $v=u-u_{ap}$, which solves \eqref{even:id:equationu-uap}. We aim at estimating $\pa_t v$ which, by \eqref{even:id:equationu-uap} and Lemma \ref{evenN=4:lem:gainreg}, is a non-radiative solution of:
\begin{equation}  \label{evenN4:defF} 
\left\{ \begin{array}{l l} \Box \pa_t v  =\pa_t E +\varphi'(u)\pa_t u-\varphi'(u_{ap})\pa_t u_{ap} =: F \\
(\pa_t v(0),\pa_{tt}v(0))=(\pa_t u(0)-\pa_t u_{ap}(0),\pa_{tt} u(0)-\pa_{tt} u_{ap}(0)).
\end{array}
\right.
\end{equation}
More precisely, we will bound the quantity
\begin{equation}
 \label{evenN4:controle:id:q} q = R^{1/2}\| (\pa_t v,\pa_{tt}v)\|_{W_R^{1/2}}.
\end{equation}
Because of Lemmas \ref{evenN=4:lem:gainreg2} and \ref{evenN=4:lem:gainreg4}, and of \eqref{evenN=4:bd:patuapmathcalH}, $q$ is finite with
\begin{equation}\label{evenN=4:controle:bd:aprioriq1q2}
q\lesssim \epsilon'.
\end{equation}

We will use several times the following inequality for $x,y\in \mathbb R$ with $|x|,|y|\ll r^{1-N/2}$ (we recall $\varphi $ is either given by \eqref{eq:nonlinearityanalytic} or by \eqref{eq:nonlinearitypower} and $N=4$):
\begin{align}
\label{even:bd:varphi6} & |\varphi''(x+y)-\varphi''(x)|\lesssim r^{N-4}|y| .
\end{align}

\noindent \textbf{Step 1}. \emph{Preliminary estimates}. We first control $v$, $\pa_tv$ and $\pa_{tt}v$ by $q$ and claim that for all $r>R+|t|$:
\begin{align} 
& \label{N=4even:bdv} |v(t,r)|\lesssim \frac{q+|\tilde c|(|\cbf|+\tilde c^2+\frac{1}{M^{1/2}})}{r^{N/2-3/2}R^{1/2}},\\
& \label{N=4even:bdpatv} |\pa_t v(t,r)|\lesssim \frac{q}{r^{N/2-1/2}R^{1/2}},
\end{align}
and that for all $t\in \mathbb R$ and $\bar R\geq R+|t|$:
\begin{equation}
\label{N=4even:bdpattv} \| \pa_{tt} v(t)\|_{L^2_{\bar R}}\lesssim \frac{q}{\bar R^{1/2}R^{1/2}}.
\end{equation}

\noindent \underline{Proof of \eqref{N=4even:bdpatv}} This is a consequence of \eqref{evenN4:controle:id:q}, of the definition of the $ W^{1/2}_{R}$ norm \eqref{id:defWkappa}, and of the radial Sobolev embedding.

\smallskip

\noindent \underline{Proof of \eqref{N=4even:bdpattv}} This is from the very definition of the $ W^{1/2}_{R}$ norm \eqref{id:defWkappa}.

\smallskip

\noindent \underline{Proof of \eqref{N=4even:bdv}:} Using \eqref{even:id:equationu-uap}, $v(0)$ satisfies:
\begin{align} \label{evenN=4:id:Deltav+}
 \Delta v(0) = &\pa_{tt}v(0)+(\varphi(u_{ap})-\varphi(u_{ap}+v))-E(0).
\end{align}
Let $r>\bar R$. Using \eqref{even:bd:varphi1} and \eqref{even:pointwiseuap}:
\begin{equation}\label{evenN=4:controle:tech2}
|\varphi(u_{ap})-\varphi(u_{ap}+v)|\lesssim r^{\frac{N-6}{2}}|v|\,|u_{ap}|\lesssim \frac{\epsilon'}{r^2}|v|.
\end{equation}
Injecting \eqref{N=4even:bdpattv}, \eqref{evenN=4:controle:tech2} and \eqref{evenN=4:L2E} in \eqref{evenN=4:id:Deltav+} we get:
\begin{equation}\label{evenN=4:controle:tech4}
\| \Delta v(0)\|_{L^2_{\bar R}}\lesssim \frac{q+|\tilde c|\left(|\cbf|_{R}+\tilde c^2+\frac{1}{M^{1/2}}\right)}{\bar R^{1/2}R^{1/2}} +\epsilon' \| \frac{v_+(0)}{r^2}\|_{L^2_{\bar R}}.
\end{equation}
Since $v(0,R)=0$ by \eqref{evenN=4:id:orthogonalityatR}, injecting \eqref{evenN=4:controle:tech4} in the weighted Hardy inequality \eqref{bd:hardyoutside} with $\kappa=1/2$ shows:
\begin{align*}
\sup_{\bar R\geq R} \bar R^{\frac 12} \| \frac{v(0)}{r^2}\|_{L^2_{\bar R}}+\sup_{r\geq {R}} r^{\frac{N-3}{2}} |v(0,r)| & \lesssim \frac{q_2+|\tilde c|\left(|\cbf|_{R}+\tilde c^2+\frac{1}{M^{1/2}}\right)}{R^{1/2}}\\
&\qquad +\epsilon' \sup_{\bar R\geq R}\bar R^{\frac 12} \| \frac{v(0)}{r^2}\|_{L^2_{\bar R}} .
\end{align*}
This implies \eqref{N=4even:bdv} for $t=0$ for $\epsilon'$ small enough. From \eqref{N=4even:bdv} for $t=0$ and \eqref{N=4even:bdpatv} we obtain \eqref{N=4even:bdv} for $t\neq 0$ using $v(t)=v(0)+\int_0^t \pa_t v(t')dt'$.\\

\noindent \textbf{Step 2}. \emph{Bound for $F$}. In this step we prove:
\begin{equation}\label{evenN=4:bd:F} 
R^{1/2}\sup_{\tilde R>R+|t|}\tilde R^{3/2}\| F(t)\|_{L^2_{\tilde R}}\lesssim \epsilon'q+|\tilde c|\left(|\cbf|_R+\tilde c^2+\frac{1}{M^{1/2}}\right).
\end{equation}
Recall $F$ is given by \eqref{evenN4:defF}. We decompose using \eqref{evenN4:defF} and $u=u_{ap}+v$:
\begin{equation}\label{evenN=4:id:Finter1} 
F=\pa_t E+\underbrace{\varphi'(u)\pa_t v}_{=I}+\underbrace{(\varphi'(u_{ap}+v)-\varphi'(u_{ap}))\pa_t u_{ap}}_{=II}.
\end{equation}
Let $r>R+|t|$. For $I$, using first \eqref{even:bd:varphi3}, then \eqref{even:bd:aprioriu} and \eqref{N=4even:bdpatv}:
\begin{equation}\label{evenN=4:bd:Finter2} 
|I|\lesssim r^{\frac{N-6}{2}}|u| |\pa_t v|\lesssim r^{\frac{N-6}{2}} \frac{\epsilon'}{r^{N/2-1}} \frac{q}{r^{N/2-1/2}R^{1/2}}\lesssim \frac{\epsilon' q}{r^{N/2+3/2}R^{1/2}}.
\end{equation}
For $II$, using first \eqref{even:bd:varphi4}, then \eqref{N=4even:bdv} and \eqref{evenN=4:bd:patuap}:
\begin{equation}\label{evenN=4:bd:Finter3} 
|II| \lesssim r^{\frac{N-6}{2}}|v||\pa_t u_{ap}| \lesssim r^{\frac{N-6}{2}} \frac{q+|\tilde c|(|\cbf|_R+\tilde c^2+\frac{1}{M^{1/2}})}{r^{N/2-3/2}R^{1/2}}\frac{\epsilon'}{r^{N/2}}\lesssim \frac{\epsilon'q+\epsilon'|\tilde c|(|\cbf|_R+\tilde c^2+\frac{1}{M^{1/2}})}{r^{N/2+3/2}R^{1/2}}.
\end{equation}
Combining \eqref{evenN=4:bd:Finter2} and \eqref{evenN=4:bd:Finter3}, then applying \eqref{even:controle:bd:universal1} we get that for $\tilde R\geq R+|t|$:
\begin{equation}\label{evenN=4:bd:Finter4} 
\| I+II\|_{L^2_{\tilde R}}\lesssim  \frac{\epsilon'q+\epsilon'|\tilde c|(|\cbf|_R+\tilde c^2+\frac{1}{M^{1/2}})}{\tilde R^{3/2}R^{1/2}}.
\end{equation}
Injecting \eqref{evenN=4:bd:patE} and \eqref{evenN=4:bd:Finter4} in \eqref{evenN=4:id:Finter1} shows the desired inequality \eqref{evenN=4:bd:F}.\\

\noindent \textbf{Step 3}. \emph{Bound for $\pa_{t}F$}. In this step we show $\pa_t F\in W^{'3/2}_R$ with:
\begin{equation}\label{evenN=4:bd:patF} 
R^{1/2}\| \pa_t F\|_{W_R^{'3/2}}\lesssim \epsilon'q+|\tilde c|\left(|\cbf|_R+|\tilde c|^2+\frac{1}{M^{1/2}}\right).
\end{equation}
In order to prove \eqref{evenN=4:bd:patF}, we first compute using \eqref{evenN4:defF}:
\begin{align} \nonumber
\pa_{t}F & =\pa_{tt} E +\varphi''(u)(\pa_t u)^2+\varphi'(u)\pa_{tt}u-\varphi''(u_{ap})(\pa_t u_{ap})^2-\varphi'(u_{ap})\pa_{tt}u_{ap}\\
 \label{evenN=4:id:patF}  &=\pa_{tt} E+\underbrace{(\varphi''(u)-\varphi''(u_{ap}))(\pa_t u)^2}_{=I}+\underbrace{\varphi''(u_{ap})\pa_t v (\pa_t u+\pa_t u_{ap})}_{=II}\\
\nonumber&\qquad +\underbrace{(\varphi'(u)-\varphi'(u_{ap}))\pa_{tt}u}_{=III}+\underbrace{\varphi'(u_{ap})\pa_{tt}v}_{=IV} .
\end{align}
Let $r>R+|t|$. For $I$, using \eqref{even:bd:varphi6}, then \eqref{N=4even:bdv} and \eqref{evenN=4:bd:patu}:
\begin{equation}\label{evenN=4:bd:patFinter1} 
|I| \lesssim r^{N-4} |v||\pa_t u|^2\lesssim r^{N-4} \frac{q+|\tilde c|(|\cbf|+\tilde c^2+\frac{1}{M^{1/2}})}{r^{N/2-3/2}}\frac{\epsilon^{'2}}{r^N} \lesssim \frac{\epsilon^{'2}q+\epsilon^{'2}|\tilde c|(|\cbf|+\tilde c^2+\frac{1}{M^{1/2}})}{r^{N/2+5/2}}.
\end{equation}
For $II$, using \eqref{even:bd:varphi5}, then \eqref{N=4even:bdpatv}, \eqref{evenN=4:bd:patu} and \eqref{evenN=4:bd:patuap}:
\begin{equation}\label{evenN=4:bd:patFinter2} 
|II| \lesssim r^{\frac{N-6}{2}} |\pa_t v|(|\pa_t u|+|\pa_t u_{ap}|) \lesssim r^{\frac{N-6}{2}} \frac{q}{r^{N/2-1/2}}\frac{\epsilon'}{r^{N/2}}=\frac{\epsilon' q}{r^{N/2+5/2}}.
\end{equation}
For $III$, we have for $r>R+|t|$ using \eqref{even:bd:varphi4}, \eqref{N=4even:bdv} and \eqref{evenN=4:bd:pattu} that:
\begin{equation}\label{evenN=4:bd:patFinter2bis} 
|III|  \lesssim r^{\frac{N-6}{2}} |v| |\pa_{tt} u|\lesssim r^{\frac{N-6}{2}}\frac{q}{r^{N/2-3/2}}\frac{\epsilon'}{r^{N/2+1}}=\frac{\epsilon'q}{r^{N/2+5/2}}.
\end{equation}
Combining \eqref{evenN=4:bd:patFinter1}, \eqref{evenN=4:bd:patFinter2} and \eqref{evenN=4:bd:patFinter2bis}, and applying \eqref{even:controle:bd:universal2} we get for $\bar R\geq R+|\bar t|$:
\begin{equation}\label{evenN=4:bd:patFinter3} 
\| I+II+III\|_{L^1L^2_{\bar R+|t-\bar t|}}\lesssim \frac{\epsilon^{'}q+\epsilon^{'2}|\tilde c|(|\cbf|+\tilde c^2+\frac{1}{M^{1/2}})}{\bar R^{3/2}}.
\end{equation}
For $IV$, we have for $r>R+|t|$ using \eqref{even:bd:varphi3} and \eqref{evenN=4:bd:uap} that
$$
|IV|\lesssim r^{\frac{N-6}{2}} |u_{ap}| |\pa_{tt} v|\lesssim  r^{\frac{N-6}{2}}\frac{\epsilon'}{r^{N/2-1}} |\pa_{tt} v|=\frac{\epsilon' |\pa_{tt}v|}{r^{2}}.
$$
Hence by \eqref{even:controle:bd:universal5} and \eqref{N=4even:bdpattv}:
\begin{equation}\label{evenN=4:bd:patFinter5} 
\| IV\|_{L^1L^2_{\bar R+|t-\bar t|}}\lesssim \epsilon' \| \frac{1}{(\bar R+|t-\bar t|)^{1/2}r^{2}}\|_{L^1L^\infty_{\bar R+|t-\bar t|}}\| (\bar R+|t-\bar t|)^{1/2} \pa_{tt} v\|_{L^\infty L^2_{\bar R+|t-\bar t|}}\lesssim \frac{\epsilon' q}{\bar R^{3/2}}.
\end{equation}
Injecting \eqref{evenN=4:bd:pattE}, \eqref{evenN=4:bd:patFinter3} and \eqref{evenN=4:bd:patFinter5} in \eqref{evenN=4:id:patF} we obtain the desired inequality \eqref{evenN=4:bd:patF}.\\

\noindent \textbf{Step 4}. \emph{End of the proof}. We recall that $\pa_t v$ is a solution of $\Box \pa_t v=F$ that is non-radiative for $r>R+|t|$, and that satisfies $\vec v(0)=\Pi^\perp_{\mathcal H,R}\vec v(0)$ by \eqref{evenN=4:id:orthogonalitypatv}. Given that $F\in W^{'1/2}_R$ and $\pa_t F\in W^{'3/2}_R$, $v$ coincides with the solution provided by Proposition \ref{pr:nonradiativeforcingmaineven}, which shows:
$$
\frac{q}{R^{1/2}}=\| (\pa_t v,\pa_{tt}v)\|_{W_R^{1/2}}\lesssim \sup_{\tilde R>R+|t|}\tilde R^{3/2}\| F(t)\|_{L^2_{\tilde R}}+\| \pa_t F\|_{W_R^{'3/2}}\lesssim  \frac{\epsilon'q+|\tilde c|\left(|\cbf|_R+|\tilde c|^2+\frac{1}{M^{1/2}}\right)}{ R^{1/2}}
$$
Hence $ q\lesssim |\tilde c|\left(|\cbf|_R+|\tilde c|^2+\frac{1}{M^{1/2}}\right)$. Injecting \eqref{evenN=4:id:defM} shows $q\lesssim |\tilde c|\left(|\cbf|_R+\tilde c^2\right)$. Hence \eqref{evenN=4:bd:u-uap1} using \eqref{evenN4:controle:id:q}, \eqref{N=4even:bdv} and \eqref{evenN=4:id:defM}.

\end{proof}

Thanks to Lemma \ref{evenN=4:lem:approximation}, one can estimate how the parameters $\cbf$ and $\tilde c$ depend on $R$, allowing us to show $\tilde c[R_2]=0$.

 \begin{lemma} \label{evenN=4:lem:dyadicdifferences}

For all $R_2\leq R_3\leq R_4\leq 10 R_3$, there holds:
\begin{align}
\label{evenN=4:bd:cbf2-cbf1} & \cbf[R_4]=\cbf[R_3]+O_{|\cdot |_{R_3}}(|\tilde c[R_3]|(|\tilde c[R_3]|^2+|\cbf [R_3]|_{R_3})),\\
\label{evenN=4:bd:tildec2-tildec1}& \tilde c[R_4]=\tilde c[R_3]+O(|\tilde c[R_3]|(|\tilde c[R_3]|^2+|\cbf [R_3]|_{R_3})).
\end{align}

\end{lemma}

\begin{proof}

The proof of Lemma \ref{evenN=4:lem:dyadicdifferences}, relying on the bound \eqref{evenN=4:bd:u-uap1} and of the Claim of the proof of Lemma \ref{lem:N=4defctildec}, is very similar to the proof of Lemma \ref{even:lem:dyadicdifferences}. We omit it.

\end{proof}

\begin{lemma} \label{lemN=4:criticalnondispersivestabilitycontrad}

There holds $\tilde c[R_2]=0$. In particular, as a consequence of Lemma \ref{evenN=4:lem:approximation} and of finite speed of propagation we get that for all $|x|>R_1+|t|$:
$$
u(t,x)=a[\cbf(R_1),R_1](t,x).
$$

\end{lemma}

\begin{proof}

For $k\geq 0$ let $\tilde c_k=\tilde c[2^kR_2]$ and $\cbf_{k}=\cbf[2^kR_2]$. The first part of the proof is to establish that $\tilde c_k,|\cbf_k|_{R_k}\in \ell^2$ with $\sum_{k\geq 1}^\infty |\cbf_k|_{R_k}^2+|\tilde c_k|^2 \lesssim \| \vec u(0)\|_{\mathcal H_{16R_0}}^2\lesssim \epsilon^{'2}$. The proof of this bound is very similar to Step 1 of the proof of Lemma \ref{lem:criticalnondispersivestabilitycontrad}, relying on \eqref{evenN=4:bd:u-uap1}; we omit it. Then, notice that from \eqref{evenN=4:bd:cbf2-cbf1} and \eqref{evenN=4:bd:tildec2-tildec1}, $\cbf$ and $\tilde c$ satisfy \eqref{even:bd:cbf2-cbf1} and \eqref{even:bd:tildec2-tildec1} with $\beta=0$. Hence Step 2, 3 and 4 of the proof of Lemma \ref{lem:criticalnondispersivestabilitycontrad} can be applied without modifications, showing $\tilde c[R_2]=0$ as desired.

\end{proof}

\section{Maximal solution}
\label{sec:maximal}

In this section we prove Theorem \ref{th:maximal}. 
\subsection{Extension of the solution}
Theorem \ref{th:maximal} will follow from the following extension result:
\begin{proposition}[Extension of non-radiative solutions]
\label{pr:extension}
Assume that the nonlinearity $\varphi$ is of the form \eqref{eq:nonlinearityanalytic} and that $N\geq 4$, or that 
$\varphi$ is of the form \eqref{eq:nonlinearitypower} and $N\in\{4,6\}$ or $N$ is odd. 
 Let $R>0$ and $a$ be a radial nonradiative solution of \eqref{eq:nonlinearwaveintro} on $\{|x|>R+|t|\}$, with $
\vec{a}\in C^0(\Rb,\Hc_{R+|t|})$.
 Then there exists $\widetilde{R}<R$ and a radial nonradiative solution $b$ of \eqref{eq:nonlinearwaveintro} on $\{|x|>\widetilde{R}+|t|\}$ such that $a(t,x)=b(t,x)$ if $|x|>R+|t|$. 
\end{proposition}
\begin{proof}
We prove the result when $N\geq 3$ is odd and $\varphi(u)=|u|^{\frac{4}{N-2}}u$. We will sketch the proof of the other cases below.

\noindent\textbf{Step 1.}  \emph{Uniform pointwise bound for $a$.}

We consider 
$$ V(\widetilde{R})=\left\{u\in C^0(\{(t,r),\; t\in \Rb,\; r\ge \widetilde{R}+|t|\}),\; \|u\|_{V_{\widetilde{R}}}<\infty\right\}.$$
where 
$$\|u\|_{V_{\widetilde{R}}}=
\sup_{\substack{t\in \Rb\\ r>\widetilde{R}+|t|}}|u(t,r)|r^{\frac{N-1}{2}}.$$
We first prove that $a\in V(R)$. Indeed, by  Proposition \ref{pr:aprioribound}, $a\in V(R')$ for some large $R'>R$.

%Indeed, by the uniqueness in Theorem \ref{th:main}, there exists $R'\geq R$ and $\cbf$ such that $a(t,r)=a_{[\cbf]}(t,r)$, $r>R'+|t|$. This implies $a\in V(R')$. 

Next, we consider $r$ such that $R+|t|<r<R'+|t|$. Then
\begin{multline}
\label{argumentVR'}
 \left|a(t,r)-a(t,R'+|t|)\right|=\left|\int_{r}^{R'+|t|}\partial_{\rho}a(t,\rho)d\rho\right|\\
 \leq \left( \int_r^{R'+|t|}\left(\partial_{\rho}a(t,\rho)\right)^2\rho^{N-1}d\rho \right)^{\frac{1}{2}}\left(\int_{r}^{R'+|t|}\frac{d\rho}{\rho^{N-1}}\right)^{\frac 12} \lesssim \frac{(R'-R)^{1/2}}{r^{\frac{N-1}{2}}} \|\partial_ra(t)\|_{L^2_{R+|t|}},
\end{multline}
and it follows from the bound $|a(t,R'+|t|)|\lesssim (R'+|t|)^{1-\frac{N}{2}}$, and the fact that 
$$\sup_{t\in \Rb}\|\vec{a}\|_{\Hc_{R}}<\infty,$$
that follows from the assumptions on $a$, that there exists a constant $C=C(a)$ such that
$$ R+|t|<r<R'+|t|\Longrightarrow |a(t,r)|\leq C r^{-\frac{N-1}{2}}.$$
Since $a\in V(R')$ we deduce $a\in V(R)$.

Note that the fact that $a\in V(R)$ implies $\varphi(r,a)\in L^1L^2(\{r>R+|t|\})$. Indeed, it implies
% $$|\varphi(r,a)|\lesssim
% \begin{cases}
%    r^{-\frac{N}{2}-2}&\text{ if }\varphi\text{ is of the form }\eqref{eq:nonlinearityanalytic}\\
%     r^{-\frac{(N-1)(N+2)}{2(N-2)}}&\text{ if }\varphi\text{ is of the form }\eqref{eq:nonlinearitypower}.
% \end{cases}$$
$$|\varphi(a)|\lesssim r^{-\frac{(N-1)(N+2)}{2(N-2)}}.$$
Noting that $|x|^{-\alpha}\in L^1L^2(\{r>R+|t|\})$ if and only if $\alpha>\frac{N}{2}+1$, we obtain that $\varphi(r,a)\in L^1L^2(\{r>R+|t|\})$.

\medskip

\noindent\textbf{Step 2.} \emph{Extension of the solution and reduction to a fixed point problem.}

 We let $\as$ be the solution of 
\begin{equation*} 
\left\{ 
\begin{array}{l l} \pa_t^2 \as-\Delta \as=\varphi(a) \indic(|x|>R+|t|),\\
\vec \as(0)=(\as_0,\as_1),
\end{array}
\right.
\end{equation*}
 where 
 $$
 \begin{cases}
  \as_0(r)=a_0(r),\; \as_1(r)=a_1(r) &\text{ if }r>R\\
 \as_0(r)=a_0(R),\; \as_1(r)=0 &\text{ if }r<R.
 \end{cases}
$$
 By Strichartz estimates, $\vec{\as}\in C^0(\Rb,\Hc)$, $\as\in S(\Rb\times \Rb^N)$. 
Since by finite speed of propagation $\as(t,x)=a(t,x)$ for $|x|>R+|t|$ and $a$ is nonradiative in this region, we have 
\begin{equation}
 \label{T70}`
 \lim_{t\to\pm\infty} \int_{|x|>R+|t|} \left|\nabla_{t,x}\as(t,x)\right|^2dx=0.
\end{equation} 
We fix $\widetilde{R}>0$, with $\widetilde{R}<R$, and let $h$ be the solution of the free wave equation \eqref{eq:freewave} on $|x|>\widetilde{R}+|t|$ given by Proposition \ref{pr:T10}, such that $\vec{h}(0,r)=0$ for $r>R$ and 
$$\lim_{t\to\infty} \int_{|x|>\widetilde{R}+|t|}|\nabla_{t,x}(h-\as)(t,x)|^2dx=0.$$
Note that this is possible because of Corollary \ref{cor:inhomogeneousradiation} and \eqref{T70}. 

Letting $\tilde{a}=\as-h$, we see that 
$$ \tilde{a}(t,x)=a(t,x),\; |x|>R+|t|,\quad \vec{\tilde{a}}(0,x)=\vec{a}(0,x),\; |x|>R.$$
and 
$$ \partial_t^2\tilde{a}-\Delta \tilde{a}=\varphi(a)\indic(|x|>R+|t|)=\varphi(\tilde{a})\indic(|x|>R+|t|),\quad |x|>\widetilde{R}+|t|.$$
Furthermore, $(\tilde{a},\partial_t\tilde{a})\in C^0(\Rb,\Hc_{R+|t|})$, $\tilde{a}\in S(\{r>R+|t|\})$
Thus $b$ satisfies the conclusion of Proposition \ref{pr:extension} if and only if $u=b-\tilde{a}$ satisfies
\begin{gather}
\label{T80}
 \partial_t^2u-\Delta u=\varphi(\tilde{a}+u)-\varphi(\tilde{a})\indic(|x|>R+|t|),\quad |x|>\widetilde{R}+|t|\\
 \label{T80'}
 \sum_{\pm} \lim_{t\to\pm\infty} \int_{|x|>\widetilde{R}+|t|} |\nabla_{t,x}u(t,x)|^2dx=0\\
 \label{T81}
 (u_0,u_1):=\vec{u}(0)\in \Hc_{\widetilde{R}},\quad u(t,x)=0\text{ for } |x|>R+|t|.
\end{gather}
Due to \eqref{T81}, we can write \eqref{T80} as
\begin{equation}
\label{goal_pointfixe}
  \partial_t^2u-\Delta u=\varphi(\tilde{a}+u)\indic(\widetilde{R}+|t|<|x|<R+|t|).
\end{equation} 
Also, if $t\in \Rb$, since $h(t,r)=0$ for $r>R+|t|$, we have that $\tilde{a}=a$ in this region, and thus $\tilde{a}\in V(R)$. Using the same argument than in \eqref{argumentVR'}, we see that $\tilde{a}\in V(\widetilde{R})$.

We let $\mathcal{B}(\widetilde{R})$ be the unit ball of $V(\widetilde{R})$ centered at the origin.

We let $\Psi$ be that maps $v\in \mathcal{B}(\widetilde{R})$ to the solution $u$ of 
\begin{equation}
 \label{eq_pointfixe}
\partial_t^2u-\Delta u=\varphi(\tilde{a}+v)\indic(\widetilde{R}+|t|<|x|<R+|t|)
 \end{equation} 
given by Proposition \ref{pr:nonradiative_compactforcing}. Taking $\widetilde{R}$ close enough to $R$, we will prove that $\Psi$ is a contraction of $\mathcal{B}(\widetilde{R})$. Note that the conclusion of Proposition \ref{pr:nonradiative_compactforcing} imposes that the fixed point $u$ of $\Psi$ is in $\Hc_{\widetilde{R},R}$, which implies that $u$ satisfies \eqref{T81}, and that $u$ is nonradiative, which yields \eqref{T80'}. 

\medskip

\noindent\textbf{Step 3.} \emph{Contraction property.}
Recall that we have assumed that $\varphi$ is of the form \eqref{eq:nonlinearitypower}. Thus
\begin{multline}
 \label{calcul_L1L2} 
 \Big\|\varphi(\tilde{a}+v)\indic(\widetilde{R}+|t|<|x|<R+|t|)\Big\|_{L^1L^2}\\
 \lesssim \|\tilde{a}+v\|^{\frac{N+2}{N-2}}_{V(\widetilde{R})}\int_{\Rb} \left(\int_{\widetilde{R}+|t|}^{R+|t|} \frac{1}{r^{\frac{(N-1)(N+2)}{N-2}}}r^{N-1}dr\right)^{1/2}dt
 \\ 
 \lesssim \frac{(R-\widetilde{R})^{1/2}}{R^{\frac{N}{N-2}}}\|\tilde{a}+v\|^{\frac{N+2}{N-2}}_{V(\widetilde{R})}.
\end{multline}
By Proposition \ref{pr:nonradiative_compactforcing}, standard energy estimates and finite speed of propagation we deduce that if $v\in V(\widetilde{R})$, then the nonradiative solution $u$ of \eqref{eq_pointfixe} satisfies 
\begin{equation}
\label{bound_vuT}
\forall t\in \Rb,\quad \|\vec{u}(t)\|_{\Hc_{\widetilde{R}+|t|}}\lesssim \frac{(R-\widetilde{R})^{1/2}}{R^{\frac{N}{N-2}}}\left(\|v\|_{V(\widetilde{R})}^{\frac{N+2}{N-2}}+\|\tilde{a}\|_{V(\widetilde{R})}^{\frac{N+2}{N-2}}\right). 
\end{equation} 
Using that $u(t,r)=0$ for $r>R+|t|$, we obtain, for $\widetilde{R}+|t|<r<R+|t|$,
\begin{multline}
\label{u_in_V}
|u(t,r)|=\left|\int_R^r \partial_ru(t,r)dr\right|\lesssim \left(\int_{R+|t|}^{r}|\partial_ru(t,r)|^2r^{N-1}dr\right)^{1/2}\left(\int_{R+|t|}^r\frac{d\sigma}{\sigma^{N-1}}\right)^{1/2}\\ \lesssim (R-\widetilde{R})^{1/2} r^{-\frac{N-1}{2}}\|\vec{u}(r)\|_{\Hc_{R_+|t|}},
\end{multline}
Combining \eqref{bound_vuT} and \eqref{u_in_V}, we obtain that $u\in V(\widetilde{R})$ and  
\begin{equation}
 \label{StabilityPsi}
\|\Psi(v)\|_{V(\widetilde{R})}=\|u\|_{V(\widetilde{R})}\lesssim \frac{R-\widetilde{R}}{R^{\frac{N}{N-2}}}\left(\|v\|^{\frac{N+2}{N-2}}+\|\tilde{a}\|^{\frac{N+2}{N-2}}\right).
 \end{equation} 
Also, the pointwise inequality 
$$ \left|\varphi(\tilde{a}+v)-\varphi(\tilde{a}+w)\right| \leq |v-w|\left(|\tilde{a}|^{\frac{4}{N-2}}+|v|^{\frac{4}{N-2}}\right)$$
and the same computation as before leads to:
\begin{equation}
 \label{ContractionPsi}
\|\Psi(v)-\Psi(w)\|_{V(\widetilde{R})}\lesssim \frac{R-\widetilde{R}}{R^{\frac{N}{N-2}}}\left(\|v\|^{\frac{4}{N-2}}+\|\tilde{a}\|^{\frac{4}{N-2}}\right)(\|v-w\|_{V(\widetilde{R})}).
 \end{equation} 
Taking $\widetilde{R}<R$ close enough to $R$, we obtain that $\Psi$ is a contraction of $\mathcal{B}$, yielding a function $u$ satisfying \eqref{T80}, \eqref{T80'} and \eqref{T81}, and concluding the proof of the proposition.

% \noindent\textbf{Step 4.} \emph{Contraction property: analytic nonlinearity.}
% 
% 
% The proof is similar when $\varphi$ is of the form \eqref{eq:nonlinearityanalytic}. Indeed, if $v\in \Bc(\widetilde{R})$, we have 
% \begin{multline}
% \label{comput_analytic}
%  |\varphi(x,\tilde{a}+v)|=\left|\sum_{k\geq 2} \varphi_k |x|^{(k-1)\left( \frac N2-1 \right)}a^k\right|\\ \leq \sum_{k\geq 2} |\varphi_k| |x|^{(k-1)\left(\frac{N}{2}-1\right)-2} |x|^{\left(\frac{1-N}{2}\right)k}\|\tilde{a}+v\|^k_{V(\widetilde{R})}\leq C(\tilde{a}) |x|^{-2-\frac{N}{2}},
% \end{multline}
% and by direct computation, denoting $u=\Psi(v)$
% $$ \|\vec{u}(t)\|_{\Hc_{\widetilde{R}+|t|}}\leq C(\tilde{a}) \frac{(R-\widetilde{R})^{1/2}}{R^{3/2}},\quad \|u\|_{V(\widetilde{R})}\leq C(\tilde{a})\frac{R-\widetilde{R}}{R^{3/2}}.$$
% Similarly, using also $\left|(\tilde{a}+v)^k-(\tilde{a}+w)^k\right|\leq k |v-w|\left( |a|+|v|+|w| \right)^{k-1}$ for $k\geq 2$, we obtain 
% $$ \left\|\Psi(v)-\Psi(w)\right\|_{V(\widetilde{R})}\leq C(\tilde{a})\frac{R-\widetilde{R}}{R^{3/2}}\|v-w\|_{V(\widetilde{R})},$$
% and the contraction property follows when $\widetilde{R}$ is close enough to $R$.

\medskip

\noindent\textbf{Step 4.} \emph{Other cases.}
Assume that the nonlinearity is of the form \eqref{eq:nonlinearityanalytic}. Then the proof above can be adapted as follows:
\begin{itemize}
 \item The a priori estimate $|a|\lesssim r^{\frac{1-N}{2}}$ obtained in Step 1 is not necessary. Noting that the $a\in S(\{|x|>R+|t|\})$, and that $\vec{a}$ is uniformly bounded in $\Hc_{R+|t|}$ by the well-posedness theory outside wave cones for \eqref{eq:nonlinearwaveintro}, we obtain immediately from Lemma \ref{lem:lipschitz} (with $v=0$) that $\varphi(a)\in L^1L^2(\{|x|>R+|t|\})$.
 \item Using the radial Sobolev inequality and Lemma \ref{lem:lipschitz}, one can show that the map $\Psi$ constructed in Step 2 is a contraction of the unit ball of the space of functions $u$ such that 
 $$ u\in S(\{|x|>\widetilde{R}+|t|\}),\quad \sup_{\substack{t\in \Rb\\ r>t+\widetilde{R}}} r^{\frac{N}{2}-1}|u(t,r)|<\infty,$$
\end{itemize}
The same strategy works for the power nonlinearity $\varphi(u)=|u|^{\frac{4}{N-2}}u$, when $3\leq N\leq 6$, replacing the space $S$ by $L^{\frac{N+2}{N-2}}L^{\frac{2(N+2)}{N-2}}$, and proving the contraction property in the unit ball of the space $L^{\frac{N+2}{N-2}}L^{\frac{2(N+2)}{N-2}}(\{|x|>\widetilde{R}+|t|\})$.

\end{proof}
\subsection{Construction of the maximal solution}
We next complete the proof of Theorem \ref{th:maximal}. 

We consider a non-radiative solution $u$ of \eqref{eq:nonlinearwaveintro} defined for $|x|>R_0+|t|$. We define $R_*$ as the infimum of $R\in [0,R_0]$ such that there exists a non-radiative solution $\tilde{u}$ of  \eqref{eq:nonlinearwaveintro} defined for $|x|>R+|t|$, and such that $u=\tilde{u}$ for $|x|>R_0+|t|$. We note that if it exists, such a non-radiative solution $\tilde{u}$ is unique (by \cite[Proposition 3.8]{DuKeMe21a}). As a consequence, there exists a function $u_*$, defined for $|x|>R_*+|t|$, and such that for any $R>R_*$, the restriction of $u_*$ to $\{|x|>R+|t|\}$ satisfies 
$$\vec{u}_*\in C^0(\Rb,\Hc_{R+|t|}),\quad u_*\in S(\{|x|>R+|t|\})$$
and is a non-radiative solution of \eqref{eq:nonlinearwaveintro}. 

We will conclude the proof of Theorem \ref{th:maximal} by contradiction, assuming that $R_*>0$, $u_*\in S(\{|x|>R_*+|t|\})$ and $\vec{u}_*(0)\in \Hc_{R_*}$. We will prove that $u_*$ is a nonradiative solution of \eqref{eq:nonlinearwaveintro} for $|x|>R_*+|t|$, yielding a contradiction with Proposition \ref{pr:extension} and the definition of $R_*$. 

We let $\tilde{u}$ be the solution of \eqref{eq:nonlinearwaveintro} for $\{|x|>R_*+|t|\}$ with initial data $\vec{u}_*(0)$ (given by Proposition \ref{pr:well-posedness}), and $I$ its maximal interval of existence. By uniqueness in the Cauchy problem outside wave cones, we see that $\tilde{u}$ coincide with $u_*$ in the set $\{|x|>R+|t|,\; t\in I\}$ for any $R>R_*$, and thus in $\{|x|>R_*+|t|,\; t\in I\}$. Since $u_*\in S(\{|x|>R_*+|t|\})$, this implies by the blow-up criterion that $I=\Rb$. As a consequence $u_*=\tilde{u}$ is a solution of \eqref{eq:nonlinearwaveintro} for $\{|x|>R_*+|t|\}$, which scatters to a linear solution in this set. Arguing as in Proposition \ref{pr:extension}, we obtain that $\varphi(u_*)\in L^1L^2(\{|x|>R_*+|t|\})$, and thus (by Corollary \ref{cor:inhomogeneousradiation} and finite speed of propagation) that there exists $G_{\pm}\in L^2([R_*,\infty))$ such that 
\begin{equation*}
\sum_{\pm} \lim_{t\rightarrow \pm \infty} \int_{R_*+|t|}^\infty \left|r^{\frac{N-1}{2}} (\pa_t u_*,\pa_r u_*)(t,r)-(G_\pm,\mp G_\pm)(r-|t|)\right|^2dr =0.
\end{equation*}
Let $R>R_*$. Since the restriction of $u_*$ to $\{|x|>R+|t|\}$ is nonradiative, we obtain that $G_{\pm}=0$ a.e. for $|x|>R+|t|$. As a consequence $G_{\pm}=0$ in $L^2([R_*,\infty)$, which implies that $u_*$ is non-radiative for $|x|>R_*+|t|$, concluding the proof of Theorem \ref{th:maximal}.

\begin{appendix}

\section{Hardy and Sobolev-type inequalities}
\label{A:Hardy}

\begin{lemma} \label{lem:Hardy}

Assume $N\geq 5$. For any $0\leq \kappa<3$, there exists $C>0$ such that for any $R>0$ and $u$ such that $r^{-2}u\in L^2_R$, $r^{-1}\partial_ru\in L^2_R$, $\Delta u\in L^2_R$:
\begin{multline}\label{bd:hardyoutside}
\sup_{\tilde R\geq R} \tilde R^\kappa \left(\| \pa_{rr}u\|_{L^2_{\tilde R}}+\| \frac{ \pa_{r}u}{r}\|_{L^2_{\tilde R}}+\| \frac{u}{r^2}\|_{L^2_{\tilde R}}\right)+\sup_{r\geq R} r^{\frac{N-4}{2}+\kappa} |u(r)|\\
\leq C\left( R^{\frac{N-4}{2}+\kappa} |u(R)| +\sup_{\tilde R\geq R}\tilde R^\kappa \| \Delta u\|_{L^2_{\tilde R}}  \right).
\end{multline}

\end{lemma}

\begin{proof}

By changing variables $ r=R\tilde r$, one notices that it suffices to prove the result for $R=1$ which we henceforth assume.\\

\noindent \textbf{Step 1}. \emph{Proof in the case $\kappa=0$.} We claim that:
\begin{equation}\label{bd:hardyoutsidetech1}
\| \pa_{rr}u\|_{L^2_1}+\| \frac{ \pa_{r}u}{r}\|_{L^2_1}+\| \frac{u}{r^2}\|_{L^2_1}+\sup_{r\geq 1}r^{\frac{N-4}{2}}|u(r)| \leq C\left( |u(1)| +\| \Delta u\|_{L^2_1}  \right).
\end{equation}

Writing $\Delta u=\pa_{rr}u+(N-1)r^{-1}\pa_r u$, we see that $\partial_r^2u\in L^2_R$ and thus that $\partial_r u$ is a continuous function of $r\geq 1$. Integrating by parts shows that
$$
N \int_{1}^\infty |\pa_r u|^2 r^{N-3} dr= |\pa_r u (1)|^2+2\int_1^\infty \Delta u \pa_r u r^{N-2}dr.
$$
Bounding $2|\Delta u \pa_r u|\leq r|\Delta u|^2+r^{-1}|\pa_r u|^2$ above, we find that:
\begin{equation}\label{bd:hardyoutsidetech2}
(N-1) \int_{1}^\infty |\pa_r u|^2 r^3 dr\leq  |\pa_r u (1)|^2+\int_1^\infty |\Delta u|^2 r^4dr.
\end{equation}
Also, since $\pa_{rr}u =\Delta u-\frac{N-1}{r}\pa_r u$ we find that
\begin{equation}\label{bd:hardyoutsidetech3}
\int_1^\infty |\pa_{rr} u|^2r^5 dr \lesssim \int_1^\infty |\Delta u|^2r^5 dr + \int_1^\infty |\pa_r u|^2r^{N-3} dr .
\end{equation}
Combining \eqref{bd:hardyoutsidetech2} and \eqref{bd:hardyoutsidetech3} with Hardy's inequality $(N-4)^2\int_1^\infty |u|^2r^{N-5}dr \leq 4 \int_1^\infty |\pa_r u|^2r^{N-3}dr$ and the radial Sobolev inequality $(N-4)u^2(r)\leq r^{4-N}\int_r^\infty |\pa_r u(\tilde r)|\tilde r^{N-3}d\tilde r$ one finds that:
\begin{equation}\label{bd:subcoercivityhardy}
\int_1^\infty \left(|\pa_{rr}u|^2+\frac{|\pa_{r}u|^2}{r^2}+\frac{|u|^2}{r^4}\right)r^{N-1} dr+\sup_{r\geq 1}r^{N-4}u^2(r) \lesssim  |\pa_r u|^2(1) +\int_1^\infty |\Delta u|^2r^{N-1}dr .
\end{equation}
Solving Laplace's equation we get:
\begin{equation}\label{bd:hardyoutsidetech4}
u(r)=\frac{1}{(2-N)r^{N-2}}\int_1^r \Delta u(\tilde r)\tilde r^{N-1}d\tilde r+\frac{1}{2-N} \int_r^\infty \Delta u(\tilde r)\tilde rd\tilde r+\frac{u(1)+\frac{1}{N-2} \int_1^\infty \Delta u(\tilde r)\tilde r d\tilde r}{r^{N-2}},
\end{equation}
and hence $\pa_r u(1)=(2-N)u(1)-\int_1^\infty \Delta u(\tilde r)\tilde r d\tilde r$, so that by Cauchy-Schwarz:%Thomas: in the label to the next equation, and in the reference right after I changed 3 to 7 to avoid multiple labels
\begin{equation}\label{bd:hardyoutsidetech7}
|\pa_r u(1)|^2\lesssim u^2(1)+ \int_1^\infty \Delta u^2 r^{N-1} dr \int_1^\infty r^{3-N}dr\lesssim u^2(1)+ \int_1^\infty \Delta u^2 r^{N-1} dr .
\end{equation} 
Combining \eqref{bd:subcoercivityhardy} and \eqref{bd:hardyoutsidetech7} shows \eqref{bd:hardyoutsidetech1} and ends Step 1.\\

\noindent \textbf{Step 2}. \emph{Proof in the case $0<\kappa<N/2$}. Let $q=|u(1)| +\sup_{\tilde R\geq 1}\tilde R^\kappa \| \Delta u\|_{L^2_{\tilde R}} $. We claim that
\begin{equation}\label{bd:hardyoutsidetech5}
\sup_{r\geq 1} r^{\frac{N-4}{2}+\kappa} |u(r)|\lesssim q.
\end{equation}
To prove it, fix $r>1$ and let $k\in \mathbb N$ be such that $2^k\leq r \leq 2^{k+1}$. Then for the first term in the right-hand side of \eqref{bd:hardyoutsidetech4} one has by Cauchy-Schwarz:
\begin{align*}
|\int_1^r \Delta u(\tilde r)\tilde r^{N-1}d\tilde r| & \lesssim \sum_{l=0}^k \left(\int_{2^l}^{2^{l+1}} |\Delta u(\tilde r)|^2\tilde r^{N-1}d\tilde r\right)^{\frac 12}\left(\int_{2^l}^{2^{l+1}} \tilde r^{N-1}d\tilde r\right)^{\frac 12}\\
&\lesssim  \sum_{l=0}^k q 2^{-\kappa l} 2^{l\frac N2} \ \lesssim q 2^{k (\frac N2-\kappa)}\approx q r^{\frac N2-\kappa}.
\end{align*}
The second term in \eqref{bd:hardyoutsidetech4} can be treated similarly and the third can be estimated directly by Cauchy-Schwarz, yielding \eqref{bd:hardyoutsidetech5}.

Let now $R>1$. Applying \eqref{bd:hardyoutsidetech1} with $\kappa=0$ as proved in Step 1, then using \eqref{bd:hardyoutsidetech5}, we obtain:
\begin{equation}\label{bd:hardyoutsidetech6}
\| \pa_{rr}u\|_{L^2_R}+\| \frac{ \pa_{r}u}{r}\|_{L^2_R}+\| \frac{u}{r^2}\|_{L^2_R}\lesssim R^{\frac{N-4}{2}} |u(R)| + \| \Delta u\|_{L^2_{ R}}\lesssim R^{-\kappa} q .
\end{equation}
Combining \eqref{bd:hardyoutsidetech5} and \eqref{bd:hardyoutsidetech6} shows \eqref{bd:hardyoutside} and ends the proof of the Lemma.

\end{proof}

\section{Gain of regularity for radiation profiles}

\label{sec:reggrain}

We first prove Proposition \ref{construction:pr:radiation}.

\begin{proof}[Proof of Proposition \ref{construction:pr:radiation}]
We only prove the last gain of regularity property, that is not proved in \cite{Friedlander62,Friedlander80,DuKeMe19}. Assume $G_\pm \in H^1$. 
Let $v$ be the solution to \eqref{eq:freewave} that has radiation profile $-\pa_\rho G_{+}$ as $t \to \infty$. Then $\| \vec v(0)\|_{\mathcal H}= \sqrt{2|\mathbb S^{N-1}|} \| \pa_\rho G_+\|_{L^2(\mathbb R)}$. For any $h>0$, consider $u_h(t,x)=h^{-1}(u(t+h,x)-u(t,x))$, and $G_{+,-h}(\rho)=h^{-1}(G_+(\rho-h)-G_+(\rho))$. Then, $u_h$ solves \eqref{eq:freewave} and has radiation profile $G_{+,-h}$ as $t \to \infty$. As $G_{+}\in H^1$ we have $G_{+,-h}\rightarrow -\pa_\rho G_+$ in $L^2(\mathbb R)$ as $h\to 0$. Therefore, by the isometry property of Proposition \ref{construction:pr:radiation}, $(u_h(0),\pa_t u_h(0))\rightarrow (v,\pa_t v(0))$ in $\dot H^1 \times L^2(\mathbb R^N)$. Since on the other hand $u_h(0)$ and $\pa_t u_h(0)$ converge to $\pa_t u(0)$ and $\pa_{tt}u(0)$ respectively in the distributional sense, we conclude that $(\pa_t u(0),\Delta u(0))=(v,\pa_t v(0))\in \mathcal H$. The other implication comes from standard propagation of regularity for the free wave equation.

\end{proof}

Next, we prove Lemma \ref{lem:gainregradiation2}, as a consequence of the following two lemmas.

\begin{lemma} \label{gainofreg:lem:hilbert}

Let $f\in L^2(\mathbb R)$ be such that $ f'\in L^2(\rho >1)$ and $ (\mathfrak H f)'\in L^2(\rho <-1)$. Then $ f'\in L^2(|\rho| >2)$ with
\begin{equation}\label{gainofreg:bd:hilbert}
\| f'\|_{L^2(|\rho| >2)}\lesssim \|  f'\|_{L^2(\rho >1)}+\| (\mathfrak H f)'\|_{L^2(\rho<-1)} +\|f\|_{L^2(\mathbb R)}.
\end{equation}

\end{lemma}

\begin{corollary} \label{gainofreg:cor:gainouter}

Assume $N$ is even, $R>0$ and $u$ is a radial solution of \eqref{eq:freewave} that has radiation profiles $G_\pm$ as $t\to \pm \infty$. If $G_+',G_-'\in L^2(\rho>R)$, then $G_+'\in L^2(|\rho|>2R)$ with 
$$
\|G_+' \|_{L^2(|\rho|>2R)}\lesssim \|G_+' \|_{L^2(\rho>R)}+\|G_-' \|_{L^2(\rho>R)}+\frac 1R \| G_+\|_{L^2(\mathbb R)}.
$$

\end{corollary}

\begin{remark} \label{gainofreg:reg:gainouter}

A similar but stronger result holds for $N$ odd for $u$ as in Corollary \ref{gainofreg:cor:gainouter}. Indeed, by \eqref{construction:id:radiationrelation} we have the identity $\|G_+' \|_{L^2(|\rho|>R)}^2= \|G_+' \|_{L^2(|\rho|>R)}^2+\|G_-' \|_{L^2(|\rho|>R)}^2$.

\end{remark}

\begin{proof}[Proof of Corollary \ref{gainofreg:cor:gainouter}] It follows, by rescaling, as a direct consequence of Lemma \ref{gainofreg:lem:hilbert} and of the identity \eqref{construction:id:radiationrelation} for $N$ even.

\end{proof}

\begin{proof}[Proof of Lemma \ref{gainofreg:lem:hilbert}]

Let $h=\mathfrak H f$. Then $ h' \in L^2(\rho<-1)$ and $h\in L^2$ by isometry of $\mathfrak H$. Let $\chi\in \mathcal C^\infty(\mathbb R)$ with $\chi(\rho)=1$ for $\rho\leq -3/2$ and $\chi(\rho)=0$ for $\rho\geq -1$. Then 
\begin{equation}\label{gainofreg:lem:hilberttech1}
f=-\mathfrak H h= -\mathfrak H(\chi h)-\mathfrak H((1-\chi) h).
 \end{equation}
First, as $\| (\chi h)'\|_{L^2(\mathbb R)}\lesssim \| h\|_{L^2(\rho\leq -1)}+\|  h'\|_{L^2(\rho\leq -1)}$ we have:
 \begin{equation}\label{gainofreg:lem:hilberttech2}
 \| (\mathfrak H(\chi h))'\|_{L^2(\mathbb R)}\lesssim \| h\|_{L^2(\rho\leq -1)}+\|  h'\|_{L^2(\rho\leq -1)}.
 \end{equation}
Second, for $\rho\leq -2$ we have
$$
\mathfrak H((1-\chi)h)(\rho)=\int_{\mathbb R} \frac{1}{\pi(\rho-\eta)}(1-\chi)(\eta)h(\eta)d\eta
$$
and $\eta\geq -3/2$ in the domain of integration. Thus, $|( \mathfrak H((1-\chi)h))'(\rho)|\lesssim \int_{\eta\geq -3/2} \frac{1}{|\rho-\eta|^2+1}|h(\eta)|d\eta$ so that by Young's inequality for convolution
\begin{equation}\label{gainofreg:lem:hilberttech3}
 \| (\mathfrak H((1-\chi)h))'(\rho)\|_{L^2(\rho \leq -2)} \lesssim \| h\|_{L^2(\rho \geq -3/2)}.
\end{equation}
Injecting \eqref{gainofreg:lem:hilberttech2}, \eqref{gainofreg:lem:hilberttech3} and $\| h\|_{L^2(\mathbb R)}=\| f\|_{L^2(\mathbb R)}$ in \eqref{gainofreg:lem:hilberttech1} shows $\|  f'\|_{L^2(\rho<-2)}\lesssim \| f\|_{L^2(\mathbb R)} +\|  h'\|_{L^2(\rho \geq -3/2)}$, and \eqref{gainofreg:bd:hilbert} follows.

\end{proof}

\begin{lemma} \label{gainofreg:lem:compactradiation}

Assume $N$ is even, $R>0$, and $u$ is a radial solution of \eqref{eq:freewave} whose radiation profile $G_+$ as $t\to \infty$ satisfies $\textup{supp}(G_+)\subset [-R,R]$. Then $( u_1,\Delta u_0)\in \mathcal H_{2R}$ with
\begin{equation}\label{gainofreg:bd:compactradiation}
\| (u_1,\Delta u_0)\|_{\mathcal H_{2R}}\lesssim \frac{1}{R} \| G_+\|_{L^2}.
\end{equation}

\end{lemma}

\begin{remark}

In the setting of Lemma \ref{gainofreg:lem:compactradiation}, one can in fact prove that $(u_0,u_1)\in \dot H^m\times \dot H^{m-1}(r>2R)$ for any $m$.

\end{remark}

\begin{remark} \label{gainofreg:rem:compactradiation}

A stronger result holds for $N$ odd for $u$ as in Lemma \ref{gainofreg:lem:compactradiation}: one has $\vec u(0,r)=\vec a_F[\cbf](0,r)$ for all $r>R$ for some $\cbf \in \mathbb R^{m_0+1}$ with $|\cbf|_R\lesssim \| G\|_{L^2(\mathbb R)}$.

Indeed, by \eqref{construction:id:radiationrelation}, $\textup{supp}(G_-)\subset [-R,R]$, so that $u$ is non-radiative for $r>R+|t|$, hence $\vec u(0,r)=\vec a_F[\cbf](0,r)$ for all $r>R$ by Classification \ref{pr:nonradiativefree}. Then, by \eqref{bd:energyaFiscbfR2} and Proposition \ref{construction:pr:radiation}, $|\cbf|_R\approx \| \vec u(0)\|_{\mathcal H_R}\lesssim \| G\|_{L^2(\mathbb R)}$.

\end{remark}

\begin{proof}[Proof of Lemma \ref{gainofreg:lem:compactradiation}]

We prove the result in the case $R=1$, and the general case follows by rescaling.

We recall from the proof of Lemma 4.3. in \cite{LiShenWei21P} that (with $G=G_+(u)$):
\begin{equation}\label{gainofreg:id:compactradiationtech1}
u(r,t)=\frac{C_N }{r^{N/2-1}} \int_{\rho=0}^\infty \int_{\omega=-1}^1 \frac{G(r\omega-\rho+t)}{\sqrt{\rho}} P_N(\omega)(1-\omega)^{-1/2}d\rho d\omega
\end{equation}
where $P_N$ is a polynomial of degree $N/2-1$, defined by $\pa_\omega^{N/2-1}((1-\omega^2)^{\frac{N-3}{2}})=P_N(\omega)(1-\omega^2)^{-1/2}$. We will prove \eqref{gainofreg:bd:compactradiation} with $R=1$ assuming $G\in C^\infty_0((|\rho|<1))$ and a standard limiting argument gives the general case of $G\in L^2((|\rho|<1))$. We thus consider first $\pa_r^2u(r,0)$. By \eqref{gainofreg:id:compactradiationtech1}:
\begin{equation}\label{gainofreg:id:compactradiationtech2}
 \pa_r^2 u_0= \frac{(1-\frac N2)(\frac N2-2)}{r^2}u_0-\frac{N-2}{r}\pa_r u_0+\frac{1}{r^{\frac{N-2}{2}}} \int_{\rho=0}^\infty \int_{\omega=-1}^1 \frac{G''(r\omega-\rho)}{\sqrt{\rho}} \tilde P_N(\omega)(1-\omega)^{-1/2}d\rho d\omega
\end{equation}
By the Hardy inequality and Proposition \ref{construction:pr:radiation}, $\| r^{-1}u_0\|_{L^2(\mathbb R^d)}\lesssim \|\pa_r u_0\|_{L^2(\mathbb R^d)}\lesssim \| G\|_{L^2(\mathbb R)}$, hence
\begin{equation}\label{gainofreg:bd:compactradiationtech1}
\| \frac{1}{r^2} u_0\|_{L^2(r\geq 2)}+\| \frac{1}{r} \pa_r u_0\|_{L^2(r\geq 2)}\lesssim \| G_+\|_{L^2(\mathbb R)}.
\end{equation}
We split the integral in \eqref{gainofreg:id:compactradiationtech2}:
\begin{align*}
\int_{\rho=0}^\infty \int_{\omega=-1}^1 \frac{G''(r\omega-\rho)}{\sqrt{\rho}} \tilde P_N(\omega)(1-\omega)^{-1/2}d\rho d\omega= I(r)+II(r)+III(r)
\end{align*}
where the regions of integration are $[0,\infty)\times [-1,-7/8)$, $[0,\infty)\times [-7/8,7/8)$ and $[0,\infty)\times (7/8,1]$ respectively. For $I(r)$, since $\omega<-7/8$ and $|r\omega -\rho|<1$ on the support of the integrand, $0<r<8/7<2$ so for $r>2$:
\begin{equation}\label{gainofreg:bd:compactradiationtech2} 
I(r)=0.
\end{equation}
For $II(r)$, integrating by parts twice with respect to $\omega$, and once with respect to $\rho$:
\begin{align*}
II(r) &  =\int_0^\infty \int_{-7/8}^{7/8} \frac 1r \pa_\omega (G'(r\omega-\rho))\tilde P_N(\omega)d\omega \frac{d\rho}{\sqrt{\rho}}\\
& \quad =\int_0^\infty \frac 1r G'(r\frac 78-\rho)\tilde P_N(\frac 78) \frac{d\rho}{\sqrt{\rho}}-\frac 1r \int_0^\infty \int_{-7/8}^{7/8} \frac 1r  G'(r\omega-\rho) \tilde P_N'(\omega)d\omega \frac{d\rho}{\sqrt{\rho}}\\
&\qquad   = -\frac{1}{2r} \int_0^\infty G(r\frac 78-\rho)\tilde P_N(\frac 78) \frac{d\rho}{\rho^{3/2}}-\frac{1}{r^2} \int_0^\infty  G(r\frac 78-\rho) \tilde P_N'(\frac78) \frac{d\rho}{\sqrt{\rho}}\\
&\qquad \qquad +\frac{1}{r^2} \int_0^\infty \int_{-7/8}^{7/8} G(r\omega-\rho) \tilde P_N''(\omega)d\omega \frac{d\rho}{\sqrt{\rho}} \qquad  \qquad = \qquad a+b+c
\end{align*}
where we used that $G(\rho)=0$ for $|\rho|\geq 1$ so that all boundary terms were zero. Note that in $a$ and $b$ we have, for $r>2$, that $\rho\geq 3r/8$ on the support of the integrand. Thus,
$$
|a|+|b|\lesssim \frac{1}{r^{5/2}}\| G\|_{L^1}\lesssim \frac{1}{r^{5/2}} \| G\|_{L^2}
$$
since $\textup{supp}(G)\subset [-1,1]$. For $c$, we split the $\omega$ integral into $[-7/8,7/8]\cap \{|r\omega|>2\}$ and $[-7/8,7/8]\cap \{|r\omega|\leq 2\}$ and denote by $c_1$ and $c_2$ the corresponding integrals. In $c_1$, since $|r\omega-\rho|\leq 1$ we have $\rho \geq 1$ so that $|c_1|\lesssim r^{-2}\| G\|_{L^1}\lesssim r^{-2}\| G\|_{L^2}$. For $c_2$ we have $\rho \leq 1+2=3$. Thus,
$$
|c_2|\lesssim \frac{1}{r^2} \int_0^3 \int_{|r\omega|\leq 2} |G(r\omega-\rho)|d\omega\frac{d\rho}{\sqrt{\rho}}=\frac{1}{r^3}\int_0^3 \int_{|s|\leq 2}|G(s-\rho)|ds\frac{d\rho}{\sqrt{\rho}}\lesssim \frac{1}{r^3} \| G\|_{L^1}\lesssim \frac{1}{r^3}\| G\|_{L^2}.
$$
All in all, for $r>2$:
\begin{equation}\label{gainofreg:bd:compactradiationtech3}
|II(r)|\lesssim \frac{1}{r^2}\| G\|_{L^2}.
\end{equation}
For $III$, for $r>2$, since $|r\omega-\rho|<1$ and $\omega \geq 7/8$ we have $\rho>|r\omega|-1\geq 7/8\cdot r -1/2\cdot r>3r/8$. Thus, $|III(r)| =\frac 34 \int_{3r/8}^\infty \int_{7/8}^1 G(r\omega-\rho)\tilde P(\omega)d\omega \rho^{-5/2}d\rho$ after integrating by parts in $\rho$, so
\begin{equation}\label{gainofreg:bd:compactradiationtech4}
|III(r)|\lesssim \frac{1}{r^{5/2}} \int_{3r/8}^\infty \int_{7/8}^1 |G(r\omega-\rho) \tilde P_N(\omega)|d\omega d\rho\lesssim \frac{1}{r^{5/2}} \| G\|_{L^1}\int_{7/8}^1 |\tilde P_N(\omega) d\omega  \lesssim \frac{1}{r^{5/2}}\| G\|_{L^2}
\end{equation}
since $\tilde P_N\in L^1(|\omega|<1)$. Injecting \eqref{gainofreg:bd:compactradiationtech1}, \eqref{gainofreg:bd:compactradiationtech2}, \eqref{gainofreg:bd:compactradiationtech3} and \eqref{gainofreg:bd:compactradiationtech4} in \eqref{gainofreg:id:compactradiationtech2} shows:
\begin{equation}\label{gainofreg:bd:compactradiationtech5}
\| \pa_{rr}u_0\|_{L^2(r\geq 2)}\lesssim \| G_+\|_{L^2(\mathbb R)}.
\end{equation}
Since $\Delta u_0=\pa_{rr}u_0+\frac{N-1}{r}\pa_r u_0$, using \eqref{gainofreg:bd:compactradiationtech1} and \eqref{gainofreg:bd:compactradiationtech5} one obtains the desired inequality \eqref{gainofreg:bd:compactradiation} for $\Delta u_0$. The estimate for $\pa_r\pa_t u(0)$ follows similarly and the lemma follows.

\end{proof}

\begin{proof}[Proof of Lemma \ref{lem:gainregradiation2}]

We prove the result in the case $R=1$, and the general case follows by rescaling. We consider first the more difficult case of $N$ even. By Corollary \ref{gainofreg:cor:gainouter}, $G_+'\in L^2(|\rho|>2)$ with:
\begin{equation}\label{gainregradiation2:bd:tech1}
\| G_+'\|_{L^2(|\rho|>2)}\lesssim \| G_+'\|_{L^2(\rho>1)}+ \| G_-'\|_{L^2(\rho>1)}+ \| G_+\|_{L^2(\mathbb R)}.
\end{equation}
Let $\chi\in \mathcal C^\infty(\mathbb R)$ with $\chi(\rho)=1$ for $|\rho|\leq 2 $ and $\chi(\rho)=0$ for $\rho \geq 4$. Let $G^{(1)}_+=\chi G_+$ and $G^{(2)}_+=(1-\chi)G_+$. Let $u^{(1)}$ and $u^{(2)}$ be the solutions of \eqref{eq:freewave} with radiation profiles $G^{(1)}_+$ and $G^{(2)}_+$ respectively as $t\to \infty$, given by Proposition \ref{construction:pr:radiation}. We have
\begin{equation}\label{gainregradiation2:id:tech1}
u=u^{(1)}+u^{(2)}.
\end{equation}
Note that $\textup{supp}(G^{(1)}_+)\subset \{|\rho|\leq 4\}$ and $\| G^{(1)}_+\|_{L^2(\mathbb R)}\lesssim \| G_+\|_{L^2(\mathbb R)}$ so that by Lemma \ref{gainofreg:lem:compactradiation}
\begin{equation}\label{gainregradiation2:bd:tech2}
\| (u^{(1)}_1,\Delta u^{(1)}_0)\|_{\mathcal H_8}\lesssim \| G_+\|_{L^2(\mathbb R)}.
\end{equation}
We have $G^{(2)'}_+=-\chi' G_++(1-\chi)G_+'$. Note that $\| \chi' G_+\|_{L^2(\mathbb R)}\lesssim \| G_+\|_{L^2(\mathbb R)}$ and that $1-\chi(\rho)\neq 0$ only when $|\rho|\geq 2$, so that $\| (1-\chi)G_+'\|_{L^2(\mathbb R)}\lesssim \| G_+'\|_{L^2(|\rho|\geq 2)} $. Hence by \eqref{gainregradiation2:bd:tech1}:
\begin{equation}\label{gainregradiation2:bd:tech3}
\| G^{(2)'}_+\|_{L^2}\lesssim\| G_+'\|_{L^2(\rho>1)}+ \| G_-'\|_{L^2(\rho>1)}+ \| G_+\|_{L^2(\mathbb R)}.
\end{equation}
By the last property of Proposition \ref{construction:pr:radiation}, we deduce from \eqref{gainregradiation2:bd:tech3} that $(\pa_t u^{(2)}(0),\pa_{tt}u^{(2)}(0))=(u_1^{(2)},\Delta u_0^{(2)})\in \mathcal H$ with
\begin{equation}\label{gainregradiation2:bd:tech4}
\| (u_1^{(2)},\Delta u_0^{(2)})\|_{\mathcal H}\lesssim \| G_+'\|_{L^2(\rho>1)}+ \| G_-'\|_{L^2(\rho>1)}+ \| G_+\|_{L^2(\mathbb R)}.
\end{equation}
Injecting \eqref{gainregradiation2:bd:tech2} and \eqref{gainregradiation2:bd:tech4} in \eqref{gainregradiation2:id:tech1} shows the desired inequality \eqref{bd:gainregradiation2} for $N$ even.

For $N$ odd, the proof of \eqref{bd:gainregradiation2} is similar but easier, using Remarks \ref{gainofreg:reg:gainouter} and \ref{gainofreg:rem:compactradiation}. This ends the proof of the Lemma.

\end{proof}

\section{Technical estimates in even dimensions} \label{A:approximate}

\begin{lemma}
For any $\alpha,R>0$ and $M>2$:
\begin{align} 
 \label{even:controle:bd:universal1}
& \| \frac{1}{r^{N/2+\alpha}}\|_{L^2_{\bar R}}\leq \frac{C_\alpha}{ R^\alpha},  \\
& \label{even:controle:bd:universal5} \int_{\mathbb R} \frac{dt}{(R+|t|)^{1+\alpha}}\leq \frac{C_\alpha}{R^\alpha}  ,\\
& \label{even:controle:bd:universal6} \int_{\mathbb R} \frac{\indic(R+|t|\leq 2MR)}{(R+|t|)^{1-\alpha}}dt \leq CM^{\alpha}R^{\alpha} ,\\
& \label{even:controle:bd:universal2} \| \frac{1}{r^{N/2+1+\alpha}}\|_{L^1 L^2_{R+|t|}}\leq \frac{C_\alpha}{R^\alpha},\\
& \label{even:controle:bd:universal4} \| \frac{1}{r^{N/2+1}}\indic(R+|t|\leq r \leq 2MR) \|_{L^1 L^2}\leq C \log M.
\end{align}
\end{lemma}

\begin{proof}
These inequalities can be proved by direct computations that we omit.
\end{proof}

\begin{proof}[Proof of Lemma \ref{even:lem:estimatesuap}]

Note first the following properties for $\chi_{MR}$. For any $i,j\geq 0$:
\begin{align}
\label{even:controleE:boundschi}& |\pa_r^i\pa_t^j \chi_{MR}|\lesssim M^{-i-j}R^{-i-j}, \qquad \textup{supp}(\chi_{MR})\subset \{ r+|t|\leq 2MR\},\\
\label{even:controleE:boundschi2} & \textup{supp}(\chi_{MR}^2-\chi_{MR}) \ , \ \textup{supp}(\pa_r^i\pa_t^j \chi_{MR})\subset \{MR\leq r+|t|\leq 2MR\} \quad \mbox{ (if }i+j\geq 1).
\end{align}

\noindent \textbf{Step 1}. \emph{Estimates on $a$, $\tilde \phi \chi_{MR}$, and $u_{ap}$}.

\smallskip

\noindent \underline{Proof of \eqref{even:pointwisea}:} Let $r>R+|t|$. From the identity \eqref{id:aF}, and using that the function $r\mapsto r|\cbf|_r$ is non-increasing, we get $|a_F[\cbf](t,r)|\lesssim |\cbf|_r r^{-\frac{N}{2}+1}\leq R|\cbf|_R r^{-\frac{N}{2}}$. By \eqref{bd:tildeaWkappaN} and \eqref{bd:SobolevW} (recalling $\kappa_N=1$ as $N$ is even) we have $|\tilde a(t,r)|\lesssim |\cbf|_R^{2} r^{-\frac{N}{2}+1}\lesssim |\cbf|_R r^{-\frac{N}{2}+1}$ as $|\cbf|_R\lesssim \epsilon'\ll 1$. Combining, we get \eqref{even:pointwisea}.

\smallskip

\noindent \underline{Proof of \eqref{even:pointwisetildephi}:} 
The bound \eqref{even:pointwisetildephi} is a direct consequence of the identity \eqref{def:tildephi}, \eqref{even:controleE:boundschi}, and of the fact that $\tilde c^2\lesssim |\tilde c|$ and $|\tilde c|^2\ln M\lesssim |\tilde c|$ as $|\tilde c|\lesssim \epsilon'\ll1$ and $\ln M\lesssim \ln |\tilde c|^{-1}$.

\smallskip

\noindent \underline{Proof of \eqref{even:pointwiseuap}:} As $u_{ap}=a+\tilde \phi \chi_{RM}$, \eqref{even:pointwiseuap} follows from \eqref{even:pointwisea} and \eqref{even:pointwisetildephi} and the fact that $|\cbf|_{R}+|\tilde c|\lesssim \epsilon'$.

\smallskip

\noindent  \underline{Proof of  \eqref{even:L2patuap}:} We have $\pa_t u_{ap}=\pa_t a +\pa_t \tilde \phi \chi_{MR}+\tilde \phi \pa_t \chi_{MR}$. Let $\bar R\geq R+|t|$. By \eqref{add_decay_aF} (as $\kappa_N=1$) and \eqref{bd:tildeaWkappaN} we have
\begin{equation}\label{even:controleE:pata}
\|\pa_t a\|_{L^2_{\bar R}}\lesssim \frac{R|\cbf|_R}{\bar R}
\end{equation}
By \eqref{def:tildephi}, \eqref{even:controleE:boundschi}, \eqref{even:controleE:boundschi2} and $ \ln M\lesssim |\ln |\tilde c||$ we have $|\pa_t \tilde \phi \chi_{MR}+\tilde \phi \pa_t \chi_{MR}|\lesssim |\tilde c| r^{-N/2}$ for $r>R+|t|$. Hence $\|\pa_t \tilde \phi \chi_{MR}+\tilde \phi \pa_t \chi_{MR}\|_{L^2_{\bar R}}\lesssim |\tilde c|/\bar R$ by \eqref{even:controle:bd:universal1}. Combining, we get \eqref{even:L2patuap}.

\smallskip

\noindent \underline{Proof of \eqref{even:L2L4uap-}:} Since both $\tilde \phi$ and $\chi$ are even in time, $u_{ap,-}=a_-$. Using this and \eqref{even:pointwisea}:
$$
 \| r^{\frac{N-6}{4}}(u_{ap}(t)-u_{ap}(-t))\|_{L^4_{R+|t|}}^2 \lesssim R^2 |\cbf|_{R}^2  \left( \int_{R+t}^\infty r^{N-6}r^{-2N}r^{N-1}dr\right)^{1/2}\lesssim \frac{R^2 |\cbf|_{R}^2}{(R+|t|)^3}.
$$
Applying \eqref{even:controle:bd:universal5} then shows \eqref{even:L2L4uap-}.\\

\noindent  \textbf{Step 2.} \emph{Estimates on $E$}. We first decompose the nonlinearity into its quadratic and cubic (and higher order) parts: $\varphi=\varphi_q+\varphi_c$ where
\begin{align*}
\varphi_q (u) &=\left\{\begin{array}{l l l} \varphi_2 r^{\frac{N-6}{2}} u^2  & \mbox{for } \eqref{eq:nonlinearityanalytic},\\ 
|u|u &  \mbox{for } \eqref{eq:nonlinearitypower} \mbox{ and }N=6, \end{array} \right.\\
\varphi_c (u) &=\left\{ \begin{array}{l l l} \sum_{k\geq 3}^\infty \varphi_k |x|^{(k-1)(\frac N2-1)-2}u^k  & \mbox{for } \eqref{eq:nonlinearityanalytic},\\ 
0 &  \mbox{for } \eqref{eq:nonlinearitypower}. \end{array} \right.
\end{align*}
We claim that for all $u,v$:
\begin{align}
&\label{bd:varphiq2} |\varphi_q(u+v)-\varphi_q(u)| \lesssim r^{\frac{N-6}{2}}|v|(|u|+|v|),\\
& \label{bd:varphiq}
|\varphi_q(u+v)-\varphi_q(u)-\varphi_q (v)| \lesssim r^{\frac{N-6}{2}}|u||v|,\\
&  \label{bd:varphiq4}|\varphi_q'(u)| \lesssim r^{\frac{N-6}{2}}|u|,\\
& \label{bd:varphiq3}|\varphi_q'(u+v)-\varphi_q'(u)| \lesssim r^{\frac{N-6}{2}}|v|,
\end{align}
and that for $|u|,|v|\ll r^{N/2-1}$:
\begin{align} \label{bd:varphic}
& |\varphi_c(u+v)-\varphi_c(u)|  \lesssim r^{N-4} |v|(u^2+v^2) ,\\
\label{bd:varphic2}
& |\varphi_c(u+v)-\varphi_c(u)-\varphi_c(v)|  \lesssim r^{N-4} |u||v|(|u|+|v|) ,\\
\label{bd:varphic3}
& |\varphi_c'(u+v)-\varphi_c'(u)|  \lesssim r^{N-4} |v|(|u|+|v|) ,\\
\label{bd:varphic4}
& |\varphi_c'(u)|  \lesssim r^{N-4} |u|^2
\end{align}
and omit the proofs that follow from standard arguments. We compute the identities using \eqref{even:id:tidephigeneralised}:
\begin{align}
\label{even:id:defpattidephi} \pa_t \tilde \phi =&\frac{1}{r^{N/2}} \left((\tilde c+\beta \hat c^2 \ln \frac r R)\pa_\sigma p_{\frac N 2-1}+\tilde c^2 \pa_\sigma \tilde p \right)(\frac{t}{r}),\\
\label{even:id:defpartildephi} \pa_r \tilde \phi =&\frac{1}{r^{N/2}} \Biggl( \left(\frac{2-N}{2}(\tilde c+\hat c^2 \ln \frac r R)+\beta \hat c^2\right)p_{\frac N2-1}\\
\nonumber&\qquad \qquad \qquad -(\tilde c+\beta \hat c^2 \ln \frac r R)\sigma \pa_\sigma  p_{\frac N2-1}+\hat c^2 (\frac{2-N}{2}\tilde p-\sigma \pa_\sigma \tilde p) \Biggr)(\frac{t}{r}),\\
\label{even:id:defpatttidephi} \pa_{tt} \tilde \phi =&\frac{1}{r^{N/2 +1}} \left((\tilde c+\beta \hat c^2 \ln \frac r R)\pa_\sigma^2 p_{\frac N 2-1}+\hat c^2 \pa_\sigma^2 \tilde p \right)(\frac{t}{r}),
\end{align}
and using \eqref{rigidity|u|u:id:eqA2} and \eqref{rigidity|u|u:id:eqA}:
\begin{align*}
 \Box (\tilde \phi \chi_{MR})=&\Box \tilde \phi \chi_{MR} +2 \pa_t \tilde \phi \pa_t \chi_{MR}-2\pa_r \tilde \phi \pa_{r}\chi_{MR} +\tilde \phi \left(\pa_{tt}-\pa_{rr}-\frac{N-1}{r} \pa_r\right) \chi_{MR}  \\
 =&\chi_{MR} \ \varphi_q(\tilde c\phi_{m_0+1})+2 \pa_t \tilde \phi \pa_t \chi_{MR}-2\pa_r \tilde \phi \pa_{r}\chi_{MR} +\tilde \phi \left(\pa_{tt}-\pa_{rr}-\frac{N-1}{r} \pa_r\right) \chi_{MR} , 
\end{align*}
and get the following expression for the error defined by \eqref{even:id:defE}:
\begin{equation}\label{even:id:defE2}
  E  =-\Box (a+\tilde \phi \chi_{MR})+\varphi(a+\tilde \phi \chi_{MR}) =E_1+E_2+E_3+E_4+E_5
\end{equation}
where
\begin{align}
\label{even:id:defE3} E_1 & =  \varphi_c(a+\tilde \phi \chi_{MR})-\varphi_c(a), \qquad \qquad E_2=\varphi_q(a+\tilde \phi \chi_{MR})-\varphi_q(a)-\varphi_q(\tilde \phi \chi_{MR}), \\
\label{even:id:defE4}E_3&= \varphi_q(\tilde \phi)(\chi_{MR}^2-\chi_{MR}), \qquad \qquad E_4=\chi_{MR}(\varphi_q(\tilde \phi)-\varphi_q(\tilde c\phi_{m_0+1})),\\
\label{even:id:defE5}E_5&=2\pa_r \tilde \phi \pa_{r}\chi_{MR}+2 \pa_t \tilde \phi \pa_t \chi_{MR}-\tilde \phi \left(\pa_{rr}+\frac{N-1}r \pa_r-\pa_{tt}\right) \chi_{MR}.
\end{align}

\smallskip

\noindent \underline{Proof of \eqref{even:L2E}:} 
We will prove:
\begin{multline}
\label{even:pointwiseE} |E|\lesssim \left(\frac{R |\cbf|_{R}|\tilde c|}{r^{N/2+2}}+\frac{|\tilde c|^3(1+\beta|\ln |\tilde c||)}{r^{N/2+1}}\right)\indic (r+|t|\leq 2MR)\\+\frac{|\tilde c|}{r^{N/2+1}}\indic (MR\leq r+|t|\leq 2MR)
\end{multline}
 
The inequality \eqref{even:L2E} follows, using \eqref{even:controle:bd:universal1}.

Let $r>R+|t|$. We have by \eqref{bd:varphic}, \eqref{even:pointwisea}, \eqref{even:pointwisetildephi} and \eqref{even:bd:cbftildectildeR}:
\begin{align} \nonumber
|E_1| & \lesssim r^{N-4}\frac{|\tilde c|\indic(r+|t|\leq 2MR)}{r^{N/2-1}}\left(\frac{|\cbf|_R^2R^2}{r^{N}}+\frac{|\tilde c|^2}{r^{N-2}} \right)\\
\label{even:bdE1} &\lesssim \left(\frac{\epsilon'|\cbf|_R R}{r^{N/2+2}}+\frac{|\tilde c|^3}{r^{N/2+1}} \right)\indic(r+|t|\leq 2MR)
\end{align}
We have by \eqref{bd:varphiq}, \eqref{even:pointwisea} and \eqref{even:pointwisetildephi} :
\begin{equation}\label{even:bdE2}
|E_2|\lesssim r^{\frac{N-6}{2}}\frac{|\tilde c|\indic(r+|t|\leq 2MR)}{r^{N/2-1}} \frac{|\cbf|_R R}{r^{N/2}} \lesssim \frac{|\tilde c||\cbf|_R R}{r^{N/2+2}}\indic(r+|t|\leq 2MR).
\end{equation}
By \eqref{bd:varphiq}, \eqref{even:controleE:boundschi2} and \eqref{even:pointwisetildephi}:
\begin{equation}\label{even:bdE3}
|E_3|\lesssim \frac{\tilde c^2}{r^{N/2+1}}\indic(RM\leq r+|t|\leq 2RM).
\end{equation}
Using first \eqref{even:controleE:boundschi}, \eqref{bd:varphiq2} and \eqref{def:tildephi}, and then \eqref{even:id:tidephigeneralised}, \eqref{even:id:tidephigeneralised2} and $\ln M\lesssim |\ln |\tilde c||$:
\begin{equation}\label{even:bdE4}
|E_4|\lesssim \indic(r+|t|\leq 2MR)r^{\frac{N-6}{2}}\tilde c^2|\psi|(|\tilde \phi|+|\tilde c||\phi_{m_0+1}|) \lesssim \frac{|\tilde c|^3(1+\beta |\ln |\tilde c||)}{r^{N/2+1}} \indic(r+|t|\leq 2MR) 
\end{equation}
By \eqref{even:controleE:boundschi2}, \eqref{even:id:tidephigeneralised}, \eqref{even:id:defpattidephi}, \eqref{even:id:defpartildephi} and $\ln M\lesssim |\ln |\tilde c||$:
\begin{equation}\label{even:bdE5}
|E_5|\lesssim \frac{|\tilde c|}{r^{N/2+1}} \indic(MR\leq r+|t|\leq 2MR).
\end{equation}
Injecting \eqref{even:bdE1}, \eqref{even:bdE2}, \eqref{even:bdE3}, \eqref{even:bdE4} and \eqref{even:bdE5} in \eqref{even:id:defE2} shows \eqref{even:pointwiseE} .

\smallskip

\noindent \underline{Proof of \eqref{even:L1L2E-}:} Using \eqref{even:id:defE2} and the fact that both $\tilde \phi$, $\phi_{m_0+1}$ and $\chi$ are even in time:
\begin{equation}\label{even:id:E-}
E_-(t)=\frac 12 (E(t)-E(-t))= E_{1,-}+E_{2,-}.
\end{equation}
Again, since $\tilde \phi$ and $\chi$ are even in time:
$$
2E_{1,-}(t)= [\varphi_c(a+\tilde \phi \chi_{MR})-\varphi_c(a)-\varphi_c(\tilde \phi\chi_{MR})](t)-[\varphi_c(a+\tilde \phi \chi_{MR})-\varphi_c(a)-\varphi_c(\tilde \phi\chi_{MR})](-t).
$$
Let $r>R+|t|$. Using first \eqref{bd:varphic2}, \eqref{even:pointwisea}, \eqref{even:pointwisetildephi}, then \eqref{even:bd:cbftildectildeR}:
\begin{equation}\label{even:bdE1-}
|E_{1,-}|\lesssim r^{N-4}\frac{|\cbf|_RR}{r^{N/2}} \frac{|\tilde c|}{r^{N/2-1}}\indic (r+|t|\leq 2MR)\left(\frac{|\cbf|_RR}{r^{N/2}} +\frac{|\tilde c|}{r^{N/2-1}} \right)\lesssim \frac{|\cbf|_R|\tilde c|}{r^{N/2+2}}.
\end{equation}
Injecting \eqref{even:bdE2} and \eqref{even:bdE1-} in \eqref{even:id:E-} shows $|E_-|\lesssim r^{-N/2-2}|\cbf|_R|\tilde c|$. The bound \eqref{even:L1L2E-} then follows from \eqref{even:controle:bd:universal2}.

\smallskip

\noindent \smallskip \underline{Proof of  \eqref{even:L2patE}:} From \eqref{even:id:defE2}, \eqref{even:id:defE3}, \eqref{even:id:defE4} and \eqref{even:id:defE5} one deduces:
\begin{align*}
 \pa_t E_1 & =  (\varphi_c'(a+\tilde \phi \chi_{MR})-\varphi_c'(a))\pa_t a +\varphi_c'(a+\tilde \phi \chi_{MR}) \pa_t(\tilde \phi \chi_MR),\\
\pa_t E_2 & =(\varphi_q'(a+\tilde \phi \chi_{MR})-\varphi_q'(a))\pa_t a+(\varphi_q'(a+\tilde \phi \chi_{MR})-\varphi_q'(\tilde \phi \chi_{MR}))\pa_t (\tilde \phi \chi_{MR}), \\
\pa_t E_3&= \varphi_q'(\tilde \phi)\pa_t\tilde \phi(\chi_{MR}^2-\chi_{MR})+ \varphi_q(\tilde \phi) \pa_t \chi_{MR}(2\chi_{MR}-1),\\
\pa_t E_4&=\pa_t \chi_{MR}(\varphi_q(\tilde \phi)-\varphi_q(\tilde c\phi_{m_0+1}))+\chi_{MR}(\varphi_q'(\tilde \phi)\hat c^2\pa_t\psi-(\varphi_q'(\tilde \phi)-\varphi_q'(\tilde c\phi_{m_0+1}))\tilde c \pa_t\phi_{m_0+1}),\\
\pa_t E_5&=2\pa_r\pa_t \tilde \phi \pa_{r}\chi_{MR}+2\pa_r\tilde \phi \pa_{r}\pa_t\chi_{MR}+2 \pa_{tt} \tilde \phi \pa_t \chi_{MR}-\pa_t \tilde \phi \left(\pa_{rr}+\frac{N-1}r \pa_r-\pa_{tt}\right) \chi_{MR}\\
& \quad -\tilde \phi \left(\pa_{rr}\pa_t+\frac{N-1}r \pa_r\pa_t-\pa_{ttt}\right) \chi_{MR}
\end{align*}
Let $r>R+|t|$. By \eqref{bd:varphic3}, \eqref{even:pointwisea} and \eqref{even:pointwisetildephi}, and by \eqref{bd:varphic4}, \eqref{even:pointwisea}, \eqref{even:pointwisetildephi}, \eqref{even:id:tidephigeneralised}, \eqref{even:id:defpattidephi}, \eqref{even:controleE:boundschi} and \eqref{even:controleE:boundschi2}:
\begin{equation}\label{even:bdpatE1}
|\pa_t E_1|\lesssim \frac{|\tilde c|(|\tilde c|+|\cbf|_R)}{r^2}|\pa_t a| \indic(r+|t|\leq 2MR)+\frac{|\tilde c|(|\cbf|_R^2+\tilde c^2)}{r^{N/2+2}}\indic(r+|t|\leq 2MR).
\end{equation}
By \eqref{bd:varphiq3} and \eqref{even:pointwisetildephi}, and by \eqref{bd:varphiq3}, \eqref{even:pointwisea}, \eqref{even:id:tidephigeneralised}, \eqref{even:id:defpattidephi}, \eqref{even:controleE:boundschi} and \eqref{even:controleE:boundschi2}:
\begin{equation}\label{even:bdpatE2}
|\pa_t E_2|\lesssim \frac{|\tilde c|}{r^2}|\pa_t a| \indic(r+|t|\leq 2MR)+\frac{|\tilde c| |\cbf|_RR}{r^{N/2+3}}\indic(r+|t|\leq 2MR).
\end{equation}
By \eqref{even:id:tidephigeneralised}, \eqref{even:id:defpattidephi} and \eqref{even:controleE:boundschi2}, and by \eqref{even:id:tidephigeneralised}, \eqref{even:controleE:boundschi} and \eqref{even:controleE:boundschi2}:
\begin{equation}\label{even:bdpatE3}
|\pa_t E_3|\lesssim \frac{\tilde c^2}{r^{N/2+2}}\indic (MR\leq r+|t|\leq 2MR).
\end{equation}
By \eqref{even:controleE:boundschi}, \eqref{even:controleE:boundschi2} and \eqref{bd:varphiq2}, and by \eqref{even:controleE:boundschi}, \eqref{bd:varphiq4}, \eqref{even:id:tidephigeneralised} and \eqref{bd:varphiq3}:
\begin{equation}\label{even:bdpatE4}
|\pa_t E_4|\lesssim \frac{|\tilde c|^3(1+\beta |\ln |\tilde c||)}{r^{N/2_2}}\indic (r+|t|\leq 2MR).
\end{equation}
By \eqref{even:id:tidephigeneralised}, \eqref{even:id:defpattidephi}, \eqref{even:id:defpartildephi}, \eqref{even:id:defpatttidephi}, \eqref{even:controleE:boundschi} and \eqref{even:controleE:boundschi2}:
\begin{equation}\label{even:bdpatE5}
|\pa_t E_5|\lesssim \frac{|\tilde c|}{r^{N/2+2}}\indic (MR\leq r+|t|\leq 2MR).
\end{equation}
Injecting \eqref{even:bdpatE1}, \eqref{even:bdpatE2}, \eqref{even:bdpatE3}, \eqref{even:bdpatE4} and \eqref{even:bdpatE5} in \eqref{even:id:defE2}, using \eqref{even:bd:cbftildectildeR} and $r>R$, shows:
\begin{equation}\label{even:bdpatEpointwise}
| \pa_t E|\lesssim \frac{ |\tilde c||\pa_t a|}{r^2} +\frac{|\tilde c|(|\cbf|_{R}+|\tilde c|^2(1+\beta |\ln |\tilde c||))}{r^{N/2+2}}+\frac{|\tilde c|}{r^{N/2+2}}\indic (MR\leq r+|t|\leq 2MR)
\end{equation}
Let now $t\in \mathbb R$ and $\bar R\geq R+|t|$. We compute for the first term using H\"older and \eqref{even:controleE:pata}:
\begin{equation}\label{even:bdpatEaveraged1}
\|  \frac{ |\tilde c||\pa_t a|}{r^2} \|_{L^2_{\bar R}}  \lesssim  |\tilde c|\| \pa_t a \|_{L^2_{\bar R}}  \| \frac{1}{r^2 } \|_{L^\infty_{\bar R}} \lesssim  \frac{  |\tilde c| |\cbf|_R R}{\bar R^3} \lesssim \frac{  |\tilde c| |\cbf|_R }{\bar R^{3/2}R^{1/2}}
\end{equation}
By \eqref{even:controle:bd:universal1} one finds:
\begin{align}
\label{even:bdpatEaveraged2}&\| \frac{|\tilde c|(|\cbf|_{R}+\tilde c^2(1+\beta |\ln |\tilde c||))}{r^{N/2+2}}\|_{L^2_{\bar R}}\lesssim \frac{|\tilde c|(|\cbf|_{R}+\tilde c^2(1+\beta |\ln |\tilde c||))}{\bar R^2}\lesssim  \frac{|\tilde c|(|\cbf|_{R}+\tilde c^2(1+\beta |\ln |\tilde c||))}{\bar R^{3/2}R^{1/2}},\\
\label{even:bdpatEaveraged3}&\| \frac{|\tilde c|}{r^{N/2+2}}\indic (MR\leq r+|t| \leq 2MR)\|_{L^2_{\bar R}} \lesssim \frac{|\tilde c|}{M^2R^2}\indic(\bar R\leq 2MR)\lesssim \frac{|\tilde c|}{\bar R^{3/2}R^{1/2}M^{1/2}}
\end{align}
Injecting \eqref{even:bdpatEaveraged1}, \eqref{even:bdpatEaveraged2} and \eqref{even:bdpatEaveraged3} in \eqref{even:bdpatEpointwise} shows 
\begin{equation}\label{even:bdpatEaveraged4}
\| \pa_t E\|_{L^2_{\bar R}}\lesssim \bar R^{-3/2}R^{-1/2}|\tilde c|(|\cbf|_R+\tilde c^2(1+\beta |\ln|\tilde c||)+M^{-1/2})
\end{equation}
Applying \eqref{even:controle:bd:universal5} then shows \eqref{even:L2patE} and ends the proof of the Lemma.

\end{proof}

\begin{proof}[Proof of Lemma \ref{evenN=4:lem:estimatesuap}]

\underline{Proof of \eqref{evenN=4:bd:uap}, \eqref{evenN=4:L2E} and \eqref{evenN=4:bd:patE}:} The inequalities \eqref{even:pointwiseuap} and \eqref{even:L2E} of Lemma \ref{even:lem:estimatesuap}, and \eqref{even:bdpatEaveraged4} of its proof, are actually valid for all even dimensions $N\geq 4$, since their proofs never relied on wether $N\equiv 4\mod 4$ or $N\equiv 6\mod 4$. Injecting $\beta=0$ (due to \eqref{def:beta}), in \eqref{even:pointwiseuap}, \eqref{even:L2E} and \eqref{even:bdpatEaveraged4} shows \eqref{evenN=4:bd:uap}, \eqref{evenN=4:L2E} and \eqref{evenN=4:bd:patE} respectively.

\smallskip

\noindent \underline{Proof of \eqref{evenN=4:bd:patuapmathcalH} and \eqref{evenN=4:bd:patuap}}: We apply Lemmas \ref{evenN=4:lem:gainreg2} and \ref{evenN=4:lem:gainreg4} to $a$, using \eqref{bd:tildeaWkappaN} and \eqref{even:controle:bd:universal1}, showing that $\| \pa_t a(t)\|_{\dot H^1_{\bar R+|t|}}\lesssim |\cbf|_R \bar R^{-1}$ and $\| \pa_{tt} a(t)\|_{L^2_{\bar R+|t|}}\lesssim |\cbf|_R \bar R^{-1}$ respectively, for any $\bar R\geq R_1+|t|$. By a direct computation using \eqref{even:id:defpattidephi}, \eqref{even:id:defpatttidephi} and $\beta=0$, $$\| (\pa_t (\tilde \phi \chi_{MR})(t),\pa_{tt}(\tilde \phi \chi_{MR})(t))\|_{\mathcal H_{\bar R}} \lesssim |\tilde c|\bar R^{-1}.$$ Combining, we get \eqref{evenN=4:bd:patuapmathcalH}. \eqref{evenN=4:bd:patuap} follows from \eqref{evenN=4:bd:patuapmathcalH} using the radial Sobolev embedding.

\smallskip

\noindent \underline{Proof of \eqref{evenN=4:bd:pattE}}. This follows from straightforward computations, very similar to that proving \eqref{evenN=4:bd:patE}. We omit them.

\end{proof}

\end{appendix}

\bibliographystyle{apalike}%acm
%\bibliography{/home/duyckaerts/ownCloud2/Recherche/toto} %A la maison
%\bibliography{/users/duyckaer/ownCloud/Recherche/toto} %A Paris 13
\bibliography{toto}

\end{document}